\documentclass[12pt,reqno]{amsart}
\usepackage[a4paper, margin=2cm]{geometry}
\usepackage[english]{babel}
\usepackage{amsmath}
\usepackage{amssymb} 
\usepackage{amsfonts}
\usepackage{lscape}
\usepackage{graphicx}
\usepackage{float}
\usepackage{leftindex}
\usepackage{comment}
\usepackage{tikz}
\usepackage{tikz-cd}
\usetikzlibrary{positioning,arrows}
\usetikzlibrary{decorations.pathmorphing}

\tikzstyle{edge}=[->,thick,font=\tiny]
\tikzstyle{face}=[double equal sign distance,-implies,font=\tiny]
\tikzset{>=Latex}
\tikzcdset{arrow style=tikz}
\newlength{\R}
\tikzset{labelsize/.style={font=\scriptsize}}
\tikzset{2cell/.style={-implies,double,double equal sign distance,shorten >=9pt, shorten <=10pt}}

\makeatletter
\newcommand{\labeleditem}[1]{
\item[\text{#1}]\protected@edef\@currentlabel{\text{#1}}\phantomsection
}
\makeatother

\newcounter{nodemaker}
\def\twocell#1#2{%
  \global\edef\mynodeone{twocell\arabic{nodemaker}}%
  \stepcounter{nodemaker}%
  \global\edef\mynodetwo{twocell\arabic{nodemaker}}%
  \stepcounter{nodemaker}%
  \ar[#1,phantom,shift left=3,""{name=\mynodeone}]%
  \ar[#1,phantom,shift right=3,""'{name=\mynodetwo}]%
  \ar[Rightarrow,from=\mynodeone,to=\mynodetwo,"\,#2"]%
}
\usepackage{scalerel}

 \numberwithin{equation}{section}
 \usepackage{color}
 \usepackage[pagebackref]{hyperref}
\hypersetup{%
    colorlinks=true,
    allcolors=violet
}

\usepackage{enumerate,xspace}
\usepackage{graphicx}

\renewcommand{\phi}{\varphi}
\renewcommand{\epsilon}{\varepsilon}

 \newtheorem{proposition}{Proposition}[section]

 \newtheorem*{lemma*}{Lemma}
 \newtheorem{theorem}[proposition]{Theorem}
 \newtheorem{corollary}[proposition]{Corollary}

\newcommand{\thistheoremname}{}
\newtheorem*{genericthm}{\thistheoremname}

 \theoremstyle{definition}
 \newtheorem{definition}[proposition]{Definition}
 \newtheorem{notation}[proposition]{Notation}
 
 \newtheorem*{definition*}{Definition}
 \newtheorem{example}[proposition]{Example}
 \newtheorem*{example*}{Example}
 
 \newtheorem{remark}[proposition]{Remark}
 \newtheorem{construction}[proposition]{Construction}
\numberwithin{equation}{section}

\newcommand{\ca}{\ensuremath{\mathcal A}\xspace}
\newcommand{\cb}{\ensuremath{\mathcal B}\xspace}
\newcommand{\cc}{\ensuremath{\mathcal C}\xspace}
\newcommand{\cd}{\ensuremath{\mathcal D}\xspace}
\newcommand{\ce}{\ensuremath{\mathcal E}\xspace}

\newcommand{\cg}{\ensuremath{\mathcal G}\xspace}

\newcommand{\ck}{\ensuremath{\mathcal K}\xspace}
\newcommand{\cl}{\ensuremath{\mathcal L}\xspace}

\newcommand{\cp}{\ensuremath{\mathcal P}\xspace}
\newcommand{\cq}{\ensuremath{\mathcal Q}\xspace}

\newcommand{\bbd}{\ensuremath{\mathbb D}\xspace}
\newcommand{\bbe}{\ensuremath{\mathbb E}\xspace}

\newcommand{\bbn}{\ensuremath{\mathbb N}\xspace}

\newcommand{\bbs}{\ensuremath{\mathbb S}\xspace}

\newcommand{\x}{\times}

\newcommand{\two}{\ensuremath{\mathbf{2}}\xspace}

\newcommand{\op}{\ensuremath{{}^{\textrm{op}}}\xspace}

\newcommand{\Set}{\ensuremath{\mathbf{Set}}\xspace}
\newcommand{\Cat}{\ensuremath{\mathbf{Cat}}\xspace}
\newcommand{\Gph}{\ensuremath{\mathbf{Gph}}\xspace}
\newcommand{\RGph}{\ensuremath{\mathbf{RGph}}\xspace}
\newcommand{\Mgmd}{\ensuremath{\mathbf{Mgmd}}\xspace}

\newcommand{\SemiCat}{\ensuremath{\mathbf{SemiCat}}\xspace}

\newcommand\ssl{\stretchrel*{$/$}{\textsc{e}}}

\DeclareMathOperator{\Eq}{Eq}
\DeclareMathOperator{\Ins}{Ins}
\DeclareMathOperator{\Fam}{Fam}
\DeclareMathOperator{\Lan}{Lan}
\DeclareMathOperator{\Ran}{Ran}

\DeclareMathOperator{\colim}{colim}
\DeclareMathOperator{\ob}{ob}

\DeclareMathOperator{\im}{im}

\DeclareMathOperator{\Elt}{Elt}

\newcommand{\Dtop}{\Delta_\mathrm{r}}
\newcommand{\Dtopop}{\Delta_\mathrm{r}^\mathrm{op}}
\newcommand{\res}[1]{#1} 
\newcommand{\und}[1]{\widehat{#1}} 
\newcommand{\ts}[1]{\mathbf{#1}} 

\begin{document}

\title{Nerves of generalized multicategories}

\author{Soichiro Fujii}
\address{National Institute of Informatics, Tokyo, Japan}
\email{s.fujii.math@gmail.com}

\author{Stephen Lack}
\address{School of Mathematical and Physical Sciences, Macquarie University NSW 2109, 
Australia}
\email{steve.lack@mq.edu.au}


\begin{abstract}
For any category $\mathcal{E}$ and monad $T$ thereon, we introduce the notion of $T$-simplicial object in $\mathcal{E}$.
Any $T$-category in the sense of Burroni induces a $T$-simplicial object as its nerve.
This nerve construction defines a fully faithful functor from the category $\mathbf{Cat}_T(\mathcal{E})$ of $T$-categories to the category $s_T\mathcal{E}$ of $T$-simplicial objects, 
whose essential image is characterized by a simple condition.
We show that the category $s_T\mathcal{E}$ is enriched over the category of simplicial sets, and that 
this induces the usual 2-category structure on $\mathbf{Cat}_T(\mathcal{E})$.
We also study enriched limits and colimits in $s_T\mathcal{E}$ and $\mathbf{Cat}_T(\mathcal{E})$, and show
that if $\mathcal{E}$ is locally finitely presentable and $T$ is finitary, then  $\mathbf{Cat}_T(\mathcal{E})$ is locally finitely presentable as a 2-category and $s_T\mathcal{E}$ is locally finitely presentable as a simplicially-enriched category.
\end{abstract} 
\date{March 12, 2026}
\maketitle
\setcounter{tocdepth}{1}
\tableofcontents

\section{Introduction}\label{sect:intro}

The notion of $T$-category for a monad $T$ is due to Burroni
\cite{Burroni-Multicategories}, and has since been studied by many
authors under names including multicategory, $T$-multicategory, and
generalized multicategory, often under the further assumption that the
monad $T$ is cartesian.

It includes as special cases many other important notions. If $T$ is the identity monad on a category
$\ce$, then a $T$-category is just an internal category in $\ce$. When
$T$ is the ultrafilter monad on $\Set$, a $T$-category is a
generalized notion of topological space \cite{Barr-RelationalAlgebras}. When $T$ is the free monoid monad
on $\Set$, a $T$-category is a multicategory in the sense of
Lambek \cite{Lambek-multicat}. When $T$ is the free category monad on the category $\Gph$ of
graphs, a $T$-category is what was called a multicat\'egorie in \cite{Burroni-Multicategories} and an fc-multicategory in
\cite{Leinster-book,Leinster-GeneralizedEnrichment}, but is now usually called a virtual double
category \cite{CruttwellShulman}.

As well as being studied in their own right, $T$-categories have been
used in relation to questions of coherence
\cite{Hermida-representable-multicats, Hermida-CoherentUniversal} and as a base for
enrichment \cite{Leinster-book,Leinster-GeneralizedEnrichment}; it is
the relation to enrichment which led to our interest \cite{Fujii-Lack-enrichment-families} in the notion. 
There are also applications to universal algebra; among other things,
a one-object multicategory is the same as a (non-symmetric) operad. The introduction to \cite{CruttwellShulman} contains a good overview to the many ways in which $T$-categories have been used.

While many aspects of $T$-categories have been heavily studied,
relatively little attention has been given to the category (in fact a
2-category) of all $T$-categories for a given $T$ (and $\ce$). This is
the first goal of the paper. The second goal, which is also helpful in
relation to the first, is to introduce and study a nerve construction
for $T$-categories. Each (small) category has an associated simplicial
set called its nerve, and there is a corresponding construction for
internal categories, taking a category in $\ce$ to a simplicial object
in $\ce$, giving a fully faithful embedding of categories in
simplicial objects.  If instead we start with a $T$-category, we show how to define its nerve, which is a new structure that we call a
$T$-simplicial object (in $\ce$). Every such $T$-simplicial object
determines an actual simplicial object in the Kleisli category $\ce_T$
of $T$, but the $T$-simplicial object retains more information about
the $T$-category; furthermore, if we wish to retain the fully faithful
nature of the usual nerve construction, the simplicial maps in $\ce_T$
are not the right notion of morphism. With the natural notion of
morphism of $T$-simplicial objects we do obtain a fully faithful nerve
construction. If $T$ is an identity monad, then a $T$-simplicial object is just a simplicial object, and the nerve construction is the usual one for internal categories.

We study $T$-categories without assuming that $T$ is cartesian, but we
do in some places suppose that $\ce$ is locally finitely presentable (lfp)
and that $T$ is finitary; in this introduction, we shall refer to this
as the ``lfp case''. (All results obtained in this lfp case generalize easily to the case of locally $\alpha$-presentable categories and monads of rank $\alpha$.)

We begin in Section~\ref{sect:T-cat} by reviewing the definition of
$T$-category, and showing in Theorem~\ref{thm:lfp} that in the lfp case the category
$\Cat_T(\ce)$ of $T$-categories is locally finitely presentable.

In Section~\ref{sect:nerve}, we introduce our nerve construction for $T$-categories, giving a fully faithful embedding $\Cat_T(\ce)\to s_T\ce$ of the category
of $T$-categories in the category of $T$-simplicial objects, and we characterize the image in terms of a ``nerve'' or ``Segal''
condition. We also show in Theorem~\ref{thm:set-lfp} that in the lfp case the category $s_T\ce$ is
locally finitely presentable. 

In Section~\ref{sect:comonadicity}, we prove a comonadicity result.
The category $s\Set=[\Delta\op,\Set]$ of simplicial sets is of course a presheaf
category, while more generally the category $s\ce =[\Delta\op,\ce]$ of simplicial
objects in $\ce$ is a functor category. This is generally not the case for
$s_T\ce$ when $T$ is a non-identity monad, but we show in
Theorem~\ref{thm:comondicity} that $s_T\ce$ is comonadic over a functor
category provided that $\ce$ has finite products. In fact 
$s_T\ce$ can also be seen as a category of {\em enriched functors}: see Remark~\ref{rmk:F-cat}.

The remainder of the paper involves enrichment in various ways.

In Section~\ref{sec:simplicial-enrichment}, we show that $s_T\ce$
can be enriched over simplicial sets, generalizing the classical fact
that the category $s\ce$ of simplicial objects in $\ce$ can be so
enriched. We write $\underline{s_T\ce}$ for the resulting
$s\Set$-category.

Then in Section~\ref{sect:2cat} we shift our attention to
$\Cat_T(\ce)$. Since this is a full subcategory of $s_T\ce$, it inherits
the simplicial enrichment of $\underline{s_T\ce}$, but we show that
the $s\Set$-valued homs of this induced structure on $\Cat_T(\ce)$ are
in fact the nerves of categories, and so obtain a 2-category structure
$\underline{\Cat_T(\ce)}$ on $\Cat_T(\ce)$.

Turning to Section~\ref{sect:powers}, we show in Theorem~\ref{thm:power-by-Delta-1} that, whenever $\ce$ has pullbacks, $\underline{s_T\ce}$ has
powers by the representable simplicial set $\Delta[1]$, and so deduce in
Theorem~\ref{thm:T-cat-power-by-2} that the
2-category $\underline{\Cat_T(\ce)}$ has powers by the arrow category $\two$.

Finally in Section~\ref{sec:local-presentability} we study local
presentability of the enriched categories in the lfp case, proving in
Theorem~\ref{thm:s-Presentable} that $\underline{s_T\ce}$ is locally finitely presentable as an $s\Set$-category,
and in Theorem~\ref{thm:Cat-presentable}  that $\underline{\Cat_T(\ce)}$ is locally finitely presentable as a 2-category.

Appendix~\ref{apx:lfp} contains proofs of certain general facts about locally finitely presentable categories used in this paper.

\subsection*{Acknowledgments}
We thank Nathanael Arkor and anonymous referees for their helpful comments on an earlier version of this paper.
The first-named author acknowledges the support of JSPS Overseas Research Fellowships and ASPIRE Grant No.\ JPMJAP2301, JST.

\section{$T$-categories}\label{sect:T-cat}
Throughout this section, let $\ce$ be a category with pullbacks and $T=(T,m\colon T^2\to T,i\colon 1_\ce\to T)$ an arbitrary monad on $\ce$, unless otherwise stated. 
Note in particular that $T$ is \emph{not} assumed to be cartesian (in the sense of \cite[Definition~4.1.1]{Leinster-book}).
In this section, we recall the notion of \emph{$T$-category} \cite{Burroni-Multicategories} (also called \emph{$T$-multicategory} \cite{Leinster-book}) and give several examples.
In preparation for later sections, we adopt the ``simplicial'' notation (as was done in \cite[Section~6]{Bourn-TCats}).
We note that a detailed definition of $T$-categories can also be found in \cite[Section~2]{Tholen-Yeganeh}.

\begin{definition}[{\cite[I.1]{Burroni-Multicategories}}]\label{def:T-graph}
A {\em $T$-graph} $(X_0,X_1,d_0,d_1)$ consists
of objects $X_0,X_1\in\ce$ equipped with morphisms $d_0\colon X_1\to X_0$ and
$d_1\colon X_1\to TX_0$ in $\ce$. 

A \emph{morphism of $T$-graphs} $(X_0,X_1,d_0,d_1)\to (Y_0,Y_1,d_0,d_1)$ is a pair $(f_0\colon X_0\to Y_0,f_1\colon X_1\to Y_1)$ of morphisms in $\ce$ making the diagram 
\[
  \begin{tikzcd}
    TX_0 \ar[d,"Tf_0"' ] & X_1 \ar[l,"d_1"' ] \ar[d,"f_1"]
    \ar[r,"d_0"] & X_0 \ar[d,"f_0"] \\
    TY_0 & Y_1 \ar[l,"d_1"] \ar[r,"d_0"'] & Y_0 
  \end{tikzcd}
\]
commute. We write $\Gph_T(\ce)$ for the category of $T$-graphs and $\ob\colon \Gph_T(\ce)\to \ce$ for the functor mapping $(X_0,X_1,d_0,d_1)$ to $X_0$.
\end{definition}

\begin{definition}[{\cite[I.1]{Burroni-Multicategories}; see also Definition~\ref{def:T-Cat-official}}]\label{def:T-cat}
A \emph{$T$-category} $X$ consists of a $T$-graph $(X_0,X_1,d_0,d_1)$, called the \emph{underlying $T$-graph} of $X$, together with the following data:
\begin{itemize}
    \labeleditem{(CD1)}\label{TCat-D1} 
    a morphism $d_1\colon X_2\to X_1$, where $X_2$ is defined by the pullback
  \begin{equation}\label{eqn:X2-as-pullback}
  \begin{tikzcd}
    X_{2} \ar[r,"d_{2}" ] \ar[d,"d_0"' ] & TX_{1} \ar[d,"Td_0"]\\
    X_{{1}} \ar[r,"d_{1}"' ] & TX_{0}
  \end{tikzcd}
  \end{equation}
  in $\ce$, and 
    \labeleditem{(CD2)}\label{TCat-D2}
    a morphism $s_0\colon X_0\to X_1$, 
\end{itemize}
satisfying the following axioms:
\begin{itemize}
    \labeleditem{(CA1)}\label{TCat-A1}
    the diagram 
  \[
  \begin{tikzcd}
    T^2X_0 \ar[d,"mX_0"' ] & TX_1 \ar[l,"Td_1"' ] & X_2 \ar[l,"d_2"' ]
    \ar[d,"d_1"] \ar[r,"d_0"] & X_1 \ar[d,"d_0"] \\
    TX_0 && X_1 \ar[ll,"d_1"] \ar[r,"d_0"'] & X_0 
  \end{tikzcd}
    \] 
    commutes, 
    \labeleditem{(CA2)}\label{TCat-A2}
    the diagram 
    \[
  \begin{tikzcd}
    &    X_0 \ar[dr,"1"] \ar[d,"s_0"] \ar[dl,"iX_0"' ] \\
    TX_0 & X_1 \ar[l,"d_1"] \ar[r,"d_0"'] & X_0
  \end{tikzcd}
  \]
    commutes,
    \labeleditem{(CA3)}\label{TCat-A3}
    the \emph{associativity law}: the diagram
    \[ \begin{tikzcd}
   X_3 \ar[r,"d_1"] \ar[d,"d_2"' ] & X_{2} \ar[d,"d_{1}"] \\
   X_{2} \ar[r,"d_{1}"'] & X_{1}
  \end{tikzcd} \] 
  commutes, 
  where $X_3$ is defined by the pullback
   \begin{equation}\label{eqn:X3-as-pullback}
       \begin{tikzcd}
    X_{3} \ar[r,"d_{3}" ] \ar[d,"d_0"' ] & TX_{2} \ar[d,"Td_0"]\\
    X_{{2}} \ar[r,"d_{2}"' ] & TX_{1}
  \end{tikzcd}
   \end{equation}
  in $\ce$ and $d_1,d_2\colon X_3\to X_2$ are the unique morphisms (induced by the universality of the pullback \eqref{eqn:X2-as-pullback}) making the diagrams 
\[
  \begin{tikzcd}
    TX_2 \ar[d,"Td_1"' ] & X_3 \ar[l,"d_3"' ] \ar[d,"d_1"]
    \ar[r,"d_0"] & X_2 \ar[d,"d_0"] \\
    TX_1 & X_2 \ar[l,"d_2"] \ar[r,"d_0"'] & X_1 
  \end{tikzcd}\quad
  \begin{tikzcd}
    T^2X_1 \ar[d,"mX_1"' ] & TX_2 \ar[l,"Td_2"' ] & X_3 \ar[l,"d_3"' ]
    \ar[d,"d_2"] \ar[r,"d_0"] & X_2 \ar[d,"d_1"] \\
    TX_1 && X_2 \ar[ll,"d_2"] \ar[r,"d_0"' ] & X_1 
  \end{tikzcd}
\]
  commute, and
    \labeleditem{(CA4)}\label{TCat-A4}
    the \emph{unit laws}: the diagram
    \[\begin{tikzcd}
    X_{1} \ar[dr,"1"']  \ar[r,"s_0"] & X_2 \ar[d,"d_1"] & X_{1} \ar[l,"s_1"' ] \ar[dl,"1"] \\
    & X_{1} &  
  \end{tikzcd}\]
    commutes, 
    where $s_0,s_1\colon X_1\to X_2$ are the unique morphisms making the diagrams 
        \[\begin{tikzcd}
    TX_{0} \ar[d,"Ts_{0}"'] & X_1 \ar[l,"d_1"' ] \ar[d,"s_0"]
    \ar[dr,"1"] \\
    TX_1 & X_{2} \ar[l,"d_{2}"] \ar[r,"d_{0}"'] & X_1 
  \end{tikzcd}\quad \begin{tikzcd}
    & X_1 \ar[r,"d_0"] \ar[d,"s_1"] \ar[dl,"iX_1"'] & X_{0}
    \ar[d,"s_{0}"] \\
    TX_1 & X_{2} \ar[l,"d_{2}"] \ar[r,"d_{0}"'] & X_1 
  \end{tikzcd}\]
    commute.
\end{itemize}

A \emph{$T$-functor} $f\colon X\to Y$ between $T$-categories $X$ and $Y$ is a morphism $(f_0,f_1)$ of underlying $T$-graphs such that the diagram
  \[\begin{tikzcd}
    X_0 \ar[d,"f_0"' ] \ar[r,"s_0"] & X_1 \ar[d,"f_1"]
    & X_2 \ar[l,"d_1"'] \ar[d,"f_2"] \\
    Y_0 \ar[r,"s_0"'] & Y_1 & Y_2 \ar[l,"d_1"] 
  \end{tikzcd}\]
commutes, where $f_2$ is the unique morphism making the diagram 
\[
  \begin{tikzcd}
    TX_1 \ar[d,"Tf_1"' ] & X_2 \ar[l,"d_2"' ] \ar[d,"f_2"]
    \ar[r,"d_0"] & X_1 \ar[d,"f_1"] \\
    TY_1 & Y_2 \ar[l,"d_2"] \ar[r,"d_0"'] & Y_1 
  \end{tikzcd}
\]
commute.
We denote the category of $T$-categories and $T$-functors between them by $\Cat_T(\ce)$.
\end{definition}

\begin{remark}\label{rmk:equality-of-T-cats}
    Strictly speaking, the above definition of $T$-category leaves ambiguous the criterion for equality between two $T$-categories. For instance, it is unclear whether the pullback diagram defining $X_2$ should be regarded as part of the structure. 
    While such questions are not mathematically significant, we will give our ``official'' definition of the category $\Cat_T(\ce)$ later (Definition~\ref{def:T-Cat-official}), as a suitable replete full subcategory of the category of \emph{$T$-simplicial objects} --- according to this latter definition, the pullback diagrams defining $X_2$ and $X_3$, as well as others, are regarded as part of the structure of a $T$-category. 
    In fact, Definition~\ref{def:T-Cat-official} will be given under completely general assumptions, without even requiring the base category $\ce$ to have all pullbacks. 
    Such base categories naturally occur as, e.g., the Kleisli category $\ce_T$ of the monad $T$; see Remark~\ref{rmk:T-cats-and-Kleisli}.
\end{remark}

\begin{remark}
    One can define the \emph{skew bicategory} $\mathbf{Span}_T(\ce)$ of \emph{$T$-spans} and identify $T$-categories with monads in $\mathbf{Span}_T(\ce)$. This is essentially observed in \cite[Proposition~II.3.15]{Burroni-Multicategories}; whereas Burroni's ``pseudo-cat\'egories'' are different from skew bicategories as one of the unitors goes in the opposite direction, that unitor is invertible in $\mathbf{Span}_T(\ce)$, and hence one can view it as a skew bicategory as well. 
    See also \cite[Proposition~5.1]{Zawadowski-lax}.
    The skew bicategory $\mathbf{Span}_T(\ce)$ is a bicategory if and only if the monad $T$ is cartesian \cite[Proposition~II.2.14]{Burroni-Multicategories}.
\end{remark}

Here are some examples of $T$-categories. 

We start with some degenerate cases of $T$-categories (Examples~\ref{ex:objects-as-T-cats} and \ref{ex:T-alg-as-T-cat}). 
For these, it is useful to make the following observation.  
In order to give a $T$-category structure on a $T$-graph $(X_0,X_1,d_0,d_1)$ with $d_0\colon X_1\to X_0$ and $d_1\colon X_1\to TX_0$ jointly monic, it suffices to show that it admits (necessarily unique) \ref{TCat-D1} and \ref{TCat-D2} satisfying \ref{TCat-A1} and \ref{TCat-A2}; the associativity and unit laws (\ref{TCat-A3} and \ref{TCat-A4}) are then automatic \cite[Proposition~I.2.2]{Burroni-Multicategories}.

\begin{example}[{\cite[Proposition~I.3.8]{Burroni-Multicategories}}]\label{ex:objects-as-T-cats}
    The functor $\ob\colon \Cat_T(\ce)\to \ce$ sending each $T$-category $X$ to $X_0$ has a left adjoint, mapping each object $E$ of $\ce$ to the (unique) $T$-category whose underlying $T$-graph is $(E,E,1_E,iE)$.
    We thus obtain a fully faithful functor $\ce\to \Cat_T(\ce)$, whose essential image consists of all $T$-categories $X$ with $d_0\colon X_1\to X_0$ invertible.

    If $\ce$ has binary products, then $\ob\colon \Cat_T(\ce)\to \ce$ has a right adjoint, mapping each $E\in\ce$ to the (unique) $T$-category whose underlying $T$-graph is the product diagram 
    \[
  \begin{tikzcd}
    TE & TE\times E \ar[l,"\pi_1"'] \ar[r,"\pi_2"] & E 
  \end{tikzcd}
    \]
    in $\ce$.
\end{example}

\begin{example}[{\cite[Proposition~I.2.3]{Burroni-Multicategories}}]\label{ex:T-alg-as-T-cat}
    Any pair $(A,a)$ of an object $A\in\ce$ and a morphism $a\colon TA\to A$ in $\ce$ gives rise to a $T$-graph $(A,TA,a,1_{TA})$, which admits a $T$-category structure if and only if $(A,a)$ is an (Eilenberg--Moore) $T$-algebra. Thus we obtain a fully faithful functor $\ce^T\to \Cat_T(\ce)$, whose essential image consists of all $T$-categories $X$ with $d_1\colon X_1\to TX_0$ invertible. 
\end{example}

\begin{example}\label{ex:internal-cats}
    When $T$ is the identity monad $1_\ce$ on $\ce$, the $1_\ce$-categories are the internal categories in $\ce$.  
\end{example}

\begin{example}[{\cite[Section~4]{Hermida-representable-multicats} and \cite[Example~4.2.7]{Leinster-book}}]\label{ex:Lambek-multicats}
    Let $T$ be the free monoid monad on $\Set$. Then $T$ is cartesian, and the $T$-categories are the ordinary multicategories in the sense of Lambek~\cite{Lambek-multicat}. 
\end{example}

\begin{example}[{\cite[III.3]{Burroni-Multicategories} and \cite[Chapter~5]{Leinster-book}}]\label{ex:vdbl}
    Let $\Gph$ be the category of (ordinary) graphs and $T$ be the free category monad on $\Gph$. Then $T$ is cartesian, and 
    the $T$-categories are the \emph{virtual double categories} (this term is due to \cite{CruttwellShulman}).
    Let $X$ be a $T$-category, with the corresponding virtual double category $\mathbb{X}$.
    Then the graph $X_0$ is the graph of \emph{objects} and \emph{horizontal morphisms} of $\mathbb{X}$, whereas the graph $X_1$ is the graph of \emph{vertical morphisms} and \emph{multicells} of $\mathbb{X}$. 
\end{example}

\begin{example}[{\cite[Remark~7.2]{Fujii-Lack-enrichment-families}}]\label{ex:vdbl-hu}
    Let $\RGph$ be the category of (ordinary) reflexive graphs and $S$ the free category monad on $\RGph$, induced from the (monadic) forgetful functor $\Cat\to \RGph$.
    In this case, the $S$-categories correspond to the \emph{unital virtual double categories} in the sense of \cite[Definition~7.1]{Fujii-Lack-enrichment-families}, i.e., 
    virtual double categories equipped with \emph{chosen} horizontal units on its objects. 
    If an $S$-category $Y$ corresponds to a unital virtual double category $\mathbb{Y}$, then the graph $Y_0$ is the graph of objects and horizontal morphisms of $\mathbb{Y}$, whereas the graph $Y_1$ is the graph of vertical morphisms and \emph{multicells having no chosen horizontal units in the domain} in $\mathbb{Y}$.
    Note that the monad $S$ is \emph{not} cartesian. 
    Indeed, the endofunctor $S\colon \RGph\to \RGph$ preserves the terminal object.
    However, in general, if the underlying endofunctor of a cartesian monad $T$ on a category $\ce$ with finite limits preserves the terminal object, then the unit of $T$ becomes a natural isomorphism, making $T$ isomorphic to the identity monad on $\ce$.
\end{example}

\begin{theorem}\label{thm:lfp}
If the category $\ce$ is locally finitely presentable and the functor
$T\colon\ce\to \ce$ is finitary, then the categories $\Gph_T(\ce)$ and
$\Cat_T(\ce)$ are also locally finitely presentable, and the forgetful
functors
$\Cat_T(\ce)\to\Gph_T(\ce)\xrightarrow{\ob}\ce$
are finitary right adjoints.
\end{theorem}

\begin{proof}
Since in a locally finitely presentable category finite limits commute
with filtered colimits, any finite limit of finitary functors between
locally finitely presentable categories is itself finitary. In
particular, the functor $T'\colon\ce\to\ce$ sending $A$ to $TA\x A$ is
also finitary. The category $\Gph_T(\ce)$ of $T$-graphs is just the
comma category $\ce/T'$, and is therefore locally finitely presentable
with the functor $\ob\colon\Gph_T(\ce)\to\ce$ a finitary right
adjoint: see Proposition~\ref{prop:comma} for example.

We now repeatedly use the following facts. Given finitary functors
$F,G\colon\ck\to\cl$ between locally finitely presentable categories,
if $G$ is a right adjoint (or equivalently just continuous) then the
inserter $\Ins(F,G)$ is locally finitely presentable and the
projection $\Ins(F,G)\to\ck$ is a finitary right adjoint. Similarly if
$F,G\colon\ck\to\cl$ are as above, with $\alpha,\beta\colon F\to G$
then the equifier $\Eq(\alpha,\beta)$ is locally finitely presentable
and the projection $\Eq(\alpha,\beta)\to\ck$ is a finitary right
adjoint. We prove them in Propositions~\ref{prop:ins}
and~\ref{prop:eq}, closely following the arguments of \cite{BirdPhD},
where $F$ was also assumed to be a right adjoint.

Let us first construct the category $\Mgmd_T(\ce)$ of \emph{$T$-magmoids} (also discussed in Subsection~\ref{subsec:T-magmoids}), i.e., $T$-graphs equipped with \ref{TCat-D1} satisfying \ref{TCat-A1}.
To this end, let $P_2\colon \Gph_T(\ce)\to \Gph_T(\ce)$ be the functor sending each $T$-graph $(X_0,X_1,d_0,d_1)$ to
the $T$-graph
\[\begin{tikzcd}
    TX_{0} & T^2X_0 \ar[l,"mX_0"'] & TX_1 \ar[l,"Td_1"'] & X_{2} \ar[l,"d_2"'] \ar[r,"d_0"] & X_1 \ar[r,"d_0"] & X_0,
  \end{tikzcd}\]
where $X_2$ is defined by the pullback \eqref{eqn:X2-as-pullback}.
This functor is finitary because $T$ is finitary and pullbacks and filtered colimits commute in $\ce$.
We say that a morphism $(f_0,f_1)\colon (X_0,X_1,d_0,d_1)\to (Y_0,Y_1,d_0,d_1)$ of $T$-graphs is \emph{vertex-trivial} if $X_0=Y_0$ and $f_0$ is the identity on that object, or equivalently, if the functor $\ob\colon \Gph_T(\ce)\to \ce$ sends $(f_0,f_1)$ to an identity morphism. 
Then a $T$-magmoid is just a $T$-graph $X$ equipped with a vertex-trivial morphism $P_2X\to X$.
Thus the category $\Mgmd_T(\ce)$ of $T$-magmoids can be obtained by first forming the inserter
\[
  \begin{tikzcd}
    \Ins(P_2,1) \ar[r,"i"] & \Gph_T(\ce) \ar[r,shift left,"P_2"] \ar[r,shift right,"1"' ]  &
    \Gph_T(\ce) 
  \end{tikzcd}
\] 
of $P_2$ and $1$, 
with structure map say $\alpha\colon P_2.i\to i$, and then the
equifier of ${\ob}.\alpha\colon {\ob}.P_2.i\to {\ob}.i$ and the identity natural transformation on ${\ob}.i$.
Hence $\Mgmd_T(\ce)$ is locally finitely presentable
and the forgetful functor $j\colon \Mgmd_T(\ce)\to\Gph_T(\ce)$ is a finitary right adjoint. 
Note that we have a natural transformation $\beta\colon P_2.j\to j$ obtained by restricting $\alpha$. The component of $\beta$ at a $T$-magmoid $X$ is the vertex-trivial morphism $P_2 jX\to jX$ given by $d_1\colon X_2\to X_1$ as in \ref{TCat-D1}.

The category $\SemiCat_T(\ce)$ of \emph{$T$-semicategories} (also discussed in Subsection~\ref{subsec:T-semicat}), i.e., $T$-magmoids satisfying \ref{TCat-A3}, can then be constructed as follows. 
Let $P_3\colon \Gph_T(\ce)\to\Gph_T(\ce)$ be the (finitary) functor sending each $T$-graph $X=(X_0,X_1,d_0,d_1)$ to the $T$-graph 
\[
  \begin{tikzcd}
    TX_{0} & T^2X_0 \ar[l,"mX_0"'] & TX_2 \ar[l,"Ts"'] & X_{3} \ar[l,"d_3"'] \ar[r,"d_0"] & X_2 \ar[r,"t"] & X_0,
  \end{tikzcd}
\]
where $(X_0,X_2,s,t)=P_2X$ and $X_3$ is defined by the pullback \eqref{eqn:X3-as-pullback}. We have two natural transformations $\beta_1,\beta_2\colon P_3.j\to P_2.j$ whose components at a $T$-magmoid $X$ are the vertex-trivial morphisms $P_3 jX\to P_2 jX$ given by $d_1,d_2\colon X_3\to X_2$ as in \ref{TCat-A3}. 
Therefore the category $\SemiCat_T(\ce)$ of $T$-semicategories is obtained as the equifier of 
$P_3.j\xrightarrow{\beta_1}P_2.j\xrightarrow{\beta}j$ and $P_3.j\xrightarrow{\beta_2}P_2.j\xrightarrow{\beta}j$.
It follows that $\SemiCat_T(\ce)$ is locally finitely presentable and the forgetful functor $\SemiCat_T(\ce)\to \Gph_T(\ce)$ is a finitary right adjoint.

Incorporating units (\ref{TCat-D2}, \ref{TCat-A2}, and \ref{TCat-A4}) is similar; to that end, we use the finitary endofunctor $P_0$ on $\Gph_T(\ce)$ sending each $T$-graph $(X_0,X_1,d_0,d_1)$ to $(X_0,X_0,1_{X_0},iX_0)$. Thus we conclude that the category $\Cat_T(\ce)$ of $T$-categories is locally finitely presentable, and the forgetful functor $\Cat_T(\ce)\to \Gph_T(\ce)$ is a finitary right adjoint.
\end{proof}

We conclude this section with a few remarks.

\begin{remark}
    If $\ce$ is complete and $T$ is an arbitrary monad on $\ce$, then the categories $\Gph_T(\ce)$ and $\Cat_T(\ce)$ are complete.
    This can be proved as in Theorem~\ref{thm:lfp}, or see \cite[Proposition~I.3.5]{Burroni-Multicategories} or \cite[Theorem~4.9]{Tholen-Yeganeh}.
\end{remark}

\begin{remark}\label{rmk:monadicity}
  We have stated Theorem~\ref{thm:lfp} in terms of locally finitely
  presentable categories, which is the most important case. But
  exactly the same argument shows that if $\ce$ is locally
  $\alpha$-presentable for some regular cardinal $\alpha$ and $T$
  preserves $\alpha$-filtered colimits then $\Gph_T(\ce)$ and $\Cat_T(\ce)$ will be locally $\alpha$-presentable, and
  the forgetful functor $\Cat_T(\ce)\to \Gph_T(\ce)$ will be a right adjoint preserving
  $\alpha$-filtered colimits.
  It is also worth noting that the forgetful functor
  $\Cat_T(\ce)\to\Gph_T(\ce)$
  is monadic; 
  using the techniques developed in \cite{Dubuc-Kelly,Kelly-Power} one can give a presentation for a monad on $\Gph_T(\ce)$ whose category of algebras is $\Cat_T(\ce)$.
\end{remark}

\begin{remark}
    Theorem~\ref{thm:lfp} and Remark~\ref{rmk:monadicity} are similar to \cite[Theorem~6.5.2]{Leinster-book}, where it is shown that for any ``suitable'' category $\ce$ and ``suitable'' monad $T$ on $\ce$, the category $\Gph_T(\ce)$ and the free $T$-category monad (called $T^+$ in \cite{Leinster-book}) thereon are again ``suitable''.
    In particular, for any pair $(\ce,T)$ of a locally finitely presentable category $\ce$ and a finitary monad $T$ on $\ce$, the pair $(\Gph_T(\ce),T^+)$ again satisfies the same conditions, and hence we can iterate this construction.
\end{remark}

\section{Nerves}\label{sect:nerve}
Throughout this section, let $\ce$ be a category with pullbacks and $T=(T,m,i)$ an arbitrary monad on $\ce$, unless otherwise stated. 

The aim of this section is to show that each $T$-category gives rise to a simplicial-object-like structure, which we call a \emph{$T$-simplicial object}.
The resulting $T$-simplicial object is called the \emph{nerve} of the $T$-category, although we tend to identify a $T$-category with its nerve (see Definition~\ref{def:T-Cat-official}).
We start with an explicit definition of $T$-simplicial object; see Subsection~\ref{subsec:T-simp-characterizations} for equivalent but more abstract formulations. 

\begin{definition}\label{def:TSimp-explicitly}
Let $\ce$ be a category and $T=(T,m,i)$ a monad on $\ce$.
A \emph{$T$-simplicial object} $X$ consists of 
\begin{itemize}
    \labeleditem{(SD1)}\label{TSimp-D1} an object $X_n$ of $\ce$ for each $n\geq 0$,
    \labeleditem{(SD2)}\label{TSimp-D2} a morphism $d_i\colon X_n\to X_{n-1}$ in $\ce$ for each $n>0$ and each $0\leq i< n$ (notice the strict inequality in the upper bound of $i$),  
    \labeleditem{(SD3)}\label{TSimp-D3} a morphism $d_n\colon X_n\to TX_{n-1}$ in $\ce$ for each $n>0$, and 
    \labeleditem{(SD4)}\label{TSimp-D4} a morphism $s_i\colon X_n\to X_{n+1}$ in $\ce$ for each $n\geq 0$ and each $0\leq i\leq n$,
\end{itemize}
such that the following diagrams (in $\ce$) commute:
\begin{itemize}
    \labeleditem{(SA1)}\label{TSimp-A1} \begin{tikzcd}
    X_n \ar[r,"d_i"] \ar[d,"d_j"'] & X_{n-1} \ar[d,"d_{j-1}"] \\
    X_{n-1} \ar[r,"d_i"'] & X_{n-2} 
  \end{tikzcd}
  (for each $n\geq 2$ and $0\leq i<j< n$)
    \labeleditem{(SA2)}\label{TSimp-A2} \begin{tikzcd}
    X_n \ar[r,"d_i"] \ar[d,"d_n"'] & X_{n-1} \ar[d,"d_{n-1}"] \\
    TX_{n-1} \ar[r,"Td_i"'] & TX_{n-2} 
  \end{tikzcd}
  (for each $n\geq 2$ and $0\leq i\leq n-2$)
    \labeleditem{(SA3)}\label{TSimp-A3} \begin{tikzcd}
    X_n \ar[rr,"d_{n-1}"] \ar[d,"d_n"'] & & X_{n-1} \ar[d,"d_{n-1}"] \\
    TX_{n-1} \ar[r,"Td_{n-1}"'] & T^2X_{n-2} \ar[r,"mX_{n-2}"'] & TX_{n-2} 
  \end{tikzcd}
  (for each $n\geq 2$)
    \labeleditem{(SA4)}\label{TSimp-A4}  \begin{tikzcd}
    X_n \ar[r,"s_i"] \ar[d,"s_j"'] & X_{n+1} \ar[d,"s_{j+1}"] \\
    X_{n+1} \ar[r,"s_i"'] & X_{n+2} 
  \end{tikzcd} 
  (for each $n\geq 0$ and $0\leq i\leq j\leq n$)
    \labeleditem{(SA5)}\label{TSimp-A5}  \begin{tikzcd}
    X_n \ar[r,"d_i"] \ar[d,"s_j"'] & X_{n-1} \ar[d,"s_{j-1}"] \\
    X_{n+1} \ar[r,"d_i"'] & X_{n} 
  \end{tikzcd}
  (for each $n\geq 1$, $0< j\leq n$, and $0\leq i<j$)
    \labeleditem{(SA6)}\label{TSimp-A6}  \begin{tikzcd}
    X_n \ar[dr,"1"] \ar[d,"s_j"'] & \\
    X_{n+1} \ar[r,"d_i"'] & X_{n} 
  \end{tikzcd}
    (for each $n\geq 0$, $0\leq j\leq n$, and $j\leq i\leq j+1$ with $i\neq n+1$)
    \labeleditem{(SA7)}\label{TSimp-A7}  \begin{tikzcd}
    X_n \ar[dr,"iX_n"] \ar[d,"s_n"'] & \\
    X_{n+1} \ar[r,"d_{n+1}"'] & TX_{n} 
  \end{tikzcd}
    (for each $n\geq 0$)
    \labeleditem{(SA8)}\label{TSimp-A8}  \begin{tikzcd}
    X_n \ar[r,"d_{i-1}"] \ar[d,"s_j"'] & X_{n-1} \ar[d,"s_{j}"] \\
    X_{n+1} \ar[r,"d_i"'] & X_{n} 
  \end{tikzcd} 
    (for each $n\geq 2$, $0\leq j \leq n-2$, and $j+1< i< n+1$)
    \labeleditem{(SA9)}\label{TSimp-A9}  \begin{tikzcd}
    X_n \ar[r,"d_{n}"] \ar[d,"s_j"'] & TX_{n-1} \ar[d,"Ts_{j}"] \\
    X_{n+1} \ar[r,"d_{n+1}"'] & TX_{n} 
  \end{tikzcd} 
    (for each $n\geq 1$ and $0\leq j < n$)
\end{itemize}
(These are essentially the usual simplicial identities, suitably adapted to the type of the data \ref{TSimp-D3}.)

Given two $T$-simplicial objects $X$ and $Y$, a \emph{morphism of $T$-simplicial objects} $f\colon X\to Y$ is a family $(f_n\colon X_n\to Y_n)_{n\geq 0}$ of morphisms in $\ce$ making the following diagrams (in $\ce$) commute:
\begin{itemize}
    \item \begin{tikzcd}
    X_n \ar[r,"f_n"] \ar[d,"d_i"'] & Y_{n} \ar[d,"d_{i}"] \\
    X_{n-1} \ar[r,"f_{n-1}"'] & Y_{n-1} 
  \end{tikzcd}
  (for each $n\geq 1$ and $0\leq i< n$)
    \item \begin{tikzcd}
    X_n \ar[r,"f_n"] \ar[d,"d_n"'] & Y_{n} \ar[d,"d_{n}"] \\
    TX_{n-1} \ar[r,"Tf_{n-1}"'] & TY_{n-1} 
  \end{tikzcd}
  (for each $n\geq 1$)
    \item \begin{tikzcd}
    X_n \ar[r,"f_n"] \ar[d,"s_i"'] & Y_{n} \ar[d,"s_{i}"] \\
    X_{n+1} \ar[r,"f_{n+1}"'] & Y_{n+1} 
  \end{tikzcd}
  (for each $n\geq 0$ and $0\leq i\leq n$)
\end{itemize}

We denote the category of $T$-simplicial objects by $s_T\ce$.
\end{definition}

We perform the construction of a $T$-simplicial object from a $T$-category step-by-step, considering various intermediate structures between $T$-graphs and $T$-categories. 
The following diagram summarizes the structures discussed in this section, where the labels on arrows indicate the data or axiom (listed in Definition~\ref{def:T-cat}) additionally imposed.
\begin{equation}\label{eqn:structures-overview}
 \begin{tikzcd}[column sep=35pt,every arrow/.append style={rightsquigarrow}]
    \text{$T$-graph} \ar[d,"\text{\ref{TCat-D1} and \ref{TCat-A1}}"'] \ar[r,"\text{\ref{TCat-D2} and \ref{TCat-A2}}"] & \text{reflexive $T$-graph}  \ar[d,"\text{\ref{TCat-D1} and \ref{TCat-A1}}"]\\
    \text{$T$-magmoid}  \ar[d,"\text{\ref{TCat-A3}}"' ] \ar[r,"\text{\ref{TCat-D2} and \ref{TCat-A2}}"] & \text{reflexive $T$-magmoid}\ar[d,"\text{\ref{TCat-A3}}"] \ar[r,"\text{\ref{TCat-A4}}"]  & \text{unital $T$-magmoid} \ar[d,"\text{\ref{TCat-A3}}"] \\
    \text{$T$-semicategory} \ar[r,"\text{\ref{TCat-D2} and \ref{TCat-A2}}"] & \text{reflexive $T$-semicategory} \ar[r,"\text{\ref{TCat-A4}}"]  & \text{$T$-category}
  \end{tikzcd}
\end{equation}
We note that \cite[Section~6]{Bourn-TCats} essentially contains the construction of a 3-truncated $T$-simplicial object from a $T$-category.
Also, \cite[Section~3]{Haderi-Unlu-simplicial-lists} defines a nerve construction for (ordinary) multicategories (see Example~\ref{ex:Lambek-multicats}), which turns out to be a special case of our nerve construction for $T$-categories. 
In particular, \emph{operadic simplicial lists} of \cite{Haderi-Unlu-simplicial-lists} coincide with $T$-simplicial objects when $T$ is the free monoid monad on $\Set$.

\subsection{$T$-graphs}
$T$-graphs are defined in Definition~\ref{def:T-graph}.
Given a $T$-graph
$(X_0,X_1,d_0,d_1)$, we can inductively construct pullbacks
\begin{equation}\label{eqn:nerve-condition}
  \begin{tikzcd}
    X_{n} \ar[r,"d_{n}" ] \ar[d,"d_0"' ] & TX_{n-1} \ar[d,"Td_0"]
    \\
    X_{{n-1}} \ar[r,"d_{n-1}"' ] & TX_{n-2}
  \end{tikzcd}
\end{equation} 
in $\ce$ for each $n\geq 2$.
The cases where $n=2$ and $3$ appear in Definition~\ref{def:T-cat}.
Pullbacks of the form \eqref{eqn:nerve-condition} play a central role in this paper; for example, they appear in the \emph{nerve condition} characterizing $T$-categories among $T$-simplicial objects.
Thus any $T$-graph $X$ induces \ref{TSimp-D1}, \ref{TSimp-D2} with $i=0$, and \ref{TSimp-D3} in Definition~\ref{def:TSimp-explicitly}, which satisfy \ref{TSimp-A2} with $i=0$ by construction.

Note that if $(f_0,f_1)\colon (X_0,X_1,d_0,d_1)\to (Y_0,Y_1,d_0,d_1)$ is a morphism of $T$-graphs, then we have a morphism $f_n\colon X_n\to Y_n$ for each $n\geq 2$, defined inductively by the commutativity of the following diagram.
\begin{equation}\label{eqn:fn}
  \begin{tikzcd}
    TX_{n-1} \ar[d,"Tf_{n-1}"' ] & X_n \ar[l,"d_n"' ] \ar[d,"f_n"]
    \ar[r,"d_0"] & X_{n-1} \ar[d,"f_{n-1}"] \\
    TY_{n-1} & Y_n \ar[l,"d_n"] \ar[r,"d_0"'] & Y_{n-1} 
  \end{tikzcd}
\end{equation}
Again, the case $n=2$ appears in Definition~\ref{def:T-cat}.

\subsection{$T$-magmoids}\label{subsec:T-magmoids}
A {\em $T$-magmoid} $X$ is a $T$-graph $(X_0,X_1,d_0,d_1)$ equipped with 
a morphism $d_1\colon X_2\to X_1$ (\ref{TCat-D1}) making the diagram
\[
  \begin{tikzcd}
    T^2X_0 \ar[d,"mX_0"' ] & TX_1 \ar[l,"Td_1"' ] & X_2 \ar[l,"d_2"' ]
    \ar[d,"d_1"] \ar[r,"d_0"] & X_1 \ar[d,"d_0"] \\
    TX_0 && X_1 \ar[ll,"d_1"] \ar[r,"d_0"'] & X_0 
  \end{tikzcd}
\] 
commute (\ref{TCat-A1}).
We can then inductively define the morphism $d_i\colon X_{n+1}\to X_n$ in $\ce$ for each $n\geq 2$ and $1\leq i\leq n$ as follows.
\begin{itemize}
    \item For $1\leq i\leq n-1$, define $d_i\colon X_{n+1}\to X_n$ as the unique morphism in $\ce$ making the following diagram commute.
    \[
    \begin{tikzcd}
    TX_n \ar[d,"Td_i"' ] & X_{n+1} \ar[l,"d_{n+1}"' ] \ar[d,"d_i"] \ar[r,"d_0"] & X_n \ar[d,"d_{i-1}"] \\
    TX_{n-1} & X_n \ar[l,"d_n"] \ar[r,"d_0"'] & X_{n-1} 
    \end{tikzcd}
    \]
    \item Define $d_n\colon X_{n+1}\to X_n$ as the unique morphism in $\ce$ making the following diagram commute. 
    \[
    \begin{tikzcd}
    T^2X_{n-1} \ar[d,"mX_{n-1}"' ] & TX_n \ar[l,"Td_n"' ] & X_{n+1} \ar[l,"d_{n+1}"' ]
    \ar[d,"d_n"] \ar[r,"d_0"] & X_n \ar[d,"d_{n-1}"] \\
    TX_{n-1} && X_n \ar[ll,"d_n"] \ar[r,"d_0"' ] & X_{n-1} 
    \end{tikzcd}
    \]
\end{itemize}
Thus we obtain all \emph{face maps} (\ref{TSimp-D2} and \ref{TSimp-D3} in Definition~\ref{def:TSimp-explicitly}), which satisfy some of the \emph{$T$-simplicial identities} involving them (\ref{TSimp-A1} with $i=0$, \ref{TSimp-A2}, and \ref{TSimp-A3}) by construction. 
Some of the remaining identities (namely, \ref{TSimp-A1} with $i>0$ and $j-i\geq 2$) follow.

\begin{proposition}\label{prop:SA1-far}
  Let $X$ be a $T$-magmoid. Then the diagram
  \[ \begin{tikzcd} X_{n+2} \ar[r,"d_p"] \ar[d,"d_q"' ] & X_{n+1}
      \ar[d,"d_{q-1}"] \\
      X_{n+1} \ar[r,"d_{p}"' ] & X_n
    \end{tikzcd} \] 
  commutes for all $n\geq 2$ and $0< p<q<n+2$ with $q-p\geq 2$. 
\end{proposition}
\begin{proof}
    We argue by induction on $n$. 
    Since $n\geq 2$, by the pullback \eqref{eqn:nerve-condition}, it suffices to show that we have $d_0.d_p.d_q=d_0.d_{q-1}.d_p$ and $d_n.d_p.d_q=d_n.d_{q-1}.d_p$.

    We have 
    \begin{align*}
        d_0.d_p.d_q
        &= d_{p-1}.d_0.d_q \tag*{\text{(by \ref{TSimp-A1} with $i=0$)}}\\
        &=d_{p-1}.d_{q-1}.d_0 \tag*{\text{(by \ref{TSimp-A1} with $i=0$)}}\\
        &=d_{q-2}.d_{p-1}.d_0 \tag*{\text{(by inductive hypothesis if $p\geq 2$; by \ref{TSimp-A1} with $i=0$ if $p=1$)}}\\
        &= d_{q-2}.d_0.d_p \tag*{\text{(by \ref{TSimp-A1} with $i=0$)}}\\
        &=d_0.d_{q-1}.d_p. \tag*{\text{(by \ref{TSimp-A1} with $i=0$)}}
    \end{align*}
    If $q\leq n$, then we have 
    \begin{align*}
        d_n.d_p.d_q
        &= Td_{p}.d_{n+1}.d_q \tag*{\text{(by \ref{TSimp-A2})}}\\
        &=Td_{p}.Td_q.d_{n+2} \tag*{\text{(by \ref{TSimp-A2})}}\\
        &=Td_{q-1}.Td_p.d_{n+2} \tag*{\text{(by inductive hypothesis)}}\\
        &= Td_{q-1}.d_{n+1}.d_p \tag*{\text{(by \ref{TSimp-A2})}}\\
        &=d_n.d_{q-1}.d_p, \tag*{\text{(by \ref{TSimp-A2})}}
    \end{align*}
    while if $q=n+1$, then we have 
    \begin{align*}
        d_n.d_p.d_q
        &= Td_{p}.d_{n+1}.d_q \tag*{\text{(by \ref{TSimp-A2})}}\\
        &=Td_{p}.mX_{n}.Td_q.d_{n+2} \tag*{\text{(by \ref{TSimp-A3})}}\\
        &=mX_{n-1}.T^2d_{p}.Td_q.d_{n+2} \tag*{\text{(by naturality of $m$)}}\\
        &=mX_{n-1}.Td_{q-1}.Td_p.d_{n+2} \tag*{\text{(by \ref{TSimp-A2})}}\\
        &= mX_{n-1}.Td_{q-1}.d_{n+1}.d_{p} \tag*{\text{(by \ref{TSimp-A2})}}\\
        &=d_n.d_{q-1}.d_p. \tag*{\text{(by \ref{TSimp-A3})}\qedhere} 
    \end{align*} 
\end{proof}

Given $T$-magmoids $X$ and $Y$, a \emph{morphism} from $X$ to $Y$ is a morphism $(f_0,f_1)$ of the underlying $T$-graphs making the diagram 
\[
  \begin{tikzcd}
    X_2 \ar[d,"f_2"'] \ar[r,"d_1"] & X_{1} \ar[d,"f_{1}"] \\
    Y_2 \ar[r,"d_1"'] & Y_{1} 
  \end{tikzcd}
\]
commute, where $f_2$ is defined by \eqref{eqn:fn}. It then follows that $(f_n\colon X_n\to Y_n)_{n\geq 0}$ commutes with all face maps, i.e., the following diagrams commute.
\begin{itemize}
    \item \begin{tikzcd}
    X_n \ar[d,"f_n"'] \ar[r,"d_i"] & X_{n-1} \ar[d,"f_{n-1}"] \\
    Y_{n} \ar[r,"d_{i}"'] & Y_{n-1} 
  \end{tikzcd}
  (for each $n\geq 1$ and $0\leq i< n$)
    \item \begin{tikzcd}
    X_n \ar[d,"f_n"'] \ar[r,"d_n"] & TX_{n-1} \ar[d,"Tf_{n-1}"] \\
    Y_{n} \ar[r,"d_n"'] & TY_{n-1} 
  \end{tikzcd}
  (for each $n\geq 1$)
\end{itemize}
This can be proved as in Proposition~\ref{prop:SA1-far}.

\subsection{$T$-semicategories}\label{subsec:T-semicat}
To obtain the remaining $T$-simplicial identities involving face maps (\ref{TSimp-A1} with $i>0$ and $j=i+1$), we have to impose the \emph{associativity law} (\ref{TCat-A3}) asserting the commutativity of the following diagram.
\[ \begin{tikzcd}
   X_3 \ar[r,"d_1"] \ar[d,"d_2"' ] & X_{2} \ar[d,"d_{1}"] \\
   X_{2} \ar[r,"d_{1}"'] & X_{1}
  \end{tikzcd} \] 
A $T$-magmoid satisfying \ref{TCat-A3} is called a \emph{$T$-semicategory}.
All remaining identities in \ref{TSimp-A1} follow from \ref{TCat-A3}.

\begin{proposition}
  Let $X$ be a $T$-semicategory. Then the diagram
  \[ \begin{tikzcd}
     X_{n+2} \ar[r,"d_{p}"] \ar[d,"d_{p+1}"'] & X_{n+1} \ar[d,"d_{p}"] \\
     X_{n+1} \ar[r,"d_{p}"'] & X_{n}
    \end{tikzcd} \] 
  commutes for all $n\ge 1$ and $1\le p\leq n$.
\end{proposition}
\begin{proof}
    Similar to Proposition~\ref{prop:SA1-far}.
\end{proof}

Thus every $T$-semicategory induces a \emph{$T$-semisimplicial object}, i.e., the data as in \ref{TSimp-D1}--\ref{TSimp-D3} satisfying the axioms \ref{TSimp-A1}--\ref{TSimp-A3}.
This defines a fully faithful functor from the category of $T$-semicategories to that of $T$-semisimplicial objects, whose essential image consists of all $T$-semisimplicial objects $X$ such that each square of the form \eqref{eqn:nerve-condition} is a pullback in $\ce$. 

\subsection{Reflexive $T$-graphs}\label{subsec:refl-T-graph}
We now move to the second column in \eqref{eqn:structures-overview}. 
A \emph{reflexive $T$-graph} is a $T$-graph $(X_0,X_1,d_0,d_1)$ equipped with a
morphism $s_0\colon X_0\to X_1$ (\ref{TCat-D2}) making the diagram
\[
  \begin{tikzcd}
&    X_0 \ar[dr,"1"] \ar[d,"s_0"] \ar[dl,"iX_0"' ] \\
    TX_0 & X_1 \ar[l,"d_1"] \ar[r,"d_0"'] & X_0
  \end{tikzcd}
  \]
commute (\ref{TCat-A2}).
We can then inductively define morphisms
$s_i\colon X_{n}\to X_{n+1}$ for each $n>0$ and for all $0\le i\le
n$ as follows.
\begin{itemize}
    \item Define $s_0\colon X_n\to X_{n+1}$ as the unique morphism making the following diagram commute. 
    \[\begin{tikzcd}
    TX_{n-1} \ar[d,"Ts_{0}"'] & X_n \ar[l,"d_n"' ] \ar[d,"s_0"]
    \ar[dr,"1"] \\
    TX_n & X_{n+1} \ar[l,"d_{n+1}"] \ar[r,"d_{0}"'] & X_n 
  \end{tikzcd}\]
    \item When $0<i<n$, define $s_i\colon X_n\to X_{n+1}$ as the unique morphism making the following diagram commute. 
  \[\begin{tikzcd}
    TX_{n-1} \ar[d,"Ts_{i}"'] & X_n \ar[l,"d_n"' ] \ar[r,"d_0"]
    \ar[d,"s_i"] & X_{n-1} \ar[d,"s_{i-1}"] \\
    TX_n & X_{n+1} \ar[l,"d_{n+1}"] \ar[r,"d_{0}"'] & X_n 
  \end{tikzcd}\]
    \item Define $s_n\colon X_n\to X_{n+1}$ as the unique morphism making the following diagram commute. 
\[\begin{tikzcd}
    & X_n \ar[r,"d_0"] \ar[d,"s_n"] \ar[dl,"iX_n"'] & X_{n-1}
    \ar[d,"s_{n-1}"] \\
    TX_n & X_{n+1} \ar[l,"d_{n+1}"] \ar[r,"d_{0}"'] & X_n 
  \end{tikzcd}\]
\end{itemize}

These \emph{degeneracy maps} (\ref{TSimp-D4}) satisfy \ref{TSimp-A5} with $i=0$, \ref{TSimp-A6} with $i=j=0$, \ref{TSimp-A7}, and \ref{TSimp-A9} by construction. 
They moreover satisfy \ref{TSimp-A4}.
\begin{proposition}\label{prop:s-s-id}
  Let $X$ be a reflexive $T$-graph. Then for each $n\geq 0$ and $0\leq i\leq j\leq n$, the diagram
\[\begin{tikzcd}
    X_n \ar[r,"s_i"] \ar[d,"s_j"'] & X_{n+1} \ar[d,"s_{j+1}"] \\
    X_{n+1} \ar[r,"s_i"'] & X_{n+2} 
  \end{tikzcd}\]
  commutes.
\end{proposition}

If $X$ and $Y$ are reflexive $T$-graphs, a 
\emph{morphism of reflexive $T$-graphs} $X\to Y$ is a morphism $(f_0,f_1)$ of underlying $T$-graphs making the diagram 
\[
  \begin{tikzcd}
    X_0 \ar[d,"f_0"']
    \ar[r,"s_0"] & X_1 \ar[d,"f_1"] \\
    Y_0 \ar[r,"s_0"'] & Y_1 
  \end{tikzcd}
\]
commute. Then it is easy to see that the family $(f_n\colon X_n\to Y_n)_{n\geq 0}$ defined in  \eqref{eqn:fn} makes
the diagram 
\[
  \begin{tikzcd}
    X_n \ar[d,"f_n"']
    \ar[r,"s_i"] & X_{n+1} \ar[d,"f_{n+1}"] \\
    Y_n \ar[r,"s_i"'] & Y_{n+1} 
  \end{tikzcd}
\]
commute for each $n> 0$ and $0\leq i\leq n$.

\subsection{Reflexive $T$-magmoids}\label{subsec:refl-T-magmoid}
A \emph{reflexive $T$-magmoid} is a $T$-graph equipped with a structure of a $T$-magmoid and a reflexive $T$-graph.
Thus a reflexive $T$-magmoid induces all face and degeneracy maps, i.e., the data \ref{TSimp-D1}--\ref{TSimp-D4} in Definition~\ref{def:TSimp-explicitly}.
In addition to the axioms satisfied by the nerve of a $T$-magmoid or that of a reflexive $T$-graph, it satisfies \ref{TSimp-A5} with $i>0$ and \ref{TSimp-A8}.
\begin{proposition}
    Let $X$ be a reflexive $T$-magmoid.
    If $0<i<j$ then the diagram on the left commutes
  while if $j+1<i<n+2$ then the diagram on the right commutes.
\[ \begin{tikzcd} X_{n+1} \ar[r,"s_j"] \ar[d,"d_i"' ] & X_{n+2}
    \ar[d,"d_i"] \\
    X_n \ar[r,"s_{j-1}"' ] & X_{n+1} \end{tikzcd}
 \qquad
  \begin{tikzcd}
    X_{n+1} \ar[r,"s_j"] \ar[d,"d_{i-1}"' ] & X_{n+2} \ar[d,"d_i"] \\
    X_n \ar[r,"s_j"' ] & X_{n+1} 
  \end{tikzcd}
\] 
\end{proposition}

Among the axioms involving degeneracy maps (\ref{TSimp-A4}--\ref{TSimp-A9}), the missing ones are \ref{TSimp-A6} with $(i,j)\neq (0,0)$.

\subsection{Unital $T$-magmoids}
We skip ``reflexive $T$-semicategory'' in the second column of \eqref{eqn:structures-overview} as we do not have much to say about it. 
To derive \ref{TSimp-A6} with $(i,j)\neq (0,0)$, 
we impose the \emph{unit laws} (\ref{TCat-A4}) asserting the commutativity of the following diagram.
    \[\begin{tikzcd}
    X_{1} \ar[dr,"1"']  \ar[r,"s_0"] & X_2 \ar[d,"d_1"] & X_{1} \ar[l,"s_1"' ] \ar[dl,"1"] \\
    & X_{1} &  
  \end{tikzcd}\]
A reflexive $T$-magmoid satisfying the unit laws is called a \emph{unital $T$-magmoid}.

\begin{proposition}
    Let $X$ be a unital $T$-magmoid. Then for each $n\geq 0$, $0\leq j\leq n$, and $j\leq i\leq j+1$ with $i\neq n+1$, the diagram 
\[
\begin{tikzcd}
    X_n \ar[dr,"1"] \ar[d,"s_j"'] & \\
    X_{n+1} \ar[r,"d_i"'] & X_{n} 
  \end{tikzcd}
\]
    commutes. 
\end{proposition}

\subsection{$T$-categories}
It follows from our discussion so far that any $T$-category $X$ gives rise to a $T$-simplicial object, i.e., \ref{TSimp-D1}--\ref{TSimp-D4} satisfying \ref{TSimp-A1}--\ref{TSimp-A9}.
The following is now immediate:
\begin{theorem}\label{nerve-theorem}
    The nerve construction described so far gives rise to a fully faithful functor $\Cat_T(\ce)\to s_T\ce$, whose essential image consists of all $T$-simplicial objects $X$ such that each square of the form \eqref{eqn:nerve-condition} is a pullback in $\ce$.
\end{theorem}
In fact, we will \emph{identify} $T$-categories with 
their nerves, and 
give the following definition of $T$-category, which does not rely on the assumption that $\ce$ has pullbacks.

\begin{definition}\label{def:T-Cat-official}
    Let $\ce$ be a category and $T$ a monad on $\ce$. (Note that $\ce$ is not assumed to have all pullbacks.) 
    A \emph{$T$-category} is a $T$-simplicial object $X$ such that each square of the form \eqref{eqn:nerve-condition} is a pullback in $\ce$. We define the category $\Cat_T(\ce)$ of $T$-categories as the full subcategory of the category $s_T\ce$ of $T$-simplicial objects consisting of all $T$-categories. 

    When $T$ is the identity monad $1_\ce$ on $\ce$, we write $s_T\ce$ and $\Cat_T(\ce)$ as $s\ce$ and $\Cat(\ce)$,\footnote{An alternative way to define internal categories in a category $\ce$ not necessarily having pullbacks is to do so representably; see e.g.\ \cite[(0.18.1)]{Duskin-simplicial}. 
    (We note that the object $A_0$ in \cite[second paragraph of (0.18.1)]{Duskin-simplicial} can be obtained by splitting the idempotent $s_0d_0$ on $A_1$.)} respectively. Note that $s\ce$ is the functor category $[\Delta\op,\ce]$.
\end{definition}

\begin{example}
    Let $\ce$ be a category, $T$ a monad on $\ce$, and $(A,a)$ a $T$-algebra.
    As Example~\ref{ex:T-alg-as-T-cat} does not rely on the existence of pullbacks in $\ce$, we obtain a $T$-category $X$ which, as a $T$-simplicial object, is given by $X_n=T^nA$, with $d_0=T^{n-1}a\colon T^nA\to T^{n-1}A$, $d_j=T^{n-1-j}mT^{j-1}A\colon T^nA\to T^{n-1}A $ ($1\leq j\leq n-1$), $d_n=1_{T^nA}\colon T^nA\to TT^{n-1}A$, and $s_j= T^{n-j}iT^{j}A\colon T^{n}A\to T^{n+1}A$.
    Note that it contains the \emph{bar resolution} of $(A,a)$ (see e.g.\ \cite[Construction~9.6]{May-iterated-loop-sp}).
\end{example}

\subsection{Equivalent formulations of $T$-simplicial objects}\label{subsec:T-simp-characterizations}
Here we give a few characterizations of $T$-simplicial objects.
From now on, we often denote a $T$-simplicial object by a bold-face letter, such as $\ts X$; 
this will make it easier to distinguish a $T$-simplicial object and certain ``underlying'' data, to be introduced shortly.

\begin{notation}\label{notation:simplicial-1}
    We denote by $\Delta$ the usual category of nonempty finite ordinals and monotone maps, and by $\Dtop$ its wide subcategory consisting of all top-preserving (equivalently, right adjoint) monotone maps.
We also write $n$ for $[n]=\{0<1<\cdots<n\}$ as an object of $\Delta$ (or of $\Dtop$). 
The inclusion $\Dtop\to\Delta$ has a left adjoint, which adjoins
a new top element to an ordinal. This in turn induces a comonad $-+1$
on $\Dtop$. It sends each $\delta_i\colon n\to n+1$ to
$\delta_i\colon n+1\to n+2$, and each $\sigma_i\colon n+1\to n$
to $\sigma_i\colon n+2\to n+1$. The components of the counit are the
$\sigma_n\colon n+1\to n$ and the components of the comultiplication
are the $\delta_{n+1}\colon n+1\to n+2$.
This comonad can be seen as a monad $R$ on $\Dtopop$, and the
Kleisli category of this monad is the inclusion
$\Dtopop\to\Delta\op$.
We thus call this inclusion $F_R\colon \Dtopop\to\Delta\op$.
\end{notation}

Let $\ts X$ be a $T$-simplicial object. 
The data \ref{TSimp-D1}, \ref{TSimp-D2}, and \ref{TSimp-D4} (i.e., those {\em not} including the $d_n\colon X_n\to
TX_{n-1}$) together define a functor $\res{X}\colon\Dtopop\to\ce$.
We can think of $d_n\colon X_n\to TX_{n-1}$ as being a map $X_n\to
X_{n-1}$ in the Kleisli category $\ce_T$ of $T$. The simplicial-like identities
involving these maps $d_n\colon X_n\to TX_{n-1}$ are just the usual
simplicial identities, but now interpreted in the Kleisli
category. Thus the functor $\res{X}\colon\Dtopop\to\ce$ extends to a functor $\und{X}\colon\Delta\op\to\ce_T$.
Hence a $T$-simplicial object $\ts X$ can be identified with a pair
of functors $\res{X},\und{X}$ making the diagram
\[ \begin{tikzcd}
   \Dtopop \ar[r,"F_R"] \ar[d,"\res{X}"'] & \Delta\op \ar[d,"\und{X}"] \\
    \ce \ar[r,"F_T"'] & \ce_T 
  \end{tikzcd} \] 
commute.
A morphism of $T$-simplicial objects $\ts X\to \ts Y$ is just a compatible pair of natural
transformations, $\res{X}\to \res{Y}$ and $\und{X}\to \und{Y}$.
We note that the fact that any $T$-category induces a (truncated) simplicial object in $\ce_T$ is observed in \cite[Section~6]{Bourn-TCats}.
The above discussion proves the following.
\begin{proposition}\label{prop:T-simp-via-pullback}
    The category $s_T\ce$ of $T$-simplicial objects is the pullback
    \[ \begin{tikzcd}
   s_T\ce \ar[r,"\und{(-)}"] \ar[d,"U"'] & {[\Delta^\mathrm{op},\ce_T]} \ar[d,"{[F_R,\ce_T]}"] \\
    {[\Dtopop,\ce]} \ar[r,"{[\Dtopop,F_T]}"'] & {[\Dtopop,\ce_T],}
  \end{tikzcd} \] 
  where $U\ts X=X$ for each $\ts X\in s_T\ce$.
\end{proposition}

By the universal property of the Kleisli object, to give an extension
$\und{X}$ of $F_T\res{X}$ along $F_R$ is equivalent to giving a natural
transformation $\xi$ as in
\begin{equation}\label{eqn:opmorphism}
  \begin{tikzcd}[row sep=small]
    \Dtopop \ar[r,"R"] \ar[dd,"\res{X}"' ] & \Dtopop
    \ar[dd,"\res{X}"] \\
    \twocell{r}{\xi} & {}\\
    \ce \ar[r,"T"' ] & \ce 
  \end{tikzcd}
\end{equation}
which makes $\res{X}$ into an opmorphism of monads \cite{ftm}. Concretely, the component of $\xi$ at
$n\in\Dtopop$ is just $d_{n+1}\colon X_{n+1}\to TX_n$, and the two conditions for $(\res X,\xi)$ to be an opmorphism of monads (i.e., compatibility with the multiplications and units) are \ref{TSimp-A3} and \ref{TSimp-A7}.

\begin{theorem}\label{thm:opmorphisms}
    A $T$-simplicial object $\ts X$ is a $T$-category if and only if $\xi$ as in \eqref{eqn:opmorphism} is a cartesian natural transformation.
\end{theorem}

\begin{proof}
  Given the comments before the theorem, we know that a $T$-simplicial object $\ts X$ is a $T$-category if and only if $\xi$ as in \eqref{eqn:opmorphism} is cartesian with respect to maps of the
  form $\delta_0\colon n\to n+1$. 
  But by the usual pasting and cancellation properties of
  pullbacks, if $\xi$ is cartesian with respect to
  $\delta_i\colon n \to n+1$ it will also be cartesian with respect
  to $\delta_{i+1}\colon n\to n+1$ (use \ref{TSimp-A1} with $j=i+1$) and $\sigma_i\colon n+1\to n$ (use \ref{TSimp-A6} with $j=i$);
  thus by an easy induction $\xi$ will be cartesian with respect to
  all maps in $\Dtopop$.
\end{proof}

\begin{remark}\label{rmk:T-cats-and-Kleisli}
    Suppose that the category $\ce$ has pullbacks and the monad $T$ is cartesian. 
    In this case, results of \cite{Bourn-TCats} imply that there is a close relation between $T$-simplicial objects (resp.\ $T$-categories) and simplicial objects in $\ce_T$ (resp.\ internal categories in $\ce_T$). 
    By \cite[Propositions~2.4 and 3.2]{Bourn-TCats}, the bijective-on-objects functor $F_T\colon \ce\to \ce_T$ is faithful, and hence we can identify $\ce$ with a wide subcategory of $\ce_T$, which is left cancellable by \cite[Proposition~3.6]{Bourn-TCats}.
    Then a $T$-simplicial object can be identified with a simplicial object $\und{X}$ in $\ce_T$ such that each $d_0\colon X_{n+1}\to X_n$ is in $\ce$,
    as can be seen by an argument much like the proof of Theorem~\ref{thm:opmorphisms}.
    Thus we have a pullback 
    \[
    \begin{tikzcd}
      s_T\ce \rar["\und{(-)}"] \dar & {[\Delta\op,\ce_T]} \dar \\
      {[\omega\op,\ce]} \rar["{[\omega\op,F_T]}"'] & {[\omega\op,\ce_T],}
    \end{tikzcd}
  \]
    where the left (resp.\ right) vertical functor sends a $T$-simplicial object $\ts X$ (resp.\ simplicial object $\ts X$ in $\ce_T$) to the sequence $X_0\xleftarrow{d_0}X_1\xleftarrow{d_0}X_2\xleftarrow{d_0}\cdots$ in $\ce$ (resp.\ in $\ce_T)$. 
    Moreover, the pullback-stability of $\ce$ in $\ce_T$ \cite[Proposition~3.12]{Bourn-TCats} implies that we can identify $T$-categories with certain internal categories in $\ce_T$:
    we have a pullback
  \[
    \begin{tikzcd}
      \Cat_T(\ce) \rar \dar & \Cat(\ce_T) \dar \\
      \ce^{\two} \rar["F^{\two}_T"'] & \ce^{\two}_T,
    \end{tikzcd}
  \]
    where the vertical functors send $\ts X$ to the morphism $d_0\colon X_1\to X_0$ \cite[Theorem~7.5]{Bourn-TCats}.
    (Since the Kleisli category $\ce_T$ does not have pullbacks in general, the category $\Cat(\ce_T)$ is defined as a full subcategory of $[\Delta\op,\ce_T]$ as in Definition~\ref{def:T-Cat-official}.)

    Since the right adjoint functor $U_T\colon \ce_T\to \ce$, being a functor of descent type \cite{Kelly-Power}, preserves and reflects all limits, one can also identify a $T$-category with a simplicial object $\widehat X\colon \Delta\op\to \ce_T$ in $\ce_T$ which comes from a $T$-simplicial object and such that $\Delta\op\xrightarrow{\widehat X}\ce_T\xrightarrow{U_T}\ce$ is (the nerve of) an internal category in $\ce$.
    This generalizes the nerve theorem of \cite[Theorem~3.1]{Haderi-Unlu-simplicial-lists} from ordinary multicategories to $T$-categories for any cartesian $T$.
\end{remark}

\subsection{Local presentability}

We conclude this section by showing that if $\ce$ is locally finitely presentable and $T$ is finitary (which we assume throughout this subsection) then $s_T\ce$ is locally finitely presentable. 
We shall prove this by gradually building up various truncated
versions of $s_T\ce$.

For each $n$, write $s^n_T\ce$ for the category of $n$-truncated
$T$-simplicial objects, consisting of objects $X_0,\ldots,X_n$ of
$\ce$, together with all the $T$-simplicial operators between them. Write $s^{n+}_T\ce$ for the category of
$n$-truncated $T$-simplicial objects, further equipped with $X_{n+1}$
and its various face maps, but not the degeneracy maps $X_n\to
X_{n+1}$.

Thus $s^0_T\ce$ is just $\ce$, while $s^{0+}_T\ce$ is the category of
$T$-graphs, $s^1_T\ce$ is the category of reflexive $T$-graphs, and so on.

There are evident forgetful functors $U_n\colon s^{n+}_T\ce\to
s^n_T\ce$ and $V_n\colon s^{n+1}_T\ce\to s^{n+}_T\ce$ for each $n$.
We shall show inductively that all of these categories $s^n_T\ce$
and $s^{n+}_T\ce$ are locally finitely presentable, and that the $U_n$
and $V_n$ are finitary right adjoints.

\begin{proposition}\label{prop:Cn_Dn}
  The composite $U_nV_n\colon s^{n+1}_T\ce\to s^n_T\ce$ has both a
  left adjoint $D_n$ with identity unit and a right adjoint $C_n$ with
  identity counit.
\end{proposition}

  Without the $T$, these would just be given by left Kan and right Kan
  extension and are well-known.

We don't need the full strength of this in what follows. What we do use is the
  functor $L_n\colon s^n_T\ce\to\ce$ sending $X$ to $(C_nX)_{n+1}$,
  and the composite functor $K_n=V_nD_n$, and the fact $U_nK_n=1$. We
  also need $L_n$ to be finitary, which is clear from the construction
  sketched below, given that $T$ is finitary and finite limits commute
  with filtered colimits in $\ce$.

\begin{proof}[Proof of Proposition~\ref{prop:Cn_Dn}]
For example, $C_0X$ has $(C_0X)_0=X_0$ and $(C_0X)_1=X_0\x TX_0$ with
face maps the two projections. And $C_1X$ has $(C_1X)_i=X_i$ for
$i=0,1$, with $(C_1X)_2$ the limit of a diagram
\[
  \begin{tikzcd}
    X_1 \ar[d] \ar[dr] & X_1 \ar[dl] \ar[dr] & TX_1 \ar[dl] \ar[d] \\
    X_0 & TX_0 & TX_0.
  \end{tikzcd}
\]
More generally,  $(C_nX)_i=X_i$ for $i\le n$, while $(C_nX)_{n+1}$ is
constructed as follows. Let $\cp_{n+1}$ be the poset of all proper subobjects of
$[n+1]$ in $\Delta$. There is a functor $\cp\op_{n+1}\to\ce$
sending $\phi\colon[m]\to[n+1]$ to $X_m$ if $\phi m=n+1$, and to
$TX_m$ otherwise, and defined on morphisms using the $T$-simplicial face maps of $X$. Then $(C_nX)_{n+1}$ is the limit of this diagram.
(It suffices to use those subobjects $\phi\colon
[m]\to[n+1]$ with $m=n$ and $m=n-1$.)

On the other hand $D_0X$ has $(D_0X)_i=X_0$ for $i=0,1$. And $D_1X$
has $(D_1X)_i=X_i$ for $i=0,1$ and $(D_1X)_2$ the colimit of the
diagram
\[
  \begin{tikzcd}
    X_0 \ar[r,"s_0"] \ar[d,"s_0"] & X_1 \\
    X_1.
  \end{tikzcd}
\]
More generally, $(D_nX)_i=X_i$ for $i\le n$, while $(D_nX)_{n+1}$ is
constructed using a colimit. Let $\cq_{n+1}$ be the poset of all proper
quotients of $[n+1]$. There is a functor $\cq\op_{n+1}\to\ce$ sending
$\phi\colon [n+1]\to [m]$ to $X_m$ and defined on morphisms using the
$T$-simplicial degeneracy maps. Then $(D_nX)_{n+1}$ is the colimit of this diagram; again, it suffices to consider quotient objects $\phi\colon [n+1]\to[m]$ with $m=n$ or $m=n-1$.
\end{proof}

\begin{theorem}\label{thm:set-lfp}
  If $\ce$ is locally finitely presentable and $T$ is finitary, then
  $s_T\ce$ is locally finitely presentable as an ordinary category.
\end{theorem}

\begin{proof}
  We first prove by induction that $s^n_T\ce$ is locally finitely
  presentable for each $n$; in fact we show that also the
  $s^{n+}_T\ce$ are locally finitely presentable and the functors
  $U_n$ and $V_n$ are finitary. For the base case $s^0_T\ce=\ce$, which is locally
  finitely presentable by assumption.

  Suppose that $s^n_T\ce$ is locally finitely presentable.  Since
  $L_n\colon s^n_T\ce\to\ce$ is constructed using the finitary $T$ and
  finite limits, it is also finitary. The comma category $\ce/L_n$ is
  $s^{n+}_T\ce$ and the projection $\ce/L_n\to s^n_T\ce$ is $U_n\colon
  s^{n+}_T\ce\to s^n_T\ce$. Thus $s^{n+}_T\ce$ is locally finitely
  presentable and $U_n$ is a finitary right adjoint by 
  Proposition~\ref{prop:comma}.

  Next we construct $s^{n+1}_T\ce$ from $s^{n+}_T\ce$. In order to
  provide $X\in s^{n+}_T\ce$ with the  structure needed to make
  it into an object of $s^{n+1}_T\ce$ we should give maps $s_i\colon
  X_n\to X_{n+1}$ for $0\le i\le n$, subject to conditions
  \ref{TSimp-A4}--\ref{TSimp-A9}. The universal property of the colimit
  defining $(D_nU_nX)_{n+1}$ is such that to give $(D_nU_nX)_{n+1}\to
  X_{n+1}$ is precisely to give such $s_i$ satisfying
  \ref{TSimp-A4}. Notice also that $(D_nU_nX)_{n+1}$ can also be
  written as $(K_nU_nX)_{n+1}$. Now the remaining conditions
  \ref{TSimp-A5}--\ref{TSimp-A9} say that this map $(K_nU_nX)_{n+1}\to
  X_{n+1}$ is the degree $n+1$ part of a map $s\colon K_nU_nX\to X$
  which is the identity in degrees $i<n+1$; in other words, such that
  $U_ns$ is the identity on $U_nX$.

  For example, if $X\in s^{1+}_T\ce$, then the map $s$ will have the
form
\[
  \begin{tikzcd}
    X_1+_{X_0}X_1 \ar[r,shift left=2] \ar[r] \ar[r,"\ssl" marking,shift
    right=2] \ar[d,"{(s_0,s_1)}"]& X_1 \ar[r,shift left=2] \ar[r,"\ssl" marking,,shift right=2] \ar[d,"1"]&
    X_0 \ar[l] \ar[d,"1"] \\
    X_2 \ar[r,shift left=2] \ar[r] \ar[r,"\ssl" marking,shift
    right=2] & X_1 \ar[r,shift left=2] \ar[r,"\ssl" marking,,shift right=2] &
    X_0 \ar[l] 
  \end{tikzcd}
\]

  Thus we can construct $s^{n+1}_T\ce$ as follows. 
First construct the inserter
$P_n\colon\mathrm{Ins}(K_nU_n,1)\to s^{n+}_T\ce$ of $K_nU_n$ and
$1$, with structure map say $p_n\colon K_nU_nP_n\to P_n$.
An object of $\mathrm{Ins}(K_nU_n,1)$ consists of $X\in s^{n+}_T\ce$
equipped with a map $s\colon K_nU_nX\to X$, but not yet required to
have $U_ns$ an identity. This $\mathrm{Ins}(K_nU_n,1)$ will be locally finitely
presentable and $P_n$ will be a finitary right adjoint by
Proposition~\ref{prop:ins}.

In order to impose the condition that $U_ns$ be an identity, we replace $\mathrm{Ins}(K_nU_n,1)$ by
the equifier
$Q_n\colon\mathrm{Eq}(U_np_n,1)\to \mathrm{Ins}(K_nU_n,1)$ of the maps
$U_np_n,1\colon U_nP_n\to U_nP_n$. Since $U_nP_n$ is a finitary right
adjoint, $\mathrm{Eq}(U_np_n,1)$ will be locally finitely presentable
and $Q_n$ a finitary right adjoint by Proposition~\ref{prop:eq}, and now $\mathrm{Eq}(U_np_n,1)$ is $s^{n+1}_T\ce$ and $V_n=P_nQ_n$.

In particular, each $s^n_T\ce$ is locally finitely presentable and
each $U_nV_n\colon s^{n+1}_T\ce\to s^n_T\ce$ is a finitary right
adjoint.

Furthermore each $U_nV_n$ is clearly an isofibration. It follows that
the limit of the chain consisting of the $s^n_T\ce$ and the $U_nV_n$
is itself locally finitely presentable. But this limit is precisely $s_T\ce$.
\end{proof}

\section{Comonadicity of $s_T\ce$}\label{sect:comonadicity}

We have assumed in most of the paper that $\ce$ has pullbacks since that allows us to construct objects of composable pairs and so on, but for the notion of $T$-simplicial object no limits in $\ce$ are needed.
In Section~\ref{sec:local-presentability} below we shall need to know that the forgetful functor $U\colon s_T\ce\to[\Dtopop,\ce]$ has a right adjoint. 
In this section we show that in fact it is comonadic
when $\ce$ has finite products.

Before constructing the comonad on $[\Dtopop,\ce]$, we first construct a comonad on the category $[\bbn,\ce]$ of sequences of objects of
$\ce$, where the set $\bbn$ is seen as a discrete category. We then lift this comonad through the forgetful $V\colon [\Dtopop,\ce]\to[\bbn,\ce]$ to obtain the desired comonad.

There is an endofunctor $K_0$ of $[\bbn,\ce]$ defined by $K_0(X)_{n+1}=TX_n$ and
$K_0(X)_0=1$. A $K_0$-coalgebra is a sequence $X$ equipped with maps
$X_{n+1}\to TX_n$ for each $n$. The cofree comonad $K$ over $K_0$ exists when $\ce$ has finite products, and can be defined
recursively with $K(X)_0=X_0$ and $K(X)_{n+1}=X_{n+1}\x TK(X)_n$. Thus
\[ K(X)_n = X_n\x T(X_{n-1}\x T(\ldots\x TX_0)) \]
while the comonad structure is given as follows.
\begin{itemize}
    \item The counit $\varepsilon\colon K(X)\to X$ is given by $\varepsilon_0\colon K(X)_0=X_0\xrightarrow{1} X_0$ and  $\varepsilon_{n+1}\colon K(X)_{n+1}=X_{n+1}\times TK(X)_{n}\xrightarrow{\pi_1} X_{n+1}$.
    \item Letting $d_{n+1}\colon K(X)_{n+1}=X_{n+1}\times TK(X)_n\xrightarrow{\pi_2} TK(X)_n$ for each $n\geq 0$, the comultiplication $\delta\colon K(X)\to KK(X)$ is given by $\delta_0\colon K(X)_0=X_0\xrightarrow{1} X_0=KK(X)_0$ and $\delta_{n+1}\colon K(X)_{n+1}\xrightarrow{\binom{1}{T\delta_n.d_{n+1}}} K(X)_{n+1}\times TKK(X)_n=KK(X)_{n+1}$.
\end{itemize}

\begin{proposition}
  If $\ce$ has finite products, the cofree comonad $K$ lifts to
  $[\Dtopop,\ce]$.
\end{proposition}

\begin{proof}
 First we define the last face map in each degree. We define $d_0\colon K(X)_1\to K(X)_0$ to be
\[ 
\begin{tikzcd}
    X_1\x TX_0 \ar[r,"{\pi_1}"] & X_1 \ar[r,"d_0"] & X_0
\end{tikzcd}
\]
and $d_{n+1}\colon K(X)_{n+2}\to K(X)_{n+1}$ to be
\[ \begin{tikzcd}[column sep=large]
   X_{n+2}\x T(X_{n+1}\x TK(X)_{n}) \ar[r,"d_{n+1}\x T{\pi_2}"] &
    X_{n+1}\x T^2K(X)_n \ar[r,"1\x mK(X)_n"] & X_{n+1}\x
    TK(X)_n.  
\end{tikzcd}\]
Next we define the last degeneracy map in each degree. We define
$s_n\colon K(X)_n\to K(X)_{n+1}$ to be
\[ \begin{tikzcd}
    K(X)_n \ar[r,"\binom{\epsilon}{1}"] & X_n\x K(X)_n \ar[rr,"s_n\x iK(X)_n"] && X_{n+1}\x TK(X)_n
\end{tikzcd}\]
Finally, for a morphism $\phi\colon[m]\to[n]$ in $\Delta_r$, we define
$(\phi+1)^*\colon K(X)_{n+1}\to K(X)_{m+1}$ by
\[\begin{tikzcd}
    X_{n+1}\x TK(X)_n \ar[rr,"{(\phi+1)^*\x T\phi^*}"] && X_{m+1}\x TK(X)_m.
\end{tikzcd} \qedhere\]
\end{proof}

\begin{remark}
  The endofunctor $K_0$ does not lift to $[\Dtopop,\ce]$. For
  example, there seems to be no way to define $d_1\colon TX_1=K_0(X)_2\to
  K_0(X)_1=TX_0$ or $s_0\colon 1=K_0(X)_0\to K_0(X)_1=TX_0$. And the
  lifted comonad does not seem to be cofree.
\end{remark}

\begin{theorem}\label{thm:comondicity}
  If $\ce$ has finite products and $T=(T,m,i)$ is an arbitrary monad on
  $\ce$, then the forgetful functor $U\colon s_T\ce\to[\Dtopop,\ce]$
  is comonadic, with the functor part $K$ of the comonad satisfying
  \[ K(X)_n = X_n\x T(X_{n-1}\x T(\ldots\x TX_0)). \]
\end{theorem}

\begin{proof}
Here we use the fact that the forgetful $V\colon[\Dtopop,\ce]\to [\bbn,\ce]$ is faithful and the comonad $K$ on $[\bbn,\ce]$ lifts to a comonad, which we here call $\widehat{K}$, on $[\Dtopop,\ce]$. In this situation, a $\widehat{K}$-coalgebra structure on $X\in[\Dtopop,\ce]$ amounts to a $K$-coalgebra structure $\zeta\colon VX\to KVX$ which has the form $\zeta=V\widehat\zeta$ for some $\widehat\zeta\colon X\to \widehat{K}X$.

Now a coalgebra for the comonad $K$ on
$[\bbn,\ce]$ is the same as a coalgebra for the endofunctor $K_0$, and
consists of giving a map $d_{n+1}\colon X_{n+1}\to TX_n$ for each
$n$. The corresponding $\zeta$ can be constructed recursively: for
$n=0$ it is the identity and for $n+1$ it is 
\[ \begin{tikzcd}
    X_{n+1} \ar[r,"\binom{1}{d_{n+1}}"] & X_{n+1}\x TX_n
    \ar[r,"1\x T\zeta_n"] & X_{n+1}\x TK(X)_n \ar[r,equals] & K(X)_{n+1}.
\end{tikzcd}\]
We just need to spell out what it means for these maps to be natural.
Naturality with respect to $\delta_0\colon [0]\to [1]$ says that the
diagram
\[ \begin{tikzcd}
    X_1 \ar[r,"\binom{1}{d_1}"] \ar[d,"d_0"] & X_1\x TX_0
    \ar[r,"{\pi_1}"] & X_1 \ar[d,"d_0"] \\
    X_0 \ar[rr,"1"] && X_0
\end{tikzcd}\]
commutes, which is clearly just true. Naturality with respect to
$\delta_{n+1}\colon [n+1]\to [n+2]$ says that the exterior of 
\[ \begin{tikzcd}
    X_{n+2} \ar[r,"\binom{1}{d_{n+2}}"] \ar[dd,"d_{n+1}"] & X_{n+2}\x TX_{n+1}
    \ar[r,"1\x T\zeta"] \ar[d,"1\x Td_{n+1}"] & X_{n+2} \x TK(X)_{n+1} \ar[r,equals] &
    X_{n+2}\x T(X_{n+1}\x TK(X)_n) \ar[d,"1\x T{\pi_2}"] \\
    & X_{n+2}\x T^2X_n \ar[rr,"1\x T^2\zeta"] \ar[d,"d_{n+1}\x mX_n"]
    && X_{n+2}\x T^2K(X)_n \ar[d,"d_{n+1}\x mK(X)_n"] \\
    X_{n+1} \ar[r,"\binom{1}{d_{n+1}}"] & X_{n+1}\x TX_n \ar[rr,"1\x T\zeta"] && X_{n+1}\x TK(X)_n
\end{tikzcd}\]
commutes. The upper right region commutes by the recursive
construction of $\zeta$, and the lower right region also commutes. The
exterior will therefore commute if and only if the diagram
\[ \begin{tikzcd}
    X_{n+2} \ar[r,"d_{n+2}"] \ar[d,"d_{n+1}"] & TX_{n+1}
    \ar[r,"Td_{n+1}"] & T^2X_n \ar[d,"mX_{n}"] \\
    X_{n+1} \ar[rr,"d_{n+1}"] && TX_n 
\end{tikzcd}\] 
does so; for the ``only if'' direction we use the fact that $\zeta$ is a split monomorphism.

Naturality with respect to $\sigma_n\colon[n+1]\to[n]$ says that the
exterior of the diagram on the left in 
\[ \begin{tikzcd}
    X_n \ar[rr,"\zeta_n"] \ar[drr,"\binom{1}{\zeta_n}"]
    \ar[ddr,"\binom{s_n}{iX_n}"] \ar[dd,"s_n"] &&
    K(X)_n \ar[d,"\binom{\epsilon}{1}"] \\
    && X_n\x K(X)_n \ar[d,"s_n\x iK(X)_n"] \\
    X_{n+1} \ar[r,"\binom{1}{d_{n+1}}"] & X_{n+1}\x TX_n \ar[r,"1\x T\zeta_n"] & X_{n+1}\x TK(X)_n
\end{tikzcd} \quad 
\begin{tikzcd}
    X_n \ar[d,"s_n"] \ar[dr,"iX_n"] \\
  X_{n+1} \ar[r,"d_{n+1}"] & TX_n
\end{tikzcd}
\]
commutes, but the upper and middle regions always commute, so this is
equivalent to the region on the left commuting, and so to the diagram
on the right commuting. 

Finally naturality with respect to $\phi+1\colon [m+1]\to [n+1]$ says
that the exterior of the diagram on the left in 
\[ \begin{tikzcd}
    X_{n+1} \ar[r,"\binom{1}{d_{n+1}}"] \ar[d,"(\phi+1)^*"] &
    X_{n+1}\x TX_n \ar[r,"1\x T\zeta"] \ar[d,"(\phi+1)^*\x
      T\phi^*"] &
    X_{n+1}\x TK(X)_n \ar[d,"(\phi+1)^*\x T\phi^*"] \\
    X_{m+1} \ar[r,"\binom{1}{d_{m+1}}"] & X_{m+1}\x TX_m
    \ar[r,"1\x T\zeta"] & X_{m+1}\x TK(X)_m
\end{tikzcd}\quad
\begin{tikzcd}
    X_{n+1} \ar[r,"d_{n+1}"] \ar[d,"(\phi+1)^*"] & TX_n
  \ar[d,"T\phi^*"] \\
  X_{m+1} \ar[r,"d_{m+1}"] & TX_m
\end{tikzcd}\]
commutes. Commutativity of the right region will hold inductively, so
it suffices to have commutativity of the left region, which in turn is
equivalent to commutativity of the diagram on the right.
\end{proof}

\section{Simplicial enrichment}\label{sec:simplicial-enrichment}
Throughout this section, let $\ce$ be a locally small category and $T=(T,m,i)$ a monad on $\ce$. We define an enrichment of the category $s_T\ce$ of $T$-simplicial objects over the category $s\Set=[\Delta\op,\Set]$ of simplicial sets (with respect to the cartesian monoidal structure on $s\Set$).
The remainder of this paper deals with aspects of this $s\Set$-enrichment of $s_T\ce$.

\begin{construction}\label{const:enrichment-by-presheaf-cat}
In general, if $\cb$ is a small category and $\cc$ is a locally small category, then the (locally small) category $[\cb,\cc]$ of all functors $\cb\to\cc$ admits an enrichment over the cartesian monoidal category $[\cb,\Set]$.
More precisely, we have a $[\cb,\Set]$-category 
$\underline{[\cb,\cc]}$ whose objects are the functors $\cb\to \cc$ and, for $X,Y\colon \cb\to \cc$, the hom-object $\underline{[\cb,\cc]}(X,Y)\in[\cb,\Set]$ is defined by mapping each $b\in\cb$ to the set $\underline{[\cb,\cc]}(X,Y)_b$ of all natural transformations from $b/\cb\xrightarrow{\text{forgetful}}\cb\xrightarrow{X}\cc$ to $b/\cb\xrightarrow{\text{forgetful}}\cb\xrightarrow{Y}\cc$;
the action of $\underline{[\cb,\cc]}(X,Y)$ on a morphism $f\colon b\to b'$ of $\cb$ is induced by the functor $f/\cb\colon b'/\cb\to b/\cb$ given by precomposing $f$.
More generally, for any $B\colon\cb\to \Set$, to give a natural transformation $B\to \underline{[\cb,\cc]}(X,Y)$ is equivalent to giving a natural transformation from $\Elt B\xrightarrow{\text{forgetful}}\cb\xrightarrow{X}\cc$ to $\Elt B\xrightarrow{\text{forgetful}}\cb\xrightarrow{Y}\cc$, where $\Elt B$ is the category of elements of $B$. 
With the evident structure, $\underline{[\cb,\cc]}$ becomes a $[\cb,\Set]$-category. 
We remark that the set $\underline{[\cb,\cc]}(X,Y)_b$ admits an end formula 
\[\underline{[\cb,\cc]}(X,Y)_b\cong \int_{b'\in \cb}[\cb(b,b'),\cc(X_{b'},Y_{b'})],\]
and more generally, for any $B\colon\cb\to\Set$, we have 
\[
[\cb,\Set]\bigl(B,\underline{[\cb,\cc]}(X,Y)\bigr)\cong \int_{b\in\cb}[B_b,\cc(X_b,Y_b)].
\]
The underlying category $\underline{[\cb,\cc]}_0$ of $\underline{[\cb,\cc]}$ is canonically isomorphic to the ordinary functor category $[\cb,\cc]$.
\end{construction}

Copowers and powers in these enriched categories are given as follows.

\begin{proposition}\label{prop:copowers-and-powers-in-BC}
    Suppose that $\cb$ is a small category and $\cc$ a locally small category.
    Let $B\colon \cb\to \Set$ and $X\colon \cb\to\cc$. We denote by $\Elt B$ the category of elements of $B$, with $P\colon \Elt B\to \cb$ the forgetful functor. 
    \begin{itemize}
        \item[(1)] If the pointwise left Kan extension 
        \begin{equation}\label{eqn:Lan-for-copower}
\begin{tikzpicture}[baseline=-\the\dimexpr\fontdimen22\textfont2\relax ]
      \node(00) at (0,0.75) {$\Elt B$};
      \node(01) at (1.5,0.75) {$\cb$};
      \node(02) at (3,0.75) {$\cc$};
      \node(10) at (0,-0.75) {$\cb$};
      \draw [->] (00) to node[auto,labelsize] {$P$} (01);
      \draw [->] (01) to node[auto,labelsize] {$X$} (02);
      \draw [->] (00) to node[swap,auto,labelsize] {$P$} (10);
      \draw [->] (10) to node[auto,swap,labelsize] {$\Lan_PXP$} (02);
      \draw [2cell] (1,0.75) to node[auto,swap,labelsize] {$\eta$} (1,-0.4);
\end{tikzpicture}
\end{equation}
        exists, then $\Lan_PXP$ is the copower $B\cdot X$ of $X$ by $B$ in $\underline{[\cb,\cc]}$, and $\eta$ corresponds to the universal element of $[\cb,\Set]\bigl(B,\underline{[\cb,\cc]}(X,B\cdot X)\bigr)$.
        For each $b\in\cb$, the object $(\Lan_PXP)_b\in\cc$ is given by the copower $B_b\cdot X_b$ of $X_b$ by $B_b$ in $\cc$, so $\underline{[\cb,\cc]}$ has copowers (as a $[\cb,\Set]$-category) whenever $\cc$ has copowers (as a $\Set$-category).
        \item[(2)] If the pointwise right Kan extension 
        \begin{equation}\label{eqn:Ran-for-power}
\begin{tikzpicture}[baseline=-\the\dimexpr\fontdimen22\textfont2\relax ]
      \node(00) at (0,0.75) {$\Elt B$};
      \node(01) at (1.5,0.75) {$\cb$};
      \node(02) at (3,0.75) {$\cc$};
      \node(10) at (0,-0.75) {$\cb$};
      \draw [->] (00) to node[auto,labelsize] {$P$} (01);
      \draw [->] (01) to node[auto,labelsize] {$X$} (02);
      \draw [->] (00) to node[swap,auto,labelsize] {$P$} (10);
      \draw [->] (10) to node[auto,swap,labelsize] {$\Ran_PXP$} (02);
      \draw [2cell] (1,-0.4) to node[auto,labelsize] {$\epsilon$} (1,0.75);
\end{tikzpicture}
\end{equation}
        exists, then $\Ran_PXP$ is the power $B\pitchfork X$ of $X$ by $B$ in $\underline{[\cb,\cc]}$, and $\epsilon$ corresponds to the universal element of $[\cb,\Set]\bigl(B,\underline{[\cb,\cc]}(B\pitchfork X, X)\bigr)$.
        For each $b\in \cb$, the object $(\Ran_PXP)_b\in \cc$ is given by the weighted limit $\{ B\times \cb(b,-),X\}$ of $X\colon \cb\to \cc$ with weight $B\times \cb(b,-)\colon \cb\to \Set$.
        Moreover, for any $Y\colon \cb\to \cc$, if the morphisms $u\colon B\to \underline{[\cb,\cc]}(Y,X)$ in $[\cb,\Set]$ and $\hat u\colon Y\to B\pitchfork X$ in $[\cb,\cc]$ correspond to each other, then for each $b,b'\in \cb$, $\beta\in B_{b'}$, and $f\in \cb(b,b')$, the diagram 
   \begin{equation}\label{eqn:power-morphism-correspondence}
    \begin{tikzpicture}[baseline=-\the\dimexpr\fontdimen22\textfont2\relax ]
      \node(1) at (0,0.75) {$Y_b$};
      \node(2) at (4,0.75) {$(B\pitchfork X)_b$};
      \node(4) at (0,-0.75) {$Y_{b'}$};
      \node(5) at (4,-0.75) {$X_{b'}$};

      \draw [->] (1) to node[auto, labelsize] {$\hat u_b$}  (2);  
      \draw [->] (2) to node[auto, labelsize] {$\pi_{b',\beta,f}$}  (5);
      \draw [->] (1) to node[auto, swap, labelsize] {$Y_{f}$}  (4);
      \draw [->] (4) to node[auto, swap, labelsize] {$u_{b',\beta}$}  (5);
\end{tikzpicture}
    \end{equation}
    commutes, where $\pi_{b',\beta,f}$ is the $(b',\beta,f)$-th projection associated with the weighted limit $\{ B\times \cb(b,-),X\}=(B\pitchfork X)_b$.
    \end{itemize}
\end{proposition}
\begin{proof}
    {[(1)]}
    We first show that for each $b\in\cb$, $(\Lan_PXP)_b$ is given by the copower $B_b\cdot X_b$. 
    Since we are dealing with a \emph{pointwise} left Kan extension, 
    $(\Lan_PXP)_b$ is given by the colimit of $P/b\xrightarrow{\text{projection}} \Elt B\xrightarrow{P}\cb\xrightarrow{X}\cc$. 
    Now, the inclusion $B_b\to P/b$ mapping each $\beta\in B_b$ to $\bigl((b,\beta), b\xrightarrow{1}b\bigr)$ is final, where we regard $B_b$ as a discrete category. 
    Thus we have $(\Lan_PXP)_b=B_b\cdot X_b$. 

    To show that $\Lan_PXP$ is the copower $B\cdot X$, 
    we verify that for each $A\colon \cb\to \Set$, with the corresponding discrete opfibration $Q\colon \Elt A\to \cb$, the left Kan extension \eqref{eqn:Lan-for-copower} is preserved by pasting the pullback diagram 
        \begin{equation}\label{eqn:pb-for-Elt-A-times-B}
\begin{tikzpicture}[baseline=-\the\dimexpr\fontdimen22\textfont2\relax ]
      \node(00) at (0,0.75) {$\Elt (A\times B)$};
      \node(01) at (3,0.75) {$\Elt B$};
      \node(10) at (0,-0.75) {$\Elt A$};
      \node(11) at (3,-0.75) {$\cb$};
      \draw [->] (00) to node[auto,labelsize] {$Q'$} (01);
      \draw [->] (01) to node[auto,labelsize] {$P$} (11);
      \draw [->] (00) to node[swap,auto,labelsize] {$P'$} (10);
      \draw [->] (10) to node[auto,swap,labelsize] {$Q$} (11);
\end{tikzpicture}
\end{equation}
    on the left; that is, the natural transformation $\eta Q'$ exhibits $(\Lan_PXP)Q$ as the (pointwise) left Kan extension of $XPQ'$ along $P'$.
    This follows from the previous discussion, because $P'$ is the discrete opfibration corresponding to the functor $\Elt A\xrightarrow{Q}\cb\xrightarrow{B}\Set$.

    Now suppose that we are given $A\colon \cb\to \Set$, $Y\colon \cb\to \cc$, and a natural transformation $\phi\colon A\times B\to \underline{[\cb,\cc]}(X,Y)$. We have to show that there exists a unique natural transformation $\hat \phi\colon A\to \underline{[\cb,\cc]}(\Lan_PXP,Y)$ making the following triangle commute.
      \begin{equation*}
    \begin{tikzcd}
      A\times B \ar[r,"\widehat{\phi}\times \eta"] \ar[dr,"\phi"'] & \underline{[\cb,\cc]}(\Lan_PXP,Y)\times \underline{[\cb,\cc]}(X,\Lan_PXP) \ar[d,"\text{composition}"]  \\
      & \underline{[\cb,\cc]}(X,Y)
    \end{tikzcd}
  \end{equation*}
    The natural transformation $\phi$ corresponds to a natural transformation
    \[
    \begin{tikzpicture}[baseline=-\the\dimexpr\fontdimen22\textfont2\relax ]
        \node(09) at (-3,0.75) {$\Elt (A\times B)$};
      \node(00) at (0,0.75) {$\Elt B$};
      \node(01) at (1.5,0.75) {$\cb$};
      \node(02) at (3,0.75) {$\cc$.};
        \node(19) at (-3,-0.75) {$\Elt A$};
      \node(10) at (0,-0.75) {$\cb$};
        \draw [->] (09) to node[auto,labelsize] {$Q'$} (00);
      \draw [->] (00) to node[auto,labelsize] {$P$} (01);
      \draw [->] (01) to node[auto,labelsize] {$X$} (02);
        \draw [->] (09) to node[auto,swap,labelsize] {$P'$} (19);
        \draw [->] (19) to node[auto,swap,labelsize] {$Q$} (10);
      \draw [->] (10) to node[auto,swap,labelsize] {$Y$} (02);
      \draw [2cell] (0,0.65) to node[auto,swap,labelsize] {$\phi'$} (0,-0.5);
\end{tikzpicture}
    \]
    Since $\eta Q'$ exhibits $(\Lan_PXP)Q$ as $\Lan_{P'}XPQ'$, there exists a unique natural transformation $\hat\phi'\colon (\Lan_PXP)Q\to YQ$ with $\phi'=\hat\phi'.(\eta Q')$, which corresponds to the required $\hat \phi$.

    {[(2)]}
    For each $b\in\cb$, $(\Ran_PXP)_b$ is given by the limit of $b/P\xrightarrow{\text{projection}} \Elt B\xrightarrow{P}\cb\xrightarrow{X}\cc$. Since the comma category 
    \begin{equation*}
\begin{tikzpicture}[baseline=-\the\dimexpr\fontdimen22\textfont2\relax ]
      \node(00) at (0,0.75) {$b/P$};
      \node(01) at (3,0.75) {$\Elt B$};
      \node(10) at (0,-0.75) {$1$};
      \node(11) at (3,-0.75) {$\cb$};
      \draw [->] (00) to node[auto,labelsize] {projection} (01);
      \draw [->] (01) to node[auto,labelsize] {$P$} (11);
      \draw [->] (00) to node[swap,auto,labelsize] {} (10);
      \draw [->] (10) to node[auto,swap,labelsize] {$b$} (11);
      \draw [2cell] (1.5,-0.55) to node[auto,labelsize] {} (1.5,0.55);
\end{tikzpicture}
\end{equation*}
    can also be given by the pullback 
    \begin{equation*}
\begin{tikzpicture}[baseline=-\the\dimexpr\fontdimen22\textfont2\relax ]
      \node(00) at (0,0.75) {$b/P$};
      \node(01) at (3,0.75) {$\Elt B$};
      \node(10) at (0,-0.75) {$b/\cb$};
      \node(11) at (3,-0.75) {$\cb$,};
      \draw [->] (00) to node[auto,labelsize] {projection} (01);
      \draw [->] (01) to node[auto,labelsize] {$P$} (11);
      \draw [->] (00) to node[swap,auto,labelsize] {} (10);
      \draw [->] (10) to node[auto,swap,labelsize] {cod} (11);
\end{tikzpicture}
\end{equation*}
    where the functor $\text{cod}\colon b/\cb\to \cb$ is the discrete opfibration corresponding to the representable functor $\cb(b,-)\colon \cb\to \Set$, 
    the composite $b/P\xrightarrow{\text{projection}} \Elt B\xrightarrow{P}\cb$ is the discrete opfibration corresponding to the functor $B\times \cb(b,-)\colon\cb\to \Set$. Thus by \cite[(3.33)]{Kelly-book}, for example, we see that $(\Ran_PXP)_b$ is the limit of $X$ weighted by $B\times \cb(b,-)$.

    To prove that $\Ran_PXP$ is the power $B\pitchfork X$, we first note that 
    for each $A\colon \cb\to \Set$, with the corresponding discrete opfibration $Q\colon\Elt A\to \cb$, the right Kan extension \eqref{eqn:Ran-for-power} is preserved by pasting the pullback diagram \eqref{eqn:pb-for-Elt-A-times-B} on the left, as we have $(b,\alpha)/\Elt A\cong b/\cb$ for each $(b,\alpha)\in \Elt A$.
    The rest of the proof is similar to that for (1).
    The final statement follows from the fact that $\varepsilon_{b',\beta}=\pi_{b',\beta,1_{b'}}\colon (B\pitchfork X)_{b'}\to X_{b'}$.
\end{proof}

\begin{construction}\label{const:enrichment-presheaf-base-change}
Construction~\ref{const:enrichment-by-presheaf-cat} has the following generalization. 
Assume that, in the situation of Construction~\ref{const:enrichment-by-presheaf-cat}, we are given another small category $\ca$ and a functor $F\colon \ca\to \cb$. 
Then, as above, we have the $[\ca,\Set]$-category $\underline{[\ca,\cc]}$ of all functors $\ca\to\cc$. 
On the other hand, the functor $F$ induces the (product-preserving) functor $\Ran_F\colon [\ca,\Set]\to [\cb,\Set]$, and hence we obtain the $[\cb,\Set]$-category $(\Ran_F)_\ast\underline{[\ca,\cc]}$ via change of base along $\Ran_F$. We will denote this $[\cb,\Set]$-category $(\Ran_F)_\ast\underline{[\ca,\cc]}$ by $F_\ast \underline{[\ca,\cc]}$ for short. 
Let $U,V\colon \ca\to \cc$. Then for any $B\colon \cb\to \Set$, we have a canonical bijection 
\[
[\cb,\Set]\bigl(B,F_\ast\underline{[\ca,\cc]}(U,V)\bigr)\cong [\ca,\Set]\bigl(BF, \underline{[\ca,\cc]}(U,V)\bigr)
\]
by $[F,\Set]\dashv \Ran_F$. 
It follows that to give a natural transformation $B\to F_\ast\underline{[\ca,\cc]}(U,V)$ is equivalent to giving a natural transformation from $\Elt(BF)\xrightarrow{\text{forgetful}}\ca\xrightarrow{U}\cc$ to $\Elt(BF)\xrightarrow{\text{forgetful}}\ca\xrightarrow{V}\cc$, and we have the end formula 
\[
[\cb,\Set]\bigl(B,F_\ast \underline{[\ca,\cc]}(U,V)\bigr)\cong \int_{a\in \ca}[B_{Fa},\cc(U_a,V_a)]. 
\]
In particular, for any $b\in \cb$, the set $F_\ast \underline{[\ca,\cc]}(U,V)_b$ is the set of all natural transformations from $b/F\xrightarrow{\text{forgetful}}\ca\xrightarrow{U}\cc$ to $b/F\xrightarrow{\text{forgetful}}\ca\xrightarrow{V}\cc$; it is straightforward to describe the $[\cb,\Set]$-category structure of $F_\ast \underline{[\ca,\cc]}$ in these terms. 
(As this latter description suggests, whenever we have a functor $\Phi\colon\cb\op\to \Cat/\ca$, we can turn $[\ca,\cc]$ into a $[\cb,\Set]$-category;
the above construction is a special case where $\Phi=(-)/F$.)
\end{construction}

Now recall Proposition~\ref{prop:T-simp-via-pullback}. 
We define the $s\Set$-category $\underline{s_T\ce}$ as the following pullback in $s\Set\text{-}\mathbf{CAT}$:
    \[ \begin{tikzcd}
   \underline{s_T\ce} \ar[r,"\und{(-)}"] \ar[d,"U"'] & {\underline{[\Delta^\mathrm{op},\ce_T]}} \ar[d,"{[F_R,\ce_T]}"] \\
    {(F_R)_\ast\underline{[\Dtopop,\ce]}} \ar[r,"{[\Dtopop,F_T]}"'] & {(F_R)_\ast\underline{[\Dtopop,\ce_T]}.}
  \end{tikzcd} \] 
Here, $[F_R,\ce_T]$ and $[\Dtopop,F_T]$ are the evident $s\Set$-functors.
Explicitly, this means the following.
Let $\ts X,\ts Y$ be $T$-simplicial objects, corresponding to the diagrams 
\[ \begin{tikzcd}
   \Dtopop \ar[r,"F_R"] \ar[d,"\res{X}"'] & \Delta\op \ar[d,"\und{X} "] \\
    \ce \ar[r,"F_T"'] & \ce_T 
  \end{tikzcd} \qquad\text{and}\qquad
  \begin{tikzcd}
   \Dtopop \ar[r,"F_R"] \ar[d,"\res{Y}"'] & \Delta\op \ar[d,"\und{Y} "] \\
    \ce \ar[r,"F_T"'] & \ce_T 
  \end{tikzcd} 
\]
respectively. 
Then, for any simplicial set $B\colon \Delta\op\to \Set$, to give a simplicial map $\theta\colon B\to \underline{s_T\ce}(\ts X,\ts Y)$ is equivalent to giving 
a pair of natural transformations 
\[
\begin{tikzpicture}[baseline=-\the\dimexpr\fontdimen22\textfont2\relax ]
      \node(1) at (0,0) {$\Elt(BF_R)$};
      \node(2) at (2,1) {$\Dtopop$};
      \node(3) at (2,-1) {$\Dtopop$};
      \node(4) at (4,0) {$\ce$};
      \draw [->] (1) to node[auto, labelsize] {forgetful}  (2);  
      \draw [->] (2) to node[auto, labelsize] {$\res{X}$}  (4);
      \draw [->] (1) to node[auto, swap, labelsize] {forgetful}  (3);
      \draw [->] (3) to node[auto, swap, labelsize] {$\res{Y}$}  (4);
      \draw[2cell] (2,0.7) to node[auto, labelsize] {$\res{\theta}$} (2,-0.7);
\end{tikzpicture}\qquad\text{and}\qquad
\begin{tikzpicture}[baseline=-\the\dimexpr\fontdimen22\textfont2\relax ]
      \node(1) at (0,0) {$\Elt B$};
      \node(2) at (2,1) {$\Delta\op$};
      \node(3) at (2,-1) {$\Delta\op$};
      \node(4) at (4,0) {$\ce_T$};
      \draw [->] (1) to node[auto, labelsize] {forgetful}  (2);  
      \draw [->] (2) to node[auto, labelsize] {$\und{X}$}  (4);
      \draw [->] (1) to node[auto, swap, labelsize] {forgetful}  (3);
      \draw [->] (3) to node[auto, swap, labelsize] {$\und{Y}$}  (4);
      \draw[2cell] (2,0.7) to node[auto, labelsize] {$\und{\theta}$} (2,-0.7);
\end{tikzpicture}
\]
such that $F_T.\res{\theta}=\und{\theta}.F'_R$, where $F'_R\colon \Elt(BF_R)\to \Elt B$ is defined by the pullback
\[
\begin{tikzcd}
   \Elt(BF_R) \ar[r,"F'_R"] \ar[d,"\text{forgetful}"'] & \Elt B \ar[d,"\text{forgetful} "] \\
    \Dtopop \ar[r,"F_R"'] & \Delta\op.
  \end{tikzcd}
\]
Note that $F'_R$ is bijective on objects since $F_R$ is. 
Hence a simplicial map $\theta\colon B\to \underline{s_T\ce}(\ts X,\ts Y)$ corresponds to the data 
\begin{itemize}
    \item a morphism $\theta_\beta\colon X_m\to Y_m$ in $\ce$ for each $m\geq 0$ and $\beta\in B_m$
\end{itemize}
satisfying the conditions
\begin{itemize}
    \item for any $\psi\colon p\to m$ in $\Dtop$ and $\beta\in B_m$,
   the diagram 
    \[
    \begin{tikzcd}
      X_m \ar[r,"\theta_\beta"] \ar[d,"\psi^\ast"'] & Y_m \dar["\psi^\ast"] \\
      X_p \rar["\theta_{\beta\psi}"'] & Y_p
    \end{tikzcd}
  \]
  (in $\ce$) commutes, where $\beta\psi\in B_p$ is the image of $\beta$ under $\psi^\ast\colon B_m\to B_p$, and 
  \item
  for any $\psi\colon
    p\to m$ in $\Delta$ which is not in $\Dtop$ and $\beta\in B_m$, the diagram 
        \[
    \begin{tikzcd}
      X_m \ar[r,"\theta_\beta"] \ar[d,"\psi^\ast"'] & Y_m \dar["\psi^\ast"] \\
      TX_p \rar["T\theta_{\beta\psi}"'] & TY_p
    \end{tikzcd}
  \]
  (in $\ce$) commutes; 
  it suffices just to do the case $\psi=\delta_m\colon m-1\to m$, so that the condition becomes $d_m.\theta_{\beta}=T\theta_{\beta\delta_m}.d_m$.
\end{itemize}

In particular, for each $n\geq 0$, an $n$-simplex
$x\in \underline{s_T\ce}(\ts X,\ts Y)_n$ consists of an assignment, for each $m\geq 0$ and $\phi\colon m\to n$ in $\Delta$, of a morphism $x_{\phi}\colon Y_m\to X_m$
in $\ce$, subject to two conditions:
\begin{itemize}
\item for any $\psi\colon p\to m$ in $\Dtop$ and $\phi\colon m\to n$ in $\Delta$,
   the diagram 
    \[
    \begin{tikzcd}
      X_m \ar[r,"x_\phi"] \ar[d,"\psi^\ast"'] & Y_m \dar["\psi^\ast"] \\
      X_p \rar["x_{\phi\psi}"'] & Y_p
    \end{tikzcd}
  \]
  (in $\ce$) commutes, and 
  \item 
  for any $m>0$ and $\phi\colon m\to n$ in $\Delta$, the diagram 
        \[
    \begin{tikzcd}
      X_m \ar[r,"x_\phi"] \ar[d,"d_m"'] & Y_m \dar["d_m"] \\
      TX_{m-1} \rar["Tx_{\phi\delta_m}"'] & TY_{m-1}
    \end{tikzcd}
  \]
  (in $\ce$) commutes.
  \end{itemize}

We follow the common practice of writing $\Delta[n]$ for the 
representable $\Delta(-,n)$, and likewise write $\Dtop[n]$ for $\Dtop(-,n)$.

\begin{remark}
    When $\ce=\Set$ and $T=1_\Set$,
    the resulting $s\Set$-category $\underline{s\Set}$ is the one 
    induced by the fact that $s\Set$ is cartesian closed: $\underline{s\Set}(\ts X,\ts Y)_n\cong s\Set(\Delta [n]\times \ts X ,\ts Y)$. 
    More generally, when $\ce$ is an arbitrary category but $T=1_\ce$,
    then $\underline{s\ce}$ coincides with the standard $s\Set$-enrichment of the category $s\ce=[\Delta\op,\ce]$ of simplicial objects in $\ce$, described in e.g.\ \cite[Definition~2.1]{Kan-css-cat} (using the copower in $\underline{s_T\ce}$; see Proposition~\ref{prop:copower-in-T-Simp} and Remark~\ref{rmk:enrichment-from-action}).
\end{remark}

\begin{remark}\label{rmk:RanFR}
    In the definition of $\underline{s_T\ce}$, we used $\Ran_{F_R}\colon [\Dtopop,\Set]\to [\Delta\op,\Set]$.
    Here we note that it is possible to give an explicit description of this. 
    Namely, for any category $\ca$ with finite products, functor $A\colon \Dtopop\to \ca$, and $n\geq 0$, we have 
    \begin{equation}\label{eqn:RanFR}
    (\Ran_{F_R}A)_n\cong A_0\times A_{1}\times\dots\times A_n.
    \end{equation}
    This follows from the fact that the functor $F_R\colon \Dtopop\to \Delta\op$ is, in addition to being a left adjoint, also a right multi-adjoint. 
    Indeed, we have 
    \begin{equation}\label{eqn:FR-right-multi-adj}
    \Delta(F_R(-),n)\cong
    \Dtop(-,0)+\Dtop(-,1)+\dots + \Dtop(-,n)
    \end{equation}
    as functors $\Dtopop\to\Set$,
    which simply says that every morphism $m\to n$ in $\Delta$ has a unique factorization of the form $m\xrightarrow{\psi}k\xrightarrow{\iota} n$, where $0\leq k\leq n$, $\psi$ is top-preserving (i.e., is in $\Dtop$), and $\iota= \delta_{n}\delta_{n-1}\dots\delta_{k+1}$.
    (To deduce \eqref{eqn:RanFR} from \eqref{eqn:FR-right-multi-adj}, first recall that $(\Ran_{F_R} A)_n$ is the weighted limit of $A$ with weight $\Delta(F_R(-),n)$. Thus the coproduct decomposition \eqref{eqn:FR-right-multi-adj} of the weight induces a product decomposition of the weighted limit $(\Ran_{F_R} A)_n$, each of whose factors is the weighted limit of $A$ with a representable weight $\Dtop(-,k)=\Dtop[k]$, which is just $A_k$.) 
\end{remark}

We shall study copowers and powers in $\underline{s_T\ce}$, and the full sub-$s\Set$-category $\underline{\Cat_T(\ce)}$ of $\underline{s_T\ce}$ consisting of all $T$-categories.
More specifically, we show the following. 
\begin{itemize}
    \item If $\ce$ has copowers (as a $\Set$-category), then so does $\underline{s_T\ce}$ (as an $s\Set$-category). See Proposition~\ref{prop:copower-in-T-Simp}. 
    \item For any $T$-category $\ts X$ and any $T$-simplicial object $\ts Y$, the simplicial set $\underline{s_T\ce}(\ts Y,\ts X)\in s\Set$ is the nerve of a category, i.e., 
    we have $\underline{s_T\ce}(\ts Y,\ts X)\in \Cat$ (Proposition~\ref{prop:homs-are-cats}).
    This implies that the $s\Set$-category $\underline{\Cat_T(\ce)}$ is in fact a $2$-category. 
    We show that the $2$-cells in $\underline{\Cat_T(\ce)}$ are given by the $T$-natural transformations (defined in \cite[IV.3]{Burroni-Multicategories} and \cite[Section~5.3]{Leinster-book} when $T$ is cartesian). See Theorem~\ref{thm:2cat-bis}.
    \item If $\ce$ has pullbacks, then $\underline{s_T\ce}$ has powers by $\Delta[1]\in s\Set$ (Theorem~\ref{thm:power-by-Delta-1}) and $T$-categories are closed under powers by $\Delta[1]$ (Proposition~\ref{prop:T-cat-power}).
    This implies that the $2$-category $\underline{\Cat_T(\ce)}$ has powers by $\mathbf{2}$ (Theorem~\ref{thm:T-cat-power-by-2}).
    \item If $\ce$ is locally finitely presentable and $T$ is finitary, then $\underline{s_T\ce}$ is locally finitely presentable as an $s\Set$-category and $\underline{\Cat_T(\ce)}$ is so as a $2$-category. See Theorems~\ref{thm:s-Presentable} and \ref{thm:Cat-presentable}.
\end{itemize}

\begin{proposition}\label{prop:copower-in-T-Simp}
    Let $\ce$ be a locally small category with copowers and $T$ an arbitrary monad on $\ce$. Then the $s\Set$-category $\underline{s_T\ce}$ has copowers. 
\end{proposition}
\begin{proof}
    Let $\ts Y$ be a $T$-simplicial object and $A$ a simplicial set. We define a new $T$-simplicial object
    $\ts Z$.
    For each $n\geq 0$, we set $Z_n=A_n\cdot Y_n$; for $a\in A_n$, we write
    $i_{a}\colon Y_n\to Z_n$ for the inclusion of the $a$-component.

    Given 
    $\phi\colon m\to n$ in $\Dtop$, define $\phi^{*}\colon Z_n\to Z_m$ so that the diagram 
\[
  \begin{tikzcd}
    Y_{n} \rar["i_{a}"] \dar["\phi^\ast"'] & Z_{n}
    \dar["\phi^\ast"] \\
    Y_m \rar["i_{\phi^\ast a}"'] & Z_m
  \end{tikzcd}
\]
    commutes for each $a\in A_n$.
    Similarly, for each $n\geq 0$, define $d_{n+1}\colon Z_{n+1}\to TZ_n$ so that the diagram
\[
  \begin{tikzcd}
    Y_{n+1} \rar["i_{a}"] \dar["d_{n+1}"'] & Z_{n+1}
    \dar["d_{n+1}"] \\
    TY_n \rar["Ti_{d_{n+1}a}"'] & TZ_n
  \end{tikzcd}
\]
    commutes for each $a\in A_{n+1}$. This now defines a $T$-simplicial object $\ts Z$. (When $T=1_\ce$, this construction appears in \cite[Definition~3.5]{Kan-homotopy-rel}.)
    It is straightforward to check that $\ts Z$ is the copower $A\cdot \ts Y$ in $\underline{s_T\ce}$.
     \end{proof}

\begin{corollary}\label{cor:finite-copower-T-Simp}
    Let $\ce$ be a category with finite copowers and $T$ an arbitrary monad on $\ce$, Then the $s\Set$-category $\underline{s_T\ce}$ has copowers by representable simplicial sets $\Delta[k]$.
\end{corollary}
\begin{proof}
    More generally, the proof of Proposition~\ref{prop:copower-in-T-Simp} shows that $\underline{s_T\ce}$ has copowers by any simplicial set $A$ such that each $A_n$ is finite. 
\end{proof}

\begin{remark}\label{rmk:enrichment-from-action}
    Suppose that the category $\ce$ has copowers.
    The construction in the proof of Proposition~\ref{prop:copower-in-T-Simp} defines a strong action ${\cdot}\colon s\Set\times s_T\ce\to s_T\ce$, such that for each $\ts Y\in s_T\ce$, the functor $(-)\cdot \ts Y\colon s\Set\to s_T\ce$ has a right adjoint $\underline{s_T\ce}(\ts Y,-)\colon s_T\ce\to s\Set$ (note that we assume that $\ce$ is locally small). 
    This then defines the $s\Set$-category $\underline{s_T\ce}$ by abstract reasons, giving an alternative way to define the $s\Set$-enrichment.
    (Even when $\ce$ does not have copowers, one can embed $\ce$ into a locally small category $\overline{\ce}$ with copowers so that the monad $T$ on $\ce$ extends to a monad on $\overline{\ce}$, and argue as above; for example, $\overline{\ce}=[\ce\op,\Set]$ would work when $\ce$ is small, whereas $\overline{\ce}=\Fam(\ce)$ and $\mathcal{P}\ce$ \cite{Day-Lack-small} work for an arbitrary $\ce$.)  
\end{remark}

\begin{remark}\label{rmk:F-cat} 
We conclude this section by sketching an alternative approach to the simplicial enrichment of $\underline{s_T\ce}$ just described.

    To give a functor $F\colon \cd_t\to\cd_\ell$ which acts as the identity
on objects is equivalent to giving a category $\bbd$ enriched over the
cartesian closed category $\Set^\two$ \cite[Example~1]{Power-premonoidal}. In slightly more detail, for
objects $A$ and $B$, we have the two hom-sets $\cd_t(A,B)$ and
$\cd_\ell(A,B)$ and a function between them; this is the
$\Set^\two$-valued hom $\bbd(A,B)$.

From this point of view, a $\Set^\two$-enriched functor is just a
commutative square of ordinary functors involving the two
identity-on-objects functors \cite[Example~1]{Power-premonoidal}.

Thus we can think of the Kleisli functors 
$F_R\colon\Dtopop\to\Delta\op$ and $F_T\colon\ce\to\ce_T$ as $\Set^\two$-enriched
categories $\bbd\op$ and $\bbe$, and so identify the category $s_T\ce$
with the category of $\Set^\two$-enriched functors from $\bbd\op$ to
$\bbe$.

It is true in general that for a complete cartesian closed category $\bbs$ and
$\bbs$-categories $\bbd$ and $\bbe$ (with $\bbd$ small)  the category
of $\bbs$-enriched
functors $\bbd\op\to\bbe$ can be enriched not just over $\bbs$ but over
the cartesian closed $\bbs$-category $[\bbd\op,\bbs]$. It is also true
in general that there is no real difference between categories
enriched over $[\bbd\op,\bbs]$ and categories enriched over the
underlying ordinary (cartesian closed) category of $\bbs$-functors
from $\bbd\op$ to $\bbs$.

This can be used to give an enrichment of $s_T\ce$ over the cartesian closed $\Set^\two$-category $[\bbd\op,\Set^\two]$, and so over its underlying cartesian closed ordinary category $s\Set$.
\end{remark}

\section{The $2$-category of $T$-categories}\label{sect:2cat}

Throughout this section, let $\ce$ be a locally small category and $T=(T,m,i)$ an arbitrary monad on $\ce$. 

\begin{proposition}\label{prop:homs-are-cats}
  Let $\ts B$ be a $T$-category and $\ts A$ a $T$-simplicial object. 
  Then the simplicial set $\underline{s_T\ce}(\ts A,\ts B)$ is
  (the nerve of) a category. 
\end{proposition}
\begin{proof}
It suffices to show that the simplicial set $\underline{s_T\ce}(\ts A,\ts B)$ satisfies the nerve condition, i.e., that for each $n\geq 0$, the square 
\[
  \begin{tikzcd}
    \underline{s_T\ce}(\ts A,\ts B)_{n+2} \ar[r,"d_{n+2}" ] \ar[d,"d_0"' ] & \underline{s_T\ce}(\ts A,\ts B)_{n+1} \ar[d,"d_0"]
    \\
    \underline{s_T\ce}(\ts A,\ts B)_{n+1} \ar[r,"d_{n+1}"' ] & \underline{s_T\ce}(\ts A,\ts B)_n
  \end{tikzcd}
\]
is a pullback of sets.
Let $n\geq 0$, and suppose that we are given $x,y\in \underline{s_T\ce}(\ts A,\ts B)_{n+1}$ with
$d_0x=d_{n+1}y$; this means that we have $x_{\delta_0\phi''}=y_{\delta_{n+1}\phi''}$ for all
$\phi''\colon [m]\to [n]$ in $\Delta$. (In this proof, we write an object $n$ of $\Delta$ as $[n]$ because we will also refer to elements of $[n]=\{0,1,\dots, n\}$.) We are to show that there is a unique
$z\in \underline{s_T\ce}(\ts A,\ts B)_{n+2}$ with $d_0z=y$ and $d_{n+2}z=x$; in
other words, $z_{\delta_0\phi'}=y_{\phi'}$ and
$z_{\delta_{n+2}\phi'}=x_{\phi'}$ for all $\phi'\colon [m]\to [n+1]$ in $\Delta$.

Thus we define a morphism $z_\phi\colon A_m\to B_m$ in $\ce$ for each $\phi\colon [m]\to [n+2]$ in $\Delta$.
We first do so for all $\phi\colon [m]\to [n+2]$ with $\phi 0>0$ (Case~1) or with $\phi m < n+2$ (Case~2). 
Using these cases and induction, we then treat the case where $\im(\phi)$ contains $\{0,n+2\}$ as a proper subset (Case~3).
Finally, using this last case, we cover the remaining case where $\im(\phi)=\{0,n+2\}$ (Case~4). 

\subsection*{Case 1: $\phi 0>0$}
This is equivalent to $\phi$ having the form $\phi=\delta_0\phi'$ for some $\phi'\colon [m]\to [n+1]$, necessarily equal to
$\sigma_0\phi$; then in order that $d_0z=y$ we must define
$z_{\phi}=z_{\delta_0\phi'}=(d_0z)_{\phi'}=y_{\phi'}$.
Furthermore, if $\psi\colon [p]\to [m]$ is
top-preserving then also $\phi\psi0>0$ and
\[ \psi^{*}z_{\phi} =
  \psi^{*}y_{\phi'}= y_{\phi'\psi}\psi^{*} =
  z_{\delta_0\phi'\psi}\psi^{*}=z_{\phi\psi}\psi^{*} \]
while if $m>0$ then 
\[ d_m.z_{\phi} = d_m.y_{\phi'}=Ty_{\phi'\delta_m}.d_m =
  Tz_{\delta_0\phi'\delta_m}.d_m = Tz_{\phi\delta_m}.d_m. \]

\subsection*{Case 2: $\phi m<n+2$}
This is equivalent to $\phi$ having the form $\phi=\delta_{n+2}\phi'$
for some $\phi'\colon [m]\to [n+1]$, necessarily equal to $\sigma_{n+1}\phi$. Similarly
to the previous case, we
must define $z_{\phi}=x_{\phi'}$, and if $\psi$ is top-preserving then
$\psi^{*}z_{\phi}=z_{\phi\psi} \psi^{*}$, while if $m>0$ then $d_mz_{\phi}=Tz_{\phi\delta_m}.d_m$.

Note also that if $\phi0>0$ and $\phi m<n+2$ then in fact
$\phi=\delta_{n+2}\delta_0\phi''$ for some $\phi''\colon [m]\to [n]$, necessarily
equal to $\sigma_0\sigma_{n+1}\phi$. Then
$x_{\delta_0\phi''}=(d_0x)_{\phi''}=(d_{n+1}y)_{\phi''}=y_{\delta_{n+1}\phi''}$
and so the two definitions are consistent.

That leaves the case where $\phi0=0$ and $\phi m=n+2$ (which rules out $m=0$).

\subsection*{Case 3: $\mathrm{im}(\phi)$ contains $\{0,n+2\}$ as a proper subset}

In this case $m\ge 2$. 
If $\phi$ is of this type, then $\phi\delta_0$ is either also of this
type or is covered by Case~1. Likewise
$\phi\delta_m$ is either of this type or
is covered by Case~2.

We now prove by induction on $m$ that there is a
unique $z_{\phi}$ with $d_0z_{\phi}=z_{\phi\delta_0}d_0$ and
$d_mz_{\phi}=Tz_{\phi\delta_m}.d_m$; we
already have a uniquely determined $z_{\phi\delta_0}$ and $z_{\phi\delta_m}$.
The inductive step follows by
\[ Td_0.Tz_{\phi\delta_m}.d_m = Tz_{\phi\delta_m\delta_0}.Td_0.d_m =
  Tz_{\phi\delta_0\delta_{m-1}}.d_{m-1}.d_0 =
  d_{m-1}.z_{\phi\delta_0}.d_0 \]
and the fact that $B_m$ can be written as a pullback as in \eqref{eqn:nerve-condition}.

We further show that this definition of $z_{\phi}$ is compatible with
all those $\delta_h\colon [m-1]\to [m]$ for which $z_{\phi\delta_h}$ has
so far been 
defined and with all $\sigma_h\colon [m+1]\to [m]$. Compatibility with
$\delta_0,\delta_m\colon [m-1]\to [m]$ holds by definition of
$z_{\phi}$. If $0<h<m$ then $z_{\phi\delta_h}$ will have been defined
just when $\phi\delta_h$ is of Case~3, while $\phi\sigma_h$ is always
of Case~3. Compatibility follows using the pullback property again. 

\subsection*{Case 4: $\mathrm{im}(\phi)=\{0,n+2\}$}

Let $i=i(\phi)$ be the least element of $[m]=\{0<1<\dots <m\}$ with $\phi i=n+2$.
For $1\le j\le n+1$, there is a unique map $\phi^j\colon[m+1]\to[n+2]$
with $\phi^j\delta_i=\phi$ and
$\mathrm{im}(\phi^j)=\{0,j,n+2\}$. Note also that the $i$ is also
determined by the facts that $\phi^j\delta_i=\phi$ and $\mathrm{im}(\phi^j)=\{0,j,n+2\}$.

Now $\phi^j$ and $\phi^j\sigma_i$ are both covered by Case~3, and so
we have maps $z_{\phi^j}$ and $z_{\phi^j\sigma_i}$ making the
solid part of the diagram
\[
  \begin{tikzcd}
    A_{m+2} \rar["z_{\phi^j\sigma_i}"] \dar[shift left=2,"d_{i+1}"]
    \dar[shift right=2,"d_i" '] & B_{m+2} \dar[shift left=2,"d_{i+1}"]
    \dar[shift right=2,"d_i" '] \\
    A_{m+1} \rar["z_{\phi^j}"] \dar["d_i"] & B_{m+1} \dar["d_i"] \\
    A_m \rar[dashed,"z_{\phi}"] & B_m 
  \end{tikzcd}
\]
commute, but the vertical forks are (split) coequalizers, so there is
a unique induced $z_{\phi}$ making the lower part commute.

On the
face of it, this definition of $z_{\phi}$ depends on $j$, but in
fact that is not the case. For given $1\le j<j'\le n+1$, there is a
unique $\phi^{j,j'}\colon[m+2]\to[n+2]$ with
$\phi^{j,j'}\delta_i=\phi^{j'}$ and
$\phi^{j,j'}\delta_{i+1}=\phi^j$. Now
\begin{gather*}
  d_i.z_{\phi^{j'}} = d_i.z_{\phi^{j,j'}\delta_i}.d_is_i =
  d_i.d_i.z_{\phi^{j,j'}}.s_i
  = d_id_{i+1}.z_{\phi^{j,j'}}.s_i \\
  =
  d_i.z_{\phi^{j,j'}\delta_{i+1}}.d_{i+1}s_i =
  d_i.z_{\phi^j}
\end{gather*}
and so $d_i.z_{\phi^j}$ is independent of $j$. Observe also that if
$\phi$ falls under Case 3, say with $\mathrm{im}(\phi)=\{0,j,n+2\}$
but $\phi\delta_h$ falls under Case 4, then $i(\phi\delta_h)=h$ and
$(\phi\delta_h)^j=\phi$, and now
$d_h.z_\phi=d_h.z_{(\phi\delta_h)^j}=z_{\phi\delta_h}.d_h$, which
completes the verification of compatibility for all $\phi$ in Case~3.

\subsection*{Remaining verifications}
We have now shown that $z$ is unique if it exists, and how to define
$z_{\phi}$ for all $\phi$. We have also checked the compatibility
conditions for $z_{\phi}$ whenever $\phi$ comes under Cases~1--3.
It remains to check the compatibility conditions for $\phi$ in Case~4.

Suppose then that $\phi\colon[m]\to[n+2]$ has image $\{0,n+2\}$ and
that $\psi\colon[p]\to[m]$ is either $\delta_h\colon[m-1]\to[m]$ for
$0<h<m$, or $\sigma_h\colon[m+1]\to[m]$ for $0\le h\le m$, or possibly
$\delta_0$ in the case that $\phi\delta_00=0$. Then
$\phi\psi$ also has image $\{0,n+2\}$. Let $i=i(\phi)$ and $i'=i(\phi\psi)$. There is a unique
$\theta\colon[p+1]\to[m+1]$ such that $\theta\delta_{i'}=\delta_i\psi$
and $\theta i'=i$. Then $(\phi\psi)^1=\phi^1\theta$, and so
\[ \psi^{*} z_{\phi} d_i = \psi^{*}d_i
  z_{\phi^1}=d_{i'}\theta^{*}z_{\phi^1} = d_{i'}
  z_{\phi^1\theta}\theta^{*}
= z_{\phi\psi}d_{i'}\theta^{*}= z_{\phi\psi}\psi^{*}d_i\]
and $\psi^{*}z_{\phi}=z_{\phi\psi}\psi^{*}$.

Next consider $\psi=\delta_0$ where $\phi\delta_00>0$. Then
$\phi\delta_0=\delta_0\tau$ for some $\tau\colon[m-1]\to[n+1]$, and
$\phi=\phi^1\delta_1$. Then $\phi^1$ is in Case~3 and $\phi^1\delta_0$
is in Case~1. Thus
\[ d_0z_{\phi}d_1 = d_0d_1z_{\phi^1}=d_0d_0
  z_{\phi^1}=d_0z_{\phi^1\delta_0}d_0
  = z_{\phi^1\delta_0\delta_0}d_0d_0= z_{\phi^1\delta_1\delta_0}d_0d_1
  = z_{\phi\delta_0}d_0d_1\]
and so $d_0z_{\phi}=z_{\phi\delta_0}d_0$.

That leaves the case of $\psi=\delta_m\colon [m-1]\to[m]$.
If $\phi(m-1)=n+2$ then $\phi\delta_m$ is still in Case~4, and $i=i(\phi)<m$
and 
$(\phi\delta_m)^1=\phi^1\delta_{m+1}$, and so
\begin{align*}
  d_m z_{\phi} d_i &= d_m d_i z_{\phi^1} = Td_i. d_{m+1} z_{\phi^1} =
                     Td_i.Tz_{\phi^1\delta_{m+1}}.d_{m+1} \\
  &=
  Tz_{\phi\delta_m}.Td_i.d_{m+1} =
  Tz_{\phi\delta_m}.d_m.d_i
\end{align*}
and $d_mz_{\phi}=Tz_{\phi\delta_m}.d_m$.

If on the other hand $\phi(m-1)<n+2$ then in fact $\phi(m-1)=0$,
$i=m$, and
$\phi\delta_m$ is in Case~2. Then
\begin{align*}
  d_m z_{\phi} d_m &= d_m d_m z_{\phi^1} =
                     mB_{m-1}.Td_m.d_{m+1}.z_{\phi^1} =
                     mB_{m-1}.Td_m.Tz_{\phi^1\delta_{m+1}}.d_{m+1} \\
                   &= mB_{m-1}.T^2z_{\phi^1\delta_{m+1}\delta_m}.Td_m.d_{m+1}
                     = Tz_{\phi^1\delta_m\delta_m}.mA_{m-1}.Td_m.d_{m+1}
  \\
  &= Tz_{\phi\delta_m}.d_m.d_m
\end{align*}
and so $d_m z_{\phi} = Tz_{\phi\delta_m}.d_m$. 
\end{proof}

Therefore the full sub-$s\Set$-category $\underline{\Cat_T(\ce)}$ of $\underline{s_T\ce}$ consisting of all $T$-categories is a $2$-category. 
The $2$-cells of $\underline{\Cat_T(\ce)}$ are the elements of $\underline{\Cat_T(\ce)}(\ts A,\ts B)_1$. 
We will give a more explicit description of these in Theorem~\ref{thm:2cat}. 
An alternative way to define a $2$-category of $T$-categories (without assuming that $T$ is cartesian) is to take the vertical $2$-category of the unital virtual double category $\mathbb{M}\mathsf{od}\bigl(\mathbb{H}\text{-}\mathsf{Kl}(\mathbb{S}\mathsf{pan}(\ce),T)\bigr)$ defined in \cite[B.11--13]{CruttwellShulman}.

\begin{proposition}\label{prop:n-cells}
   Let $\ts B$ be a $T$-category and $\ts A$ a $T$-simplicial object. Then to give an element $x$ of
   $\underline{s_T\ce}(\ts A,\ts B)_n$ we need only give $x_\phi\colon A_m\to B_m$ for
   $\phi\colon [m]\to [n]$ in $\Delta$ with $m\le 1$, and need only check
   the compatibility condition of the induced family $(x_\phi\colon A_m\to B_m)_{m\geq 0,\phi\in\Delta(m,n)}$ with maps 
   $\delta_1\colon [1]\to [2]$, $\delta_0\colon [0]\to [1]$, $\delta_1\colon [0]\to [1]$, and 
   $\sigma_0\colon [1]\to[0]$.
   \end{proposition}

  \begin{proof}
    First observe that the definition of $x_\phi\colon A_m\to B_m$ is determined
    by the cases with $m\le 1$. This follows by an easy
    induction: if $m\geq 2$ then use the fact that $B_m$ is a pullback to
    construct $x_\phi$ from $x_{\phi\delta_0}$ and
    $x_{\phi\delta_m}$.

    Suppose then we have data as in the proposition. Then using the
    same pullback property, we can recursively construct $x_\phi$ for all
    $\phi\colon [m]\to [n]$ in such a way that
    $d_0.x_\phi=x_{\phi\delta_0}.d_0$ and
    $d_m.x_\phi=Tx_{\phi\delta_m}.d_m$ holds. Similarly,
    we can use the pullback property to prove inductively the compatibility conditions.
  \end{proof}

We now deduce:

\begin{theorem}\label{thm:2cat}
For an arbitrary monad $T=(T,m,i)$ on a category $\ce$,
the category $\Cat_T(\ce)$ of $T$-categories can be made into a $2$-category $\underline{\Cat_T(\ce)}$,
with $2$-cells between $f,g\colon \ts A\to \ts B$ consisting
of a morphism $\hat\alpha\colon A_1\to B_1$ such that the diagrams
\begin{equation}\label{eqn:T-nat-tr-axioms}
\begin{tikzcd}
    TA_0 \ar[d,"Tf_0"] & A_1 \ar[l,"d_1" '] \ar[r,"d_0"]
    \ar[d,"\hat\alpha"] & A_0 \ar[d,"g_0"] \\
    TB_0 & B_1 \ar[l,"d_1" '] \ar[r,"d_0"] & B_0
\end{tikzcd}
     \quad
     \begin{tikzcd}
        A_2 \ar[r,"\hat\alpha'"] \ar[d,"d_1"] & B_2 \ar[d,"d_1"]
    \\
    A_1 \ar[r,"\hat\alpha"] & B_1  
     \end{tikzcd}\quad
     \begin{tikzcd}
         A_2 \ar[r,"\hat\alpha''"] \ar[d,"d_1"] & B_2 \ar[d,"d_1"]
    \\
    A_1 \ar[r,"\hat\alpha"] & B_1 
     \end{tikzcd}
\end{equation}
commute, where $\hat\alpha'$ and $\hat\alpha''$ are defined by
commutativity of
\[
\begin{tikzcd}
    TA_1 \ar[d,"Tf_1"] & A_2 \ar[l,"d_2"'] \ar[r,"d_0"]
    \ar[d,"\hat\alpha'"] & A_1 \ar[d,"\hat\alpha"] \\
    TB_1 & B_2 \ar[l,"d_2"'] \ar[r,"d_0"] & B_1 
\end{tikzcd}\quad
\begin{tikzcd}
    TA_1 \ar[d,"T\hat\alpha"] & A_2 \ar[l,"d_2"'] \ar[r,"d_0"]
    \ar[d,"\hat\alpha''"] & A_1 \ar[d,"g_1"] \\
    TB_1 & B_2 \ar[l,"d_2"'] \ar[r,"d_0"] & B_1.
\end{tikzcd}
\]
\end{theorem}

\begin{proof}
  In the notation of Proposition~\ref{prop:n-cells}, $\hat\alpha$ is
  given by $x_1$ for the identity $1\colon[1]\to[1]$.
  The maps $x_{\phi}$ with
  $0\notin\im(\phi)$ determine $g$, and the maps $x_{\phi}$
  with $1\notin\im(\phi)$ determine $f$.
  The definitions of $\hat \alpha'$ and $\hat\alpha''$ agree with those of $x_{\sigma_0}$ and $x_{\sigma_1}$, respectively, in the proof of Proposition~\ref{prop:n-cells}.
  
  Commutativity of the first two squares in \eqref{eqn:T-nat-tr-axioms} is equivalent to the conditions on
  $x_{1}$ involving $\delta_1,\delta_0\colon[0]\to[1]$
  respectively. Commutativity of the next two squares in \eqref{eqn:T-nat-tr-axioms} is equivalent to the
  conditions on $x_{\sigma_0}$ and $x_{\sigma_1}$ involving
  $\delta_1\colon[1]\to[2]$. 
  Other compatibility conditions (such as those with $\sigma_0\colon [1]\to[0]$) follow from the assumption that $f$ and $g$ are $T$-functors.
\end{proof}

\begin{definition}\label{def:T-nat-tr}
    Let $\ts A$ and $\ts B$ be $T$-categories and $f,g\colon \ts A\to \ts B$ be $T$-functors. A \emph{$T$-natural transformation} $\alpha\colon f\to g$ is a morphism $\alpha\colon A_0\to B_1$ in $\ce$ making the diagrams 
    \[ \begin{tikzcd}
        B_0 \ar[d,"iB_0"] & A_0 \ar[l,"f_0"'] \ar[dr,"g_0"]
    \ar[d,"\alpha"] \\
    TB_0 & B_1 \ar[l,"d_1"'] \ar[r,"d_0"] & B_0
    \end{tikzcd}\quad
    \begin{tikzcd}
        A_1 \ar[r,"\alpha'"] \ar[d,"\alpha''"] & B_2 \ar[d,"d_1"]
    \\
    B_2 \ar[r,"d_1"] & B_1 
    \end{tikzcd}
\]
commute, where $\alpha'$ and $\alpha''$ are defined by
commutativity of
\[  \begin{tikzcd}
    B_1 \ar[d,"iB_1"] & A_1 \ar[l,"f_1"'] \ar[r,"d_0"]
    \ar[d,"\alpha'"] & A_0 \ar[d,"\alpha"] \\
    TB_1 & B_2 \ar[l,"d_2"'] \ar[r,"d_0"] & B_1 
\end{tikzcd} \quad
\begin{tikzcd}
    TA_0 \ar[d,"T\alpha"] & A_1 \ar[l,"d_1"'] \ar[dr,"g_1"]
    \ar[d,"\alpha''"]
    \\
    TB_1 & B_2 \ar[l,"d_2"'] \ar[r,"d_0"] & B_1.
\end{tikzcd}
\]
\end{definition}

\begin{theorem}\label{thm:2cat-bis}
  A $2$-cell $\hat\alpha$ as in Theorem~\ref{thm:2cat} is completely determined by the morphism
$\alpha\colon A_0\to B_1$ given by the composite
$\hat\alpha.s_0$.
This gives rise to a bijective correspondence between the set of all $2$-cells $f\to g$ and the set of all $T$-natural transformations $f\to g$. 
\end{theorem}

\begin{proof}
  To see that we may take $\alpha$ to be $\hat\alpha.s_0$, simply take
  $\alpha'=\hat\alpha'.s_1$ and $\alpha''=\hat\alpha''.s_0$.
  The various conditions in Theorem~\ref{thm:2cat} then easily imply the
  conditions above.

  Conversely, given $\alpha$ (along with $\alpha'$ and $\alpha''$), we
  take $\hat\alpha$ to be the common value $d_1\alpha'=d_1\alpha''$.
\end{proof}

Thus the $2$-cells in the $2$-category $\underline{\Cat_T(\ce)}$ are the straightforward analogue of internal natural transformations, and coincide with the ones considered in \cite[IV.3]{Burroni-Multicategories} and \cite[Section~5.3]{Leinster-book} (under the additional assumptions that the category $\ce$ has pullbacks and the monad $T$ is cartesian).

\section{Powers by $\Delta[1]$}\label{sect:powers}
Throughout this section, let $\ce$ be a locally small category with pullbacks and $T$ an arbitrary monad on $\ce$. 
Recall that we write $\Delta[n]$ for the 
the representable $\Delta(-,n)\colon \Delta\op\to\Set$, and $\Dtop[n]$ for $\Dtop(-,n)\colon\Dtopop\to\Set$.
In this section, we show that the $s\Set$-category $\underline{s_T\ce}$ of $T$-simplicial objects has powers by $\Delta[1]\in s\Set$. 
We then observe that the full sub-$s\Set$-category $\underline{\Cat_T(\ce)}$ of $\underline{s_T\ce}$ consisting of all $T$-categories is closed under powers by $\Delta[1]$; it follows that the $2$-category $\underline{\Cat_T(\ce)}$ has powers by $\mathbf{2}$.

\subsection{Powers by $\Dtop[1]$ in $\underline{[\Dtopop,\ce]}$}\label{subsec:power-by-Dtop-1}
We first construct powers by $\Dtop[1]\in[\Dtopop,\Set]$ in $\underline{[\Dtopop,\ce]}$, which is a $[\Dtopop,\Set]$-category by Construction~\ref{const:enrichment-by-presheaf-cat}.

\begin{notation}\label{notation:chi-G}
We write $\chi_j$ (or $\chi^m_j$ when we wish to record its domain) for the map $m\to 1$ sending $i$ to $0$ if $i<j$,
and to $1$ otherwise. This defines a map in $\Delta$ for $0\le j\le
m+1$, and in $\Dtop$ if in fact $j\le m$. Every map in $\Delta$ into $1$ has
this form. 
These satisfy $\chi_h\delta_j=\chi_h$ if $h\le j$, and
$\chi_h\delta_j=\chi_{h-1}$ otherwise. In particular, $\chi_{j+1}\delta_j=\chi_j=\chi_j\delta_j$.
We also write $G$ for $\Dtop[1]$.
\end{notation}

As shown in Proposition~\ref{prop:copowers-and-powers-in-BC}(2), the power $G\pitchfork X$ of $X\in [\Dtopop,\ce]$ by $G$ is given by setting $(G\pitchfork X)_n$ to be the weighted limit of $X$ with weight $G\times \Dtop[n]=\Dtop[1]\times\Dtop[n]$.
Now observe that $\Dtop[1]\times\Dtop[n]$ can be expressed as the following colimit of representables in $[\Dtopop,\Set]$.
\begin{equation}\label{eqn:cylinder-decomposition}
\begin{tikzpicture}[baseline=-\the\dimexpr\fontdimen22\textfont2\relax ]
      \node(0)  at (6,2.25) {$\Dtop[1]\times\Dtop[n]$};
      \node(00) at (0,-0.75) {$\Dtop[n+1]$};
      \node(01) at (3,-0.75) {$\Dtop[n+1]$};
      \node(02) at (6,-0.75) {$\cdots$};
      \node(03) at (9,-0.75) {$\Dtop[n+1]$};
      \node(04) at (12,-0.75) {$\Dtop[n+1]$};
      \node(10) at (1.5,-2.25) {$\Dtop[n]$};
      \node(11) at (4.5,-2.25) {$\Dtop[n]$};
      \node(12) at (6,-2.25) {$\cdots$};
      \node(13) at (7.5,-2.25) {$\Dtop[n]$};
      \node(14) at (10.5,-2.25) {$\Dtop[n]$};
      \draw [->] (10) to node[auto,labelsize] {$\delta_n$} (00);   
      \draw [->] (10) to node[swap,auto,labelsize] {$\delta_n$} (01);
      \draw [->] (11) to node[auto,labelsize] {$\delta_{n-1}$} (01);  
      \draw [->] (13) to node[auto,swap,labelsize] {$\delta_{2}$} (03);  
      \draw [->] (14) to node[auto,labelsize] {$\delta_{1}$} (03); 
      \draw [->] (14) to node[auto,swap,labelsize] {$\delta_{1}$} (04);  
      \draw [->] (00) to node[auto,labelsize] {$(\chi_{n+1},\sigma_n)$} (0);
      \draw [->] (01) to node[near start,fill=white,labelsize] {$(\chi_{n},\sigma_{n-1})$} (0);
      \draw [->] (03) to node[near start,fill=white,labelsize] {$(\chi_{2},\sigma_{1})$} (0);
      \draw [->] (04) to node[auto,swap,labelsize] {$(\chi_{1},\sigma_{0})$} (0);
\end{tikzpicture}
\end{equation}
This is an analogue of the well-known fact for simplicial sets (see e.g.\ \cite[II.5.5]{Gabriel-Zisman}), and we omit its straightforward verification by means of a combinatorial argument.

Since the operation of taking weighted limits of $X$ turns colimits of weights (in $[\Dtopop,\Set]$) into limits in $\ce$, and since weighted limits with respect to representable weights are given by evaluations, we see that $(G\pitchfork X)_n$ is the following limit in $\ce$.
\begin{equation}\label{eqn:power-by-G-as-iterated-pullback}
\begin{tikzpicture}[baseline=-\the\dimexpr\fontdimen22\textfont2\relax ]
      \node(0)  at (6,2.25) {$(G\pitchfork X)_n$};
      \node(00) at (0,-0.75) {$X_{n+1}$};
      \node(01) at (3,-0.75) {$X_{n+1}$};
      \node(02) at (6,-0.75) {$\cdots$};
      \node(03) at (9,-0.75) {$X_{n+1}$};
      \node(04) at (12,-0.75) {$X_{n+1}$};
      \node(10) at (1.5,-2.25) {$X_n$};
      \node(11) at (4.5,-2.25) {$X_n$};
      \node(12) at (6,-2.25) {$\cdots$};
      \node(13) at (7.5,-2.25) {$X_n$};
      \node(14) at (10.5,-2.25) {$X_n$};
      \draw [<-] (10) to node[auto,labelsize] {$d_n$} (00);   
      \draw [<-] (10) to node[swap,auto,labelsize] {$d_n$} (01);
      \draw [<-] (11) to node[auto,labelsize] {$d_{n-1}$} (01);  
      \draw [<-] (13) to node[auto,swap,labelsize] {$d_{2}$} (03);  
      \draw [<-] (14) to node[auto,labelsize] {$d_{1}$} (03); 
      \draw [<-] (14) to node[auto,swap,labelsize] {$d_{1}$} (04);  
      \draw [<-] (00) to node[auto,labelsize] {$\pi_n$} (0);
      \draw [<-] (01) to node[auto,near start,labelsize] {$\pi_{n-1}$} (0);
      \draw [<-] (03) to node[auto,swap, near start,labelsize] {$\pi_1$} (0);
      \draw [<-] (04) to node[auto,swap,labelsize] {$\pi_0$} (0);
\end{tikzpicture}
\end{equation}
In particular, this limit exists in $\ce$ because it can be constructed as an iterated pullback in $\ce$. 
We note that the face and degeneracy maps for $G\pitchfork X$ are determined as follows.
\begin{itemize}
    \item For each $n\geq 0$ and $0\leq i<n+1$, $d_i\colon (G\pitchfork X)_{n+1}\to (G\pitchfork X)_n$ is the unique morphism in $\ce$ making the following diagrams commute: 
    \begin{itemize}
        \item 
        \begin{tikzcd}
	{(G\pitchfork X)_{n+1}} & {X_{n+2}} \\
	{(G\pitchfork X)_n} & {X_{n+1}}
	\arrow["{\pi_j}", from=1-1, to=1-2]
	\arrow["{d_i}"', from=1-1, to=2-1]
	\arrow["{d_{i+1}}", from=1-2, to=2-2]
	\arrow["{\pi_j}"', from=2-1, to=2-2]
\end{tikzcd} (for each $0\leq j\leq i-1$)
        \item 
        \begin{tikzcd}
	{(G\pitchfork X)_{n+1}} & {X_{n+2}} \\
	{(G\pitchfork X)_n} & {X_{n+1}}
	\arrow["{\pi_{j+1}}", from=1-1, to=1-2]
	\arrow["{d_i}"', from=1-1, to=2-1]
	\arrow["{d_{i}}", from=1-2, to=2-2]
	\arrow["{\pi_j}"', from=2-1, to=2-2]
\end{tikzcd} (for each $i\leq j\leq n$)
    \end{itemize}
    \item For each $n\geq 1$ and $0\leq i\leq n-1$, $s_i\colon (G\pitchfork X)_{n-1}\to (G\pitchfork X)_n$ is the unique morphism in $\ce$ making the following diagrams commute: 
    \begin{itemize}
        \item 
        \begin{tikzcd}
	{(G\pitchfork X)_{n-1}} & {X_{n}} \\
	{(G\pitchfork X)_n} & {X_{n+1}}
	\arrow["{\pi_j}", from=1-1, to=1-2]
	\arrow["{s_i}"', from=1-1, to=2-1]
	\arrow["{s_{i+1}}", from=1-2, to=2-2]
	\arrow["{\pi_j}"', from=2-1, to=2-2]
\end{tikzcd} (for each $0\leq j\leq i$)
        \item 
        \begin{tikzcd}
	{(G\pitchfork X)_{n-1}} & {X_{n}} \\
	{(G\pitchfork X)_n} & {X_{n+1}}
	\arrow["{\pi_{j-1}}", from=1-1, to=1-2]
	\arrow["{s_i}"', from=1-1, to=2-1]
	\arrow["{s_{i}}", from=1-2, to=2-2]
	\arrow["{\pi_j}"', from=2-1, to=2-2]
\end{tikzcd} (for each $i+1\leq j\leq n$)
    \end{itemize}
\end{itemize}
Note in particular that the family $\bigl(\pi_n\colon (G\pitchfork X)_n\to X_{n+1}\bigr)_{n\geq 0}$ defines a natural transformation $g\colon G\pitchfork X\to XR$. 
We also have a natural transformation $t\colon G\pitchfork X\to X$ induced by $\delta_0\colon \Dtop[0]\to \Dtop[1]=G$; note that $\Dtop[0]\pitchfork X\cong X$.
Explicitly, we have $t=\bigl((G\pitchfork X)_n\xrightarrow{\pi_0}X_{n+1}\xrightarrow{d_0} X_n\bigr)_{n\geq 0}$.

The universal property of $G\pitchfork X$ says that we have 
\[
[\Dtopop,\ce](Y,G\pitchfork X)\cong \underline{[\Dtopop,\ce]}(Y,X)_1
\]
for each $Y\in [\Dtopop,\ce]$.
For future reference, we spell out the details of this bijective correspondence. 

\begin{notation}\label{notation:u-chi}
An element $u$ of $\underline{[\Dtopop,\ce]}(Y,X)_1$ is a family $(u_\phi\colon Y_m\to X_m)_{m\geq 0,\,\phi\in\Dtop(m,1)}$, where every morphism $\phi\colon m\to 1$ in $\Dtop$ is of the form $\chi^m_k$ for some $0\leq k\leq m$ (see Notation~\ref{notation:chi-G}). 
We write $u_{\chi^m_k}\colon Y_m\to X_m$ as $u_{m,k}$, and hence $u$ is a family of the form $(u_{m,k}\colon Y_m\to X_m)_{m\geq 0,\,0\leq k\leq m}$.
\end{notation}

\begin{proposition}\label{prop:power-by-Dtop-1-correspondence-detail}
    Let $X,Y\in [\Dtopop,\ce]$. 
    \begin{itemize}
        \item[(1)] Given $u=(u_{m,k}\colon Y_m\to X_m)_{m\geq 0,\,0\leq k\leq m}\in \underline{[\Dtopop,\ce]}(Y,X)_1$, the corresponding morphism
        $\hat u = \bigl(\hat u_m\colon Y_m\to (G\pitchfork X)_m \bigr)_{m\geq 0} \colon Y\to G\pitchfork X$ in $[\Dtopop,\ce]$ is determined by the commutativity of
    \[
        \begin{tikzcd}
	{Y_m} & {(G\pitchfork X)_m} \\
	{Y_{m+1}} & {X_{m+1}}
	\arrow["{\hat u_m}", from=1-1, to=1-2]
	\arrow["{s_k}"', from=1-1, to=2-1]
	\arrow["{\pi_{k}}", from=1-2, to=2-2]
	\arrow["{u_{m+1,k+1}}"', from=2-1, to=2-2]
        \end{tikzcd}
    \]
        for each $0\leq k\leq m$.
        \item[(2)] Given a morphism $\hat u = \bigl(\hat u_m\colon Y_m\to (G\pitchfork X)_m \bigr)_{m\geq 0} \colon Y\to G\pitchfork X$ in $[\Dtopop,\ce]$, the corresponding element $u=(u_{m,k}\colon Y_m\to X_m)_{m\geq 0,\,0\leq k\leq m}\in \underline{[\Dtopop,\ce]}(Y,X)_1$ is obtained by setting $u_{m,k}$ to be the composite 
        $Y_m\xrightarrow{\hat u_m}(G\pitchfork X)_m\xrightarrow{\pi_k} X_{m+1}\xrightarrow{d_k} X_m$.
    \end{itemize}
\end{proposition}
\begin{proof}
    Both (1) and (2) follow from the commutativity of \eqref{eqn:power-morphism-correspondence} and the correspondence between \eqref{eqn:cylinder-decomposition} and \eqref{eqn:power-by-G-as-iterated-pullback}.
    For (1), let $b= m$, $b'=m+1$, $\beta=\chi^{m+1}_{k+1}$, and $f=\sigma_k$ in \eqref{eqn:power-morphism-correspondence} and observe that the projection $\pi_{b',\beta,f}$ there coincides with $\pi_k\colon (G\pitchfork X)_m\to X_{m+1}$, as $(\beta,f)=(\chi^{m+1}_{k+1},\sigma_k)$ in \eqref{eqn:cylinder-decomposition} corresponds to $\pi_k$ in \eqref{eqn:power-by-G-as-iterated-pullback}. 
    For (2), let $b=b'=m$, $\beta= \chi^m_{k}$, and $f=$ the identity on $m$ in \eqref{eqn:power-morphism-correspondence}, and observe that $\beta$ is the composite $m\xrightarrow{\delta_k}m+1\xrightarrow{\chi^{m+1}_{k+1}}1$ and $f$ is the composite $m\xrightarrow{\delta_k}m+1\xrightarrow{\sigma_k}m$, where the pair $(\chi^{m+1}_{k+1},\sigma_k)$ in \eqref{eqn:cylinder-decomposition} corresponds to $\pi_k$ in \eqref{eqn:power-by-G-as-iterated-pullback}.
\end{proof}

\begin{remark}
    We have essentially the same description of powers by $\Delta[1]$ in an $s\Set$-category of the form $\underline{[\Delta\op,\ce]}$. This is because replacing $\Dtop$ by $\Delta$ in \eqref{eqn:cylinder-decomposition} yields a colimit diagram in $s\Set$. Hence for $G=\Delta[1]$ and $X\colon \Delta\op\to \ce$, the power $G\pitchfork X$ in $\underline{[\Delta\op,\ce]}$ can be constructed by the limit as in \eqref{eqn:power-by-G-as-iterated-pullback}. 
    However, unlike the case of $\underline{[\Dtopop,\ce]}$, the family $\bigl(\pi_n\colon (G\pitchfork X)_n\to X_{n+1} \bigr)_{n\geq 0}$ does not form a natural transformation $G\pitchfork X\to X$ since it does not commute with the last face maps (induced by morphisms of the form $\delta_{n+1}\colon n\to n+1$ in $\Delta$). 
    Instead, we have a natural transformation $G\pitchfork X\to X$ induced by $\delta_1\colon \Delta[0]\to \Delta[1]$ and given concretely by 
    $\bigl((G\pitchfork X)_n \xrightarrow{\pi_n} X_{n+1}\xrightarrow{d_{n+1}} X_n \bigr)_{n\geq 0}$, which in turn is not available over $\Dtop$.
\end{remark}

\subsection{Simplices in hom simplicial sets}
\begin{notation}
    Recall from the discussion just after Proposition~\ref{prop:T-simp-via-pullback} that a $T$-simplicial object can be identified with a monad opmorphism from $(\Dtopop,R)$ to $(\ce,T)$.
    We regard a $T$-simplicial object $\ts X$ as a pair $(X,\xi)$ consisting of $X\colon \Dtopop\to \ce$ and $\xi\colon XR\to TX$ (see \eqref{eqn:opmorphism}) satisfying the suitable properties.
\end{notation}

Let $\ts X$ be a $T$-simplicial object. 
The power $\ts L$ of $\ts X$ by $\Delta[1]$ in $\underline{s_T\ce}$, if exists, is characterized by the existence of a family of isomorphisms of simplicial sets
\begin{equation}
\label{eqn:natural-iso-homs}
\underline{s_T\ce}(\ts Y,\ts L)\cong \underline{s\Set}\bigl(\Delta[1],\underline{s_T\ce}(\ts Y,\ts X)\bigr),\quad \text{$s\Set$-natural in $\ts Y\in \underline{s_T\ce}$.}
\end{equation}
The family of bijections between the sets of $0$-simplices induced by \eqref{eqn:natural-iso-homs} is
\begin{equation}
\label{eqn:natural-iso-homs-0-simp}
\underline{s_T\ce}(\ts Y,\ts L)_0\cong \underline{s_T\ce}(\ts Y,\ts X)_1,\quad \text{natural in $\ts Y\in {s_T\ce}$,}
\end{equation}
thanks to the Yoneda lemma. 
Since we have $\underline{s_T\ce}(\ts Y,\ts L)_0\cong s_T\ce(\ts Y,\ts L)$, \eqref{eqn:natural-iso-homs-0-simp} determines the $T$-simplicial object $\ts L$ up to isomorphism. 

Thus we first aim to construct a $T$-simplicial object $\ts L$ with \eqref{eqn:natural-iso-homs-0-simp}, i.e., a representation of the presheaf $\underline{s_T\ce}(-,\ts X)_1\colon (s_T\ce)\op\to\Set$.
It is not difficult to see that \eqref{eqn:natural-iso-homs-0-simp} in fact implies the seemingly stronger condition \eqref{eqn:natural-iso-homs} (see Proposition~\ref{prop:strengthening-univ-prop}).
We first analyze the right-hand side $\underline{s_T\ce}(\ts Y,\ts X)_1$ of \eqref{eqn:natural-iso-homs-0-simp}.
In the following proposition, $\underline{[\Dtopop,\ce]}$ denotes the $[\Dtopop,\Set]$-category defined in Construction~\ref{const:enrichment-by-presheaf-cat}.

\begin{proposition}\label{prop:n-simplex-inductively}
    Let $\ts X=(X,\xi^{\ts X})$ and $\ts Y=(Y,\xi^{\ts Y})$ be $T$-simplicial objects and $n> 0$. To give an element of $\underline{s_T\ce}(\ts Y,\ts X)_n$ is equivalent to giving $u\in \underline{[\Dtopop,\ce]}(Y,X)_n$ and $v\in \underline{s_T\ce}(\ts Y,\ts X)_{n-1}$ such that 
    \begin{itemize}
        \item[(a)] for each $m> 0$ and $\phi\colon m\to n$ in $\Dtop$ such that $\phi\delta_m\colon m-1\to n$ is a morphism in $\Dtop$, the diagram 
\[
  \begin{tikzcd}
    Y_{m} \rar["u_{\phi}"] \dar["d_{m}"'] & X_{m}
    \dar["d_{m}"] \\
    TY_{m-1} \rar["Tu_{\phi\delta_m}"] & TX_{m-1}
  \end{tikzcd}
\]
        commutes, and 
        \item[(b)] for each $m>0$ and $\phi\colon m\to n$ in $\Dtop$ such that $\phi\delta_m\colon m-1\to n$ is not a morphism in $\Dtop$ (equivalently, such that there exists a necessarily unique morphism $\tilde \phi\colon m-1\to n-1$ in $\Delta$ with $\phi\delta_m=\delta_n\tilde\phi$), the diagram 
\[
  \begin{tikzcd}
    Y_{m} \rar["u_{\phi}"] \dar["d_{m}"'] & X_{m}
    \dar["d_{m}"] \\
    TY_{m-1} \rar["Tv_{\tilde\phi}"] & TX_{m-1}
  \end{tikzcd}
\]
        commutes.
    \end{itemize}
\end{proposition}
\begin{proof}
    Since a morphism $\phi\colon m\to n$ in $\Delta$ (with $m\geq 0$) is either in $\Dtop$ or of the form $\delta_n\phi'$ for a unique morphism $\phi'\colon m\to n-1$ in $\Delta$, the function 
    \[
    \Dtop(m,n)+\Delta(m,n-1)\xrightarrow{[\text{inclusion},\,\delta_n(-)]}\Delta(m,n)
    \]
    is bijective (cf.\ Remark~\ref{rmk:RanFR}). 
    Therefore to give a family $x=(x_\phi\colon Y_m\to X_m)_{m\geq 0, \phi\in \Delta(m,n)}$ is equivalent to giving a pair of families 
    $u=(u_\phi\colon Y_m\to X_m)_{m\geq 0, \phi\in \Dtop(m,n)}$ and $v=(v_{\phi'}\colon Y_m\to X_m)_{m\geq 0, \phi'\in \Delta(m,n-1)}$, these being related by $u_\phi=x_\phi$ and $v_{\phi'}=x_{\delta_n\phi'}$.
    We show that, under this correspondence, $x$ is an element of $\underline{s_T\ce}(\ts Y,\ts X)_n$ if and only if 
    $u$ is an element of $\underline{[\Dtopop,\ce]}(Y,X)_n$, $v$ is an element of $\underline{s_T\ce}(\ts Y,\ts X)_{n-1}$, and $u$ and $v$ satisfy (a) and (b).

    The condition for $x$ to be an element of $\underline{s_T\ce}(\ts Y,\ts X)_{n}$ is 
    \begin{enumerate}
        \item for each $\ell,m\geq 0$, $\psi\colon \ell\to m$ in $\Dtop$, and $\phi\colon m\to n$ in $\Delta$, we have $\psi^\ast.x_\phi=x_{\phi\psi}.\psi^\ast$, and
        \item for each $m> 0$ and $\phi\colon m\to n$ in $\Delta$, we have $d_{m}.x_\phi=Tx_{\phi\delta_{m}}.d_{m}$.
    \end{enumerate}
    We can decompose the conditions (1) and (2) as follows:
    \begin{itemize}
        \item[(1a):] (1) holds for all $\phi$ in $\Dtop$,
        \item[(1b):] (1) holds for all $\phi$ not in $\Dtop$,
        \item[(2a):] (2) holds for all $\phi$ in $\Dtop$ such that  $\phi\delta_m$ is in $\Dtop$,
        \item[(2b):] (2) holds for all $\phi$ in $\Dtop$ such that $\phi\delta_m$ is not in $\Dtop$, and
        \item[(2c):] (2) holds for all $\phi$ not in $\Dtop$.
    \end{itemize}
    The condition (1a) is equivalent to the condition that $u$ is in $\underline{[\Dtopop,\ce]}(Y,X)_{n}$.
    The conjunction of (1b) and (2c) is equivalent to the condition that $v$ is in $\underline{s_T\ce}(\ts Y,\ts X)_{n-1}$. Finally, (2a) is equivalent to (a) and (2b) is equivalent to (b).
\end{proof}

We write the right adjoint of the inclusion $F_R\colon \Dtopop\to \Delta\op$ as $U_R\colon \Delta\op\to \Dtopop$. Explicitly, we have $U_R [n]= [n+1]$ on objects and, for each $\phi\colon [n]\to [m]$ in $\Delta$, the morphism $U_R\phi\colon [m+1]\to [n+1]$ in $\Dtop$ maps each $k\in [m+1]$ with $0\leq k\leq m$ to $\phi(k)$ and $m+1$ to $n+1$. 
(Thus we have $U_R(n\xrightarrow{\delta_i}n+1)=n+1\xrightarrow{\delta_i}n+2$ and $U_R(n\xrightarrow{\sigma_i}n-1)=n+1\xrightarrow{\sigma_i}n$.)
Also note that the endofunctor $R=(-)+1$ on $\Dtopop$ coincides with the composite $\Dtopop\xrightarrow{F_R}\Delta\op\xrightarrow{U_R}\Dtopop$. 
Observe that for each $n>0$ and $m\geq 0$, the function  
\begin{equation}\label{eqn:Dtop-m+1-n-decomposition}
\Dtop(m,n)+\Delta(m,n-1)\xrightarrow{[\sigma_n.R(-),\ U_R(-)]}\Dtop(m+1,n)
\end{equation}
is bijective. This is because 
each morphism $\psi\colon [m+1]\to [n]$ in $\Dtop$ satisfies either $\psi(m)=n$ or $\psi(m)\leq n-1$; in the former case $\psi$ is in the image of $\Dtop(m,n)\xrightarrow{\sigma_n.R(-)}\Dtop(m+1,n)$, whereas in the latter case it is in the image of $\Delta(m,n-1)\xrightarrow{U_R(-)}\Dtop(m+1,n)$. 
The coproduct decomposition \eqref{eqn:Dtop-m+1-n-decomposition} corresponds to the cases (a) and (b) below.

We use Notation~\ref{notation:u-chi} in the following.

\begin{corollary}\label{cor:1-simplex-in-hom-simplicial-set}
    Let $\ts X=(X,\xi^{\ts X})$ and $\ts Y=(Y,\xi^{\ts Y})$ be $T$-simplicial objects. To give an element of $\underline{s_T\ce}(\ts Y,\ts X)_1$ is equivalent to giving $u\in \underline{[\Dtopop,\ce]}(Y,X)_1$ and a morphism $v\colon \ts Y\to \ts X$ in $s_T\ce$ such that 
    \begin{itemize}
        \item[(a)] for each $m\geq 0$ and $0\leq k \leq m$, 
        the diagram 
\[
  \begin{tikzcd}
    Y_{m+1} \rar["u_{{m+1},k}"] \dar["d_{m+1}"'] & X_{m+1}
    \dar["d_{m+1}"] \\
    TY_{m} \rar["Tu_{m,k}"] & TX_{m}
  \end{tikzcd}
\]
        commutes, and 
        \item[(b)] for each $m\geq 0$, the diagram 
\[
  \begin{tikzcd}
    Y_{m+1} \rar["u_{{m+1},{m+1}}"] \dar["d_{m+1}"'] & X_{m+1}
    \dar["d_{m+1}"] \\
    TY_{m} \rar["Tv_{m}"] & TX_{m}
  \end{tikzcd}
\]
        commutes.
    \end{itemize}
\end{corollary}

\begin{corollary}\label{cor:1-simplex-as-diagram-in-Dtopop-E}
    Let $\ts X=(X,\xi^{\ts X})$ and $\ts Y=(Y, \xi^{\ts Y})$ be $T$-simplicial objects. To give an element of $\underline{s_T\ce}(\ts Y,\ts X)_1$ is equivalent to giving a pair $\hat u\colon Y\to G\pitchfork X$ and $v \colon Y\to X$ of morphisms in $[\Dtopop,\ce]$ making the following diagrams (in $[\Dtopop,\ce]$) commute:
\begin{equation}\label{eqn:1-simplex-as-diagram-in-Dtopop-E}
  \begin{tikzcd}[column sep=0.5cm]
    YR \rar["vR"] \dar["\xi^{\ts Y}"'] & XR
    \dar["\xi^{\ts X}"] \\
    TY \rar["Tv"] & TX
  \end{tikzcd}\quad
  \begin{tikzcd}[column sep=0.5cm]
    YR \dar["\xi^{\ts Y}"'] \rar["\hat uR"] & (G\pitchfork X)R \rar["\text{can}"] &
    G\pitchfork XR \dar["G\pitchfork\xi^{\ts X}"] \\
    TY \rar["T\hat u"] & T(G\pitchfork X) \rar["\text{can}"] & G\pitchfork TX
  \end{tikzcd}\quad
  \begin{tikzcd}[column sep=0.5cm]
    Y \rar["\hat u"] \dar["v"'] & G\pitchfork X \rar["g"] & XR \dar["\xi^{\ts X}"] \\
    X \ar[rr,"iX"] && TX 
  \end{tikzcd}
\end{equation}
    where the morphisms labelled by ``can'' are the canonical comparison morphisms and the morphism $g$ is defined in Subsection~\ref{subsec:power-by-Dtop-1}.
\end{corollary}
\begin{proof}
    Clearly a morphism $v$ making the leftmost diagram in \eqref{eqn:1-simplex-as-diagram-in-Dtopop-E} commute is a morphism $v\colon \ts Y\to \ts X$ in $s_T\ce$. Also, by the universal property of $G\pitchfork X$, the morphism $\hat u$ corresponds to an element $u\in \underline{[\Dtopop,\ce]}(Y,X)_1$, as spelled out in Proposition~\ref{prop:power-by-Dtop-1-correspondence-detail}.
    Thus it suffices to show that the commutativity of the middle diagram in \eqref{eqn:1-simplex-as-diagram-in-Dtopop-E} is equivalent to condition (a) of Corollary~\ref{cor:1-simplex-in-hom-simplicial-set}, and similarly, the commutativity of the rightmost diagram in \eqref{eqn:1-simplex-as-diagram-in-Dtopop-E} is equivalent to condition (b) of Corollary~\ref{cor:1-simplex-in-hom-simplicial-set}.

    The commutativity of the middle diagram in \eqref{eqn:1-simplex-as-diagram-in-Dtopop-E} asserts the equality of two morphisms of type $YR\to G\pitchfork TX$, or equivalently, the equality of two elements in $\underline{[\Dtopop,\ce]}(YR, TX)_1$. 
    By Proposition~\ref{prop:power-by-Dtop-1-correspondence-detail}(2), this is equivalent to saying that for any $m\geq 0$ and $0\leq k\leq m$, the composites 
    \begin{equation}\label{eqn:composite-1}
    YR_m\xrightarrow{\hat uR_m} (G\pitchfork X) R_m\xrightarrow{{can_m}} (G\pitchfork X R)_m\xrightarrow{(G\pitchfork \xi^{\ts X})_m} (G\pitchfork TX)_m\xrightarrow{\pi_k} TX_{m+1}\xrightarrow{Td_k}TX_m
    \end{equation}
    and 
    \begin{equation}\label{eqn:composite-2}
    YR_m\xrightarrow{\xi^{\ts Y}_m} TY_m \xrightarrow{T\hat u_m} T(G\pitchfork X)_m\xrightarrow{{can_m}} (G\pitchfork TX)_m\xrightarrow{\pi_k} TX_{m+1}\xrightarrow{Td_k}TX_m
    \end{equation}
    coincide.
    Now it is not difficult to see that \eqref{eqn:composite-1} is equal to the composite $Y_{m+1}\xrightarrow{u_{m+1,k}}X_{m+1}\xrightarrow{d_{m+1}} TX_m$, whereas \eqref{eqn:composite-2} is equal to the composite $Y_{m+1}\xrightarrow{d_{m+1}}TY_m\xrightarrow{Tu_{m,k}}TX_m$. 
    Hence the commutativity of the middle diagram in \eqref{eqn:1-simplex-as-diagram-in-Dtopop-E} is equivalent to condition (a) of Corollary~\ref{cor:1-simplex-in-hom-simplicial-set}.

    As for the equivalence of the commutativity of the rightmost diagram in \eqref{eqn:1-simplex-as-diagram-in-Dtopop-E} and condition (b) of Corollary~\ref{cor:1-simplex-in-hom-simplicial-set}, first observe that for each $m\geq 0$, the composite 
    \[
    Y_m\xrightarrow{\hat u_m} (G\pitchfork X)_m\xrightarrow{g_m}XR_m\xrightarrow{\xi^{\ts X}_m}TX_m
    \]
    is equal to the composite
    \[
    Y_m\xrightarrow{s_m} Y_{m+1}\xrightarrow{u_{m+1,m+1}} X_{m+1}\xrightarrow{d_{m+1}}TX_m,
    \]
    whereas the composite 
    \[
    Y_m\xrightarrow{v_m}X_m\xrightarrow{iX_m} TX_m
    \]
    is equal to the composite 
    \[
    Y_m\xrightarrow{s_m} Y_{m+1} \xrightarrow{d_{m+1}}TY_m\xrightarrow{Tv_m} TX_m
    \]
    by the naturality of $i$ and \ref{TSimp-A7}.
    Thus (b) of Corollary~\ref{cor:1-simplex-in-hom-simplicial-set} implies the commutativity of the rightmost diagram in \eqref{eqn:1-simplex-as-diagram-in-Dtopop-E}.
    For the converse, note that since $v$ is a morphism of $T$-simplicial objects, the square 
    \[
  \begin{tikzcd}
    Y_{m+1} \rar["v_{m+1}"] \dar["d_{m+1}"'] & X_{m+1}
    \dar["d_{m+1}"] \\
    TY_{m} \rar["Tv_{m}"] & TX_{m}
  \end{tikzcd}
\]
    commutes. 
    Hence (b) of Corollary~\ref{cor:1-simplex-in-hom-simplicial-set} is equivalent to the commutativity of the exterior of 
\[
  \begin{tikzcd}
    Y_{m+1} \ar[rrr,bend left=15, "u_{m+1,m+1}"] \rar["\hat u_{m+1}"'] \dar["v_{m+1}"'] & (G\pitchfork X)_{m+1}
    \rar["\pi_{m+1}"'] &
    X_{m+2} \rar["d_{m+1}"'] \dar["d_{m+2}"] & X_{m+1}
    \ar[dd,"d_{m+1}"] \\
    X_{m+1} \ar[rr,"iX_{m+1}"] \dar["d_{m+1}"'] && TX_{m+1}
    \dar["Td_{m+1}"] \\
    TX_m \ar[rr,"iTX_m"] \ar[rrr,bend right=15, "1"'] && T^2X_m \rar["mX_m"] & TX_m
  \end{tikzcd}
\]
    which follows from the commutativity of the rightmost diagram in \eqref{eqn:1-simplex-as-diagram-in-Dtopop-E}.
\end{proof}

\subsection{The recursive construction of $\ts L$}

The power $\ts L=(L,\xi^{\ts L})$ of the $T$-simplicial object $\ts X=(X,\xi^{\ts X})$ by $\Delta[1]$ comes equipped with a universal element in $\underline{s_T\ce}(\ts L,\ts X)_1$, which by Corollary~\ref{cor:1-simplex-as-diagram-in-Dtopop-E} 
corresponds to a pair $(q\colon L\to G\pitchfork X, p\colon L\to X)$ of morphisms in $[\Dtopop,\ce]$ making the diagrams 
\begin{equation}\label{eqn:universal-1-simplex-to-X-1}
  \begin{tikzcd}
    LR \rar["pR"] \dar["\xi^{\ts L}"'] & XR
    \dar["\xi^{\ts X}"] \\
    TL \rar["Tp"] & TX
  \end{tikzcd}
\end{equation}
\begin{equation}\label{eqn:universal-1-simplex-to-X-2}
  \begin{tikzcd}
    LR \dar["\xi^{\ts L}"'] \rar["qR"] & (G\pitchfork X)R \rar["can"] &
    G\pitchfork XR \dar["G\pitchfork\xi^{\ts X}"] \\
    T\res L \rar["Tq"] & T(G\pitchfork X) \rar["can"] & G\pitchfork TX
  \end{tikzcd}
\end{equation}
\begin{equation}\label{eqn:universal-1-simplex-to-X-3}
  \begin{tikzcd}
    L \rar["q"] \dar["p"'] & G\pitchfork X \rar["g"] & XR \dar["\xi^{\ts X}"] \\
    X \ar[rr,"iX"] && TX 
  \end{tikzcd}
\end{equation}
in $[\Dtopop,\ce]$
commute. 
The commutativity of \eqref{eqn:universal-1-simplex-to-X-1}--\eqref{eqn:universal-1-simplex-to-X-3} is equivalent to the commutativity of the following two diagrams, in $\ce$ and in $[\Dtopop,\ce]$ respectively:
\begin{equation}\label{eqn:L-recursively-1}
\begin{tikzcd}
    L_0 \rar["q_0"] \dar["p_0"'] & (G\pitchfork X)_0
    \dar["d_1.\pi_0"] \\
    X_0 \rar["iX_0"] & TX_0
\end{tikzcd}
\end{equation}
\begin{equation}\label{eqn:L-recursively-2}
\setlength{\R}{2.5cm}
\begin{tikzpicture}[baseline=-\the\dimexpr\fontdimen22\textfont2\relax ]
  \node (c) at (0,0) {$LR$};
  \node (circ1) at ( 60:\R) {$G\pitchfork T\res X$};
  \node (circ2) at (120:\R) {$TL$};
  \node (circ3) at (180:\R) {$TX$};
  \node (circ4) at (240:\R) {$XR$};
  \node (circ5) at (300:\R) {$TXR$};
  \node (circ6) at (360:\R) {$(G\pitchfork \res X)R$};
  \draw [<-] (circ2) to node[midway,fill=white,labelsize] {$\xi^{\ts L}$} (c);
  \draw [<-] (circ1) to node[auto,swap,labelsize] {$can.Tq$} (circ2);
  \draw [->] (circ2) to node[auto,swap,labelsize] {$Tp$} (circ3);
  \draw [<-] (circ3) to node[auto,swap,labelsize] {$\xi^{\ts X}$} (circ4);
  \draw [->] (circ4) to node[auto,swap,labelsize] {$iXR$} (circ5);
  \draw [<-] (circ5) to node[auto,swap,labelsize] {$\xi^{\ts X}R.gR$} (circ6);
  \draw [->] (circ6) to node[auto,swap,labelsize] {$G\pitchfork\xi^{\ts X}.can$} (circ1);
  \draw [<-] (circ4) to node[midway,fill=white,labelsize] {$pR$} (c);
  \draw [<-] (circ6) to node[auto,swap,labelsize] {$qR$} (c);
\end{tikzpicture}
\end{equation}
Indeed, \eqref{eqn:universal-1-simplex-to-X-1} and \eqref{eqn:universal-1-simplex-to-X-2} correspond to the squares in \eqref{eqn:L-recursively-2} with $T\res X$ and $G\pitchfork T\res X$ as the codomains, respectively; 
\eqref{eqn:universal-1-simplex-to-X-3} in dimension $0$ corresponds to \eqref{eqn:L-recursively-1}; and \eqref{eqn:universal-1-simplex-to-X-3} in positive dimensions corresponds to the square in \eqref{eqn:L-recursively-2} with $T\res XR$ as the codomain.

We will construct the tuple $(\res L,\xi^{\ts L},p,q)$ so that \emph{both diagrams \eqref{eqn:L-recursively-1} and \eqref{eqn:L-recursively-2} become limit diagrams} (in $\ce$ and in $[\Dtopop,\ce]$, respectively). 
The requirement that \eqref{eqn:L-recursively-1} should be a limit diagram simply means that we define $(L_0,p_0,q_0)$ as the pullback of $(G\pitchfork X)_0\xrightarrow{\pi_0} X_1\xrightarrow{d_1} TX_0$ and $iX_0\colon X_0\to TX_0$ in $\ce$ (or equivalently, since $\pi_0\colon (G\pitchfork X)_0\to X_1$ is an isomorphism, the pullback of $d_1\colon X_1\to TX_0$ and $iX_0$).
The requirement that \eqref{eqn:L-recursively-2} should be a limit diagram in $[\Dtopop,\ce]$ is less straightforward, since the tuple $(\res L,\xi^{\ts L},p,q)$ we are constructing appears not only in the limit cone $(\res LR, \xi^{\ts L},pR,qR)$ but also in the diagram over which the limit is taken. 
However, thanks to the dimension shift $R=(-)+1$, this requirement in effect gives a recursive construction of the data $(\res L,\xi^{\ts L},p,q)$. 
Namely, for each $n\geq 0$, we define the tuple $\bigl(L_{n+1},d_{n+1}\colon L_{n+1}\to TL_n, p_{n+1}\colon L_{n+1}\to X_{n+1},q_{n+1}\colon L_{n+1}\to (G\pitchfork \res X)_{n+1}\bigr)$ as the following limit in $\ce$:
\begin{equation}\label{eqn:L-recursively-in-E}
\setlength{\R}{2.5cm}
\begin{tikzpicture}[baseline=-\the\dimexpr\fontdimen22\textfont2\relax ]
  \node (c) at (0,0) {$L_{n+1}$};
  \node (circ1) at ( 60:\R) {$(G\pitchfork T\res X)_n$};
  \node (circ2) at (120:\R) {$TL_n$};
  \node (circ3) at (180:\R) {$TX_n$};
  \node (circ4) at (240:\R) {$X_{n+1}$};
  \node (circ5) at (300:\R) {$TX_{n+1}$};
  \node (circ6) at (360:\R) {$(G\pitchfork \res X)_{n+1}$};
  \draw [<-] (circ2) to node[midway,fill=white,labelsize] {$d_{n+1}$} (c);
  \draw [<-] (circ1) to node[auto,swap,labelsize] {$can_n.Tq_n$} (circ2);
  \draw [->] (circ2) to node[auto,swap,labelsize] {$Tp_n$} (circ3);
  \draw [<-] (circ3) to node[auto,swap,labelsize] {$d_{n+1}$} (circ4);
  \draw [->] (circ4) to node[auto,swap,labelsize] {$iX_{n+1}$} (circ5);
  \draw [<-] (circ5) to node[auto,swap,labelsize] {$d_{n+2}.\pi_{n+1}$} (circ6);
  \draw [->] (circ6) to node[auto,swap,labelsize] {$(G\pitchfork\xi^{\ts X}.can)_n$} (circ1);
  \draw [<-] (circ4) to node[midway,fill=white,labelsize] {$p_{n+1}$} (c);
  \draw [<-] (circ6) to node[auto,swap,labelsize] {$q_{n+1}$} (c);
\end{tikzpicture}
\end{equation}
Of course, we first have to show that this limit indeed exists in $\ce$. (Recall that we only assume the existence of pullbacks in $\ce$.)

To this end, observe that the diagram 
\[
\setlength{\R}{2.5cm}
\begin{tikzpicture}
  \node (circ1) at ( 60:\R) {$(G\pitchfork T\res X)_n$};
  \node (circ2) at (120:\R) {$TL_n$};
  \node (circ3) at (180:\R) {$TX_n$};
  \node (circ4) at (240:\R) {$X_{n+1}$};
  \node (circ5) at (300:\R) {$TX_{n+1}$};
  \node (circ6) at (360:\R) {$(G\pitchfork \res X)_{n+1}$};
  \draw [<-] (circ1) to node[auto,swap,labelsize] {$can_n.Tq_n$} (circ2);
  \draw [->] (circ2) to node[auto,swap,labelsize] {$Tp_n$} (circ3);
  \draw [<-] (circ3) to node[auto,swap,labelsize] {$d_{n+1}$} (circ4);
  \draw [->] (circ4) to node[auto,swap,labelsize] {$iX_{n+1}$} (circ5);
  \draw [<-] (circ5) to node[auto,swap,labelsize] {$d_{n+2}.\pi_{n+1}$} (circ6);
  \draw [->] (circ6) to node[auto,swap,labelsize] {$(G\pitchfork\xi^{\ts X}.can)_n$} (circ1);
\end{tikzpicture}
\]
over which the limit is taken can be extended to the commutative diagram 
\[
\setlength{\R}{2.5cm}
\begin{tikzpicture}
  \node (c) at (0,0) {$TX_n$};
  \node (circ1) at ( 60:\R) {$(G\pitchfork T\res X)_n$};
  \node (circ2) at (120:\R) {$TL_n$};
  \node (circ3) at (180:\R) {$TX_n$};
  \node (circ4) at (240:\R) {$X_{n+1}$};
  \node (circ5) at (300:\R) {$TX_{n+1}$};
  \node (circ6) at (360:\R) {$(G\pitchfork \res X)_{n+1}$};
  \draw [->] (circ1) to node[midway,fill=white,labelsize] {$mX_n.Td_{n+1}.\pi_n$} (c);
  \draw [<-] (circ1) to node[auto,swap,labelsize] {$can_n.Tq_n$} (circ2);
  \draw [->] (circ2) to node[auto,swap,labelsize] {$Tp_n$} (circ3);
  \draw [<-] (circ3) to node[auto,swap,labelsize] {$d_{n+1}$} (circ4);
  \draw [->] (circ4) to node[auto,swap,labelsize] {$iX_{n+1}$} (circ5);
  \draw [<-] (circ5) to node[auto,swap,labelsize] {$d_{n+2}.\pi_{n+1}$} (circ6);
  \draw [->] (circ6) to node[auto,swap,labelsize] {$(G\pitchfork\xi^{\ts X}.can)_n$} (circ1);
  \draw [->] (circ3) to node[midway,fill=white,labelsize] {$1$} (c);
  \draw [->] (circ5) to node[midway,fill=white,labelsize] {$mX_n.Td_{n+1}$} (c);
\end{tikzpicture}
\]
in $\ce$.
Since the inclusion functor
\[
\setlength{\R}{1cm}
\left(\begin{tikzpicture}[baseline=-\the\dimexpr\fontdimen22\textfont2\relax ]
  \node (circ1) at ( 60:\R) {$\bullet$};
  \node (circ2) at (120:\R) {$\bullet$};
  \node (circ3) at (180:\R) {$\bullet$};
  \node (circ4) at (240:\R) {$\bullet$};
  \node (circ5) at (300:\R) {$\bullet$};
  \node (circ6) at (360:\R) {$\bullet$};
  \draw [<-] (circ1) to (circ2);
  \draw [->] (circ2) to (circ3);
  \draw [<-] (circ3) to (circ4);
  \draw [->] (circ4) to (circ5);
  \draw [<-] (circ5) to (circ6);
  \draw [->] (circ6) to (circ1);
\end{tikzpicture}\right)
\quad\longrightarrow\quad
\left(\begin{tikzpicture}[baseline=-\the\dimexpr\fontdimen22\textfont2\relax ]
  \node (c) at (0,0) {$\bullet$};
  \node (circ1) at ( 60:\R) {$\bullet$};
  \node (circ2) at (120:\R) {$\bullet$};
  \node (circ3) at (180:\R) {$\bullet$};
  \node (circ4) at (240:\R) {$\bullet$};
  \node (circ5) at (300:\R) {$\bullet$};
  \node (circ6) at (360:\R) {$\bullet$};
  \draw [->] (circ1) to (c);
  \draw [<-] (circ1) to (circ2);
  \draw [->] (circ2) to (circ3);
  \draw [<-] (circ3) to (circ4);
  \draw [->] (circ4) to (circ5);
  \draw [<-] (circ5) to (circ6);
  \draw [->] (circ6) to (circ1);
  \draw [->] (circ3) to (c);
  \draw [->] (circ5) to (c);
\end{tikzpicture}\right)
\]
is initial, and since its codomain is simply-connected in the sense of \cite{Pare-simply-connected}, the limit \eqref{eqn:L-recursively-in-E} exists in $\ce$ by \cite[Theorem~2]{Pare-simply-connected}.
\begin{remark}
    Concretely, given any commutative diagram 
    \[
\setlength{\R}{1.8cm}
\begin{tikzpicture}
  \node (c) at (0,0) {$G$};
  \node (circ1) at ( 60:\R) {$A$};
  \node (circ2) at (120:\R) {$B$};
  \node (circ3) at (180:\R) {$C$};
  \node (circ4) at (240:\R) {$D$};
  \node (circ5) at (300:\R) {$E$};
  \node (circ6) at (360:\R) {$F$};
  \draw [->] (circ1) to node[midway,fill=white,labelsize] {$g$} (c);
  \draw [<-] (circ1) to node[auto,swap,labelsize] {$a$} (circ2);
  \draw [->] (circ2) to node[auto,swap,labelsize] {$b$} (circ3);
  \draw [<-] (circ3) to node[auto,swap,labelsize] {$c$} (circ4);
  \draw [->] (circ4) to node[auto,swap,labelsize] {$d$} (circ5);
  \draw [<-] (circ5) to node[auto,swap,labelsize] {$e$} (circ6);
  \draw [->] (circ6) to node[auto,swap,labelsize] {$f$} (circ1);
  \draw [->] (circ3) to node[midway,fill=white,labelsize] {$h$} (c);
  \draw [->] (circ5) to node[midway,fill=white,labelsize] {$i$} (c);
\end{tikzpicture}
    \]
    in a category with pullbacks, its limit can be constructed as follows. 
    Let $U,V,W$ be the following pullbacks.
    \[
    \begin{tikzcd}
    U \rar["u_1"] \dar["u_2"'] & B
    \dar["a"] \\
    F \rar["f"] & A
    \end{tikzcd}\quad
    \begin{tikzcd}
    V \rar["v_1"] \dar["v_2"'] & B
    \dar["b"] \\
    D \rar["c"] & C
    \end{tikzcd}\quad
    \begin{tikzcd}
    W \rar["w_1"] \dar["w_2"'] & B
    \dar["ga=hb"] \\
    E \rar["i"] & G
    \end{tikzcd}
    \]
    Let $x\colon U\to W$ be the unique morphism with $w_1x=u_1$ and $w_2x=eu_2$, and similarly, let $y\colon V\to W$ be the unique morphism with $w_1y=v_1$ and $w_2y=dv_2$. 
    Then the pullback of $x$ and $y$ (together with the evident projections) is the limit of the original diagram. 
\end{remark}

Now, there exists a unique object $L\in [\Dtopop,\ce]$ extending the sequence $(L_n)_{n\geq 0}$ of objects in $\ce$, such that 
\begin{itemize}
    \item $\xi^{\ts L}=(d_{n+1}\colon L_{n+1}\to TL_n)_{n\geq 0}\colon LR\to TL$, $p=(p_n\colon L_n\to X_n)_{n\geq 0}\colon L\to X$, and $q=\bigl(q_n\colon L_n\to (G\pitchfork X)_n\bigr)_{n\geq 0}\colon L\to G\pitchfork X$ all become morphisms in $[\Dtopop,\ce]$, and 
    \item $(L,\xi^{\ts L})$ becomes a $T$-simplicial object.
\end{itemize}
Indeed, the above conditions are just enough to specify how we should define the (non-last) face and degeneracy maps of $L$. 

Thus we obtain a $T$-simplicial object $\ts L=(L,\xi^{\ts L})$.
Thanks to the commutativity of \eqref{eqn:universal-1-simplex-to-X-1}--\eqref{eqn:universal-1-simplex-to-X-3}, the pair $(p,q)$ defines an element of $\underline{s_T\ce}(\ts L,\ts X)_1$.
It is now straightforward to show that this has the required universal property, and hence gives a representation of the presheaf $\underline{s_T\ce}(-,\ts X)_1\colon (s_T\ce)\op\to \Set$.

\subsection{The universal property}
\begin{proposition}\label{prop:strengthening-univ-prop}
    Let $\ts X$ be a $T$-simplicial object. If a $T$-simplicial object $\ts L$ admits a family of bijections \eqref{eqn:natural-iso-homs-0-simp} naturally in $\ts Y\in s_T\ce$, then $\ts L$ is the power of $\ts X$ by $\Delta[1]$. 
\end{proposition}
\begin{proof}
  If $\ce$ has finite copowers
  then $\underline{s_T\ce}$ has copowers by the $\Delta[k]$ for each $k\geq 0$ by Corollary~\ref{cor:finite-copower-T-Simp}, which gives the
  general universal property. If not, then we can embed $\ce$ in some
  $\ce'$ which does have finite copowers (and extend $T$), in such a way that existing
  finite limits are preserved. 
\end{proof}

Our discussion so far proves the following. 
\begin{theorem}\label{thm:power-by-Delta-1}
    If $\ce$ has pullbacks and $T$ is an arbitrary monad on $\ce$, then the $s\Set$-category $\underline{s_T\ce}$ of $T$-simplicial objects has powers by $\Delta[1]$.
\end{theorem}

  \subsection{$T$-categories}

    Recall from Definition~\ref{def:T-Cat-official} that we regard $T$-categories as $T$-simplicial objects satisfying the nerve condition.

  \begin{proposition}\label{prop:T-cat-power}
    If $\ts X$ is a $T$-category then so is its power $\ts L$ by $\Delta[1]$.
  \end{proposition}

  \begin{proof}
    Let $g\colon A\to TL_{n+1}$ be given. What is needed in order to
    give a map $f\colon A\to L_{n+1}$ with $d_{n+1}f=Td_0.g$?

    We should give a map $f'\colon A\to X_{n+1}$ and a map $f_j\colon
    A\to X_{n+2}$ for $0\le j\le n+1$, making the diagrams
    \[
      \begin{tikzcd}
        A \ar[rr,"f' "] \dar["g"] && X_{n+1} \dar["d_{n+1}"] \\
        TL_{n+1} \rar["Td_0"] & TL_n \rar["Tp_n"] & TX_n
      \end{tikzcd}\quad
      \begin{tikzcd}
        A \rar["f_{n+1}"] \dar["f' "] & X_{n+2} \dar["d_{n+2}"] \\
        X_{n+1} \rar["iX_{n+1}"] & TX_{n+1}
      \end{tikzcd}\quad
      \begin{tikzcd}
        A \rar["f_j"] \dar["f_{j+1}"] & X_{n+2} \dar["d_{j+1}"] \\
        X_{n+2} \rar["d_{j+1}"] & X_{n+1}
      \end{tikzcd}
    \]
    \[
      \begin{tikzcd}
        A \ar[rrr,"f_j"] \dar["g"] &&& X_{n+2} \ar[d,"d_{n+2}"] \\
        TL_{n+1} \rar["Tq_{n+1}"] &
        T(G\pitchfork X)_{n+1} \rar["T\pi_{j+1}"] & TX_{n+2}
        \rar["Td_0"] & TX_{n+1} 
      \end{tikzcd}
    \]
    commute for $0\le j\le n$.

    By the first of these, and the fact that
    $Tp_n.Td_0=Td_0.Tp_{n+1}$ there is a unique $h'\colon A\to
    X_{n+2}$ making the diagram
    \[
      \begin{tikzcd}
        & A \dar["h' "] \ar[dl,"f' " '] \rar["g"] & TL_{n+1}
        \dar["Tp_{n+1}"] \\
        X_{n+1} & X_{n+2} \ar[r,"d_{n+2}"] \lar["d_0"'] & TX_{n+1}
      \end{tikzcd}
  \]
  commute.

  Then by the second, and the fact that
$iX_{n+1}.f'=iX_{n+1}.d_0h'=Td_0.iX_{n+2}.h'$, there is a unique
$h_{n+2}\colon A\to X_{n+3}$ making the diagram
\[
  \begin{tikzcd}
    & A \dar["h_{n+2}"] \ar[dl,"f_{n+1}" '] \rar["h' "] & X_{n+2}
    \dar["iX_{n+2}"] \\
    X_{n+2} & X_{n+3} \lar["d_0"'] \ar[r,"d_{n+3}"] & TX_{n+2}
  \end{tikzcd}
\]
commute.

By the last, for $1\le j\le n+1$ there is a unique $h_j\colon
A\to X_{n+3}$ making the diagram
\[
  \begin{tikzcd}
    & A \dar["h_j"] \ar[dl,"f_{j-1}" '] \rar["g"] & TL_{n+1}
    \rar["Tq_{n+1}"] & T(G\pitchfork X)_{n+1} \dar["T\pi_j"] \\
    X_{n+2} & X_{n+3} \lar["d_0"'] \ar[rr,"d_{n+3}"] && TX_{n+2}
  \end{tikzcd}
\]
commute. 
By the commutativity of the above diagram with $j=1$,
there is a unique $h_0\colon
A\to X_{n+3}$ making the diagram
\[
  \begin{tikzcd}
    X_{n+3} \dar["d_1"] & A \dar["h_0"] \ar[l,"h_{1}" '] \rar["g"] & TL_{n+1}
    \rar["Tq_{n+1}"] & T(G\pitchfork X)_{n+1} \dar["T\pi_0"] \\
    X_{n+2} & X_{n+3} \lar["d_1"'] \ar[rr,"d_{n+3}"] && TX_{n+2}
  \end{tikzcd}
\]
commute.
By the remaining cases of the third, we have $d_{j+1}h_j=d_{j+1}h_{j+1}$ for $1\le
j\le n+1$.

There is now a unique $h\colon A\to L_{n+2}$ with $d_{n+2}h=g$, with
$p_{n+2}h=h'$, and with $\pi_jq_{n+2}h=h_j$.
  \end{proof}

  This now proves:

  \begin{theorem}\label{thm:T-cat-power-by-2}
    If $\ce$ has pullbacks then the $2$-category $\underline{\Cat_T(\ce)}$ of
    $T$-categories in $\ce$ has powers by $\two$ for any monad $T$ on $\ce$.
  \end{theorem}

  If we suppose that $\ce$ has not just pullbacks but finite limits, then we can do better:

  \begin{theorem}
  If $\ce$ is a category with finite limits and $T$ an arbitrary monad on $\ce$, then the $2$-category $\underline{\Cat_T(\ce)}$ has finite limits.
  \end{theorem}

  \begin{proof}
  If $\ce$ has pullbacks then $\underline{\Cat_T(\ce)}$ has pullbacks and powers by $\two$. If $\ce$ has a terminal object then so does $\underline{\Cat_T(\ce)}$. Thus $\underline{\Cat_T(\ce)}$ has finite conical limits and powers by $\two$ and so has all (2-categorical) finite limits.
  \end{proof}

  \begin{theorem}
  If $\ce$ has pullbacks and finite copowers, and $\Cat_T(\ce)$ is reflective in
  $s_T\ce$ (as a category) then $\underline{\Cat_T(\ce)}$ has copowers by $\two$.
  \end{theorem}

  \begin{proof}
   Write $N\colon\Cat_T(\ce)\to s_T\ce$ for the inclusion functor and $L$ for its left adjoint. We claim that $L(\Delta[1]\cdot N\ts X)$ gives the copower by $\two$ of the $T$-category $\ts X$. This will be the case provided that the canonical functor
   \[ \underline{\Cat_T(\ce)}(L(\Delta[1]\cdot N\ts X),\ts Y)\rightarrow [\two,\underline{\Cat_T(\ce)}(\ts X,\ts Y)] \]
   is invertible for each $T$-category $\ts Y$. This in turn will be the case provided that the induced function 
   \[ \Cat_0(\two,\underline{\Cat_T(\ce)}(L(\Delta[1]\cdot N\ts X),\ts Y)\rightarrow  \Cat_0(\two\times\two,\underline{\Cat_T(\ce)}(\ts X,\ts Y))    \]
   is bijective. Since $\ce$ has pullbacks, $\Cat_T(\ce)$ has powers by $\two$ by Theorem~\ref{thm:T-cat-power-by-2}, and we have 
   \begin{align*}
       \Cat_0(\two,\underline{\Cat_T(\ce)}(L(\Delta[1]\cdot N\ts X),\ts Y)) 
       &\cong \Cat_T(\ce)(L(\Delta[1]\cdot N\ts X),\two\pitchfork \ts Y) \\
       &\cong s_T(\ce)(\Delta[1]\cdot N\ts X,N(\two\pitchfork \ts Y)) \\
       &\cong s\Set(\Delta[1],\underline{s_T(\ce)}(N\ts X,N(\two\pitchfork \ts Y)))\\
       &\cong \Cat_0(\two,\underline{\Cat_T(\ce)}(\ts X,\two\pitchfork \ts Y)) \\
       &\cong \Cat_0(\two\x\two,\underline{\Cat_T}(\ce)(\ts X,\ts Y))
   \end{align*}
   as required.
  \end{proof}

\section{Local presentability}\label{sec:local-presentability}

Now suppose that $\ce$ is locally finitely presentable and that $T$ is finitary.
We have already seen in Theorem~\ref{thm:set-lfp} that the ordinary category $s_T\ce$ is locally finitely presentable; we now show that the $s\Set$-enriched category $\underline{s_T\ce}$ is also locally finitely presentable.

We know by Proposition~\ref{prop:copower-in-T-Simp} that $\underline{s_T\ce}$ has copowers, and by Theorem~\ref{thm:power-by-Delta-1} that it has powers by $\Delta[1]$. 

We shall show that it has powers by $\Delta[n]$ for all $n$; since any simplicial set is a colimit of these, and $s_T\ce$ has conical limits, it will then follow that $\underline{s_T\ce}$ has all powers, and so is complete and cocomplete. We then deduce that it is locally finitely presentable as an enriched category. First, however, we give a new description of the $n$-simplices of the hom-objects of $\underline{s_T\ce}$.

We write $U$ for the forgetful functor $s_T\ce\to[\Dtopop,\ce]$ as well as for $[\Delta\op,\Set]\to[\Dtopop,\Set]$.

\subsection{The $n$-simplices of $\underline{s_T\ce}$ revisited}

First we need some preliminary constructions. Observe that
the functors $T_{!},R^{*}\colon [\Dtopop,\ce]\to[\Dtopop,\ce]$ given
by postcomposing with $T$ and precomposing with $R$ are both enriched
over $[\Dtopop,\Set]$, and so there are canonical comparisons
expressing the extent to which they preserve powers by objects of
$[\Dtopop,\Set]$.

\begin{proposition}\label{prop:Lambda}
  There is a map $\Lambda\colon U\Delta[n-1]\to R^{*}\Dtop[n]$ which
in degree $m$ sends $\phi\colon m\to n-1$ in $\Delta$ to its
top-preserving extension $m+1\to n$.
\end{proposition}

\begin{proposition}\label{prop:Pi}
  The forgetful $U\colon s_T\ce\to[\Dtopop,\ce]$ is enriched over
  $[\Dtopop,\Set]$, and so there are induced maps
  $\Pi\colon U(M\pitchfork \ts X)\to UM\pitchfork X$ for $M\in[\Delta^\mathrm{op},\Set]$
  and $\ts X\in s_T\ce$, whenever these powers exist.
\end{proposition}

\begin{proof}
  The unit $\eta\colon M\to \underline{s_T\ce}(M\pitchfork\ts X,\ts X)$ induces
  \[
    \begin{tikzcd}
      UM \rar["U\eta"] & U\underline{s_T\ce}(M\pitchfork\ts X,\ts X)
      \rar &  \underline{[\Dtopop,\ce]}(U(M\pitchfork\ts X),U\ts X)
    \end{tikzcd} \]
  which determines $U(M\pitchfork\ts X)\to UM\pitchfork U\ts
  X=UM\pitchfork X$ by the universal property.
\end{proof}

\begin{proposition}\label{prop:Gamma}
  There are maps $\Gamma\colon R^{*}(M\pitchfork
  X)\to R^{*}M\pitchfork R^{*}X$ for $M\in[\Dtopop,\Set]$ and
  $X\in[\Dtopop,\ce]$ whenever these powers exist.
\end{proposition}

\begin{proof}
  First calculate the components of the powers
  \begin{align*}
    R^{*}(M\pitchfork X)_n &= (M\pitchfork X)_{n+1} \\
                           &= \{ M\x \Dtop[n+1],X\} \\
    (R^{*}M\pitchfork R^{*}X)_n &= \{ R^{*}M\x\Dtop[n], R^{*}X\} \\
    &= \{ \Lan_R(MR\x \Dtop[n]),X\} 
  \end{align*}
  and now the desired map may be constructed as the composite
  \[
    \begin{tikzcd}
      \Lan_R(MR\x \Dtop[n]) \rar &
      \Lan_R(MR)\x \Lan_R(\Dtop[n]) \rar &
      M\x \Dtop[n+1]
    \end{tikzcd}
  \]
  in which the first map is the canonical comparison for $\Lan_R$
  applied to a product, and the second comes from the counit
  $\Lan_R(MR)\to M$ and the fact that $\Lan_R$ sends a representable
  $\Dtop(-,n)$ to $\Dtop(-,Rn)$ and $Rn=n+1$.
\end{proof}

By Proposition~\ref{prop:n-simplex-inductively}, to give an $n$-simplex
in $\underline{s_T\ce}(\ts Y,\ts X)$ is to give:
\begin{enumerate}
\item an $n$-simplex in
$\underline{[\Dtopop,\ce]}(Y,X)$, or equivalently a map $u\colon
Y\to \Dtop[n]\pitchfork X$ in $[\Dtopop,\ce]$
\item an $n-1$-simplex in $\underline{s_T\ce}(\ts Y,\ts X)$, or
  equivalently a map $v\colon \ts Y\to \Delta[n-1]\pitchfork \ts X$
\item satisfying conditions (a) and (b) of the proposition.
\end{enumerate} 
We now look at how to express these conditions (a) and (b).

\begin{proposition}\label{prop:u-u-condition-a}
  Condition (a) holds if and only if the diagram
\[
  \begin{tikzcd}
    R^{*}Y \dar["\xi^{\ts Y}"] \rar["R^{*}u"] &
    R^{*}(\Dtop[n]\pitchfork X) \rar & \Dtop[n]\pitchfork R^{*}X
    \dar["\xi^{\ts X}"] \\
    T_{!}Y \rar["T_{!}u"] & T_{!}(\Dtop[n]\pitchfork X) \rar &
    \Dtop[n]\pitchfork T_{!}X
  \end{tikzcd}
\]
commutes, where the unnamed maps are these canonical comparisons. 
\end{proposition}

\begin{proof}
  Look at degree $m$ of the diagram, and evaluate at a map
  $\theta\colon m\to n$ in $\Dtop$. This gives the diagram
\[
  \begin{tikzcd}
    Y_{m+1} \dar["d_{m+1}"] \rar["u_{\overline\theta}"] & X_{m+1} \dar["d_{m+1}"] \\
    TY_m \rar["Tu_{\theta}"] & TX_m
  \end{tikzcd}
\]
where $\overline\theta\colon m+1\to n$ is the top-preserving extension
of $\theta$. Then $\overline\theta\delta_{m+1}=\theta$ and so
commutativity is exactly condition (a).  
\end{proof}

\begin{proposition}\label{prop:u-v-cond-b}
  Condition (b) holds if and only if the diagram
  \[
    \begin{tikzcd}
      R^{*}Y \dar["\xi^{\ts Y}"] \rar["R^{*}u"] &
      R^{*}(\Dtop[n]\pitchfork X) \rar["\Gamma"] & 
      R^{*}\Dtop[n]\pitchfork R^{*}X \rar["\Lambda\pitchfork1"] &
      U\Delta[n-1]\pitchfork R^{*}X
      \dar["{U\Delta[n-1]\pitchfork\xi^{\ts X}}"]
      \\
      T_{!}Y \rar["TUv"] & T_{!}U(\Delta[n-1]\pitchfork\ts X) \rar["T\Pi"]
      &
       T_{!}(U\Delta[n-1]\pitchfork X) \rar &
       U\Delta[n-1]\pitchfork T_{!}X 
    \end{tikzcd}
  \]
  commutes in which once again the unnamed map is the canonical comparison.
\end{proposition}

\begin{proof}
  Look at degree $m$ of the diagram and evaluate at a map
  $\theta\colon m\to n-1$ in $\Dtop$. This gives the map
  \[
    \begin{tikzcd}
      Y_{m+1} \dar["d_{m+1}"] \rar["u_{\Lambda\theta}"] & X_{m+1} \dar["d_{m+1}"] \\
      TY_m \rar["v_{\theta}"] & TX_m
    \end{tikzcd} \]
  and $\Lambda\theta.\delta_{m+1}=\delta_{m+1}.\theta$ and so the
  diagram commutes if and only if (b) holds.
\end{proof}

\subsection{Powers by representables}

\begin{proposition}\label{prop:powers}
  $\underline{s_T\ce}$ has powers by representable simplicial sets
  $\Delta[n]$. 
\end{proposition} 

\begin{proof}
We have already seen that $\underline{s_T\ce}$ has copowers, so in
order to construct powers it suffices to check the 1-dimensional
universal property. In other words, we should show that for each $\ts
X\in s_T\ce$ the functor $(s_T\ce)\op\to\Set$ sending $\ts Y$ to
$\underline{s_T\ce}(\ts Y,\ts X)_n$ is representable.  We shall do this by
showing that it is a limit of representable functors; since $s_T\ce$
is locally finitely presentable and so in particular complete, the
result follows.

We do this using the analysis given in the previous section of the
$n$-simplices in $\underline{s_T\ce}(\ts Y,\ts X)$. We saw that to
give such an $n$-simplex we should give a map $v\colon \ts Y\to
\Delta[n-1]\pitchfork\ts X$ and a map $u\colon Y\to \Dtop[n]\pitchfork
X$ subject to two conditions.

Now $Y=U\ts Y$, and $U$ has a right adjoint $Q$ by Theorem~\ref{thm:comondicity}, thus to give
$u$ is equivalently to give $\widetilde{u}\colon \ts Y\to
Q(\Dtop[n]\pitchfork X)$.

Now the restriction map $R^{*}\colon [\Dtopop,\ce]\to[\Dtopop,\ce]$
has a right adjoint $R_{*}$ given by right Kan extension, and so we
can reformulate the condition in
Proposition~\ref{prop:u-u-condition-a} as saying that two composites
\[
  \begin{tikzcd}
    \ts Y \rar["\widetilde{u}"] & Q(\Dtop[n]\pitchfork U\ts X)
    \ar[r,shift left=1] \ar[r,shift right=1] &
    QR_{*}(\Dtop[n]\pitchfork T_{!}U\ts X)
  \end{tikzcd}
\]
are equal; here the unnmaed maps are natural in $\ts X$.

Similarly, we can reformulate the condition in
Proposition~\ref{prop:u-v-cond-b} as commutativity of a square
\[
  \begin{tikzcd}
    \ts Y \rar["\widetilde{u}"] \dar["v"] & Q(\Dtop[n]\pitchfork U\ts
    X) \dar \\
    \Delta[n-1]\pitchfork\ts X \rar & QR_{*}(U\Delta[n-1]\pitchfork
    T_{!}U\ts X)
  \end{tikzcd}
\]
where once again the unnamed maps are natural in $\ts X$.

It will follow that we can construct the power $\Delta[n]\pitchfork \ts X$
as a limit
\[
  \begin{tikzcd}
    \Delta[n]\pitchfork \ts X \rar \dar & Q(\Dtop[n]\pitchfork U\ts X)
    \dar \rar[shift left=1] \rar[shift right=1] &
    QR_{*}(\Dtop[n]\pitchfork T_{!}U\ts X) \\
    \Delta[n-1]\pitchfork\ts X \rar & QR_*(U\Delta[n-1]\pitchfork
    T_{!}U\ts X).
  \end{tikzcd}
\]
\end{proof}

\subsection{Enriched local presentability}

\begin{theorem}\label{thm:s-Presentable}
  If $\ce$ is locally finitely presentable and $T$ is finitary, then
  $s_T\ce$ is locally finitely presentable as a simplicially enriched
  category. 
\end{theorem}

\begin{proof}
 We saw in Theorem~\ref{thm:set-lfp} that the ordinary category $s_T\ce$ is locally finitely presentable. 
 We saw in Proposition~\ref{prop:copower-in-T-Simp} that the enriched category $\underline{s_T\ce}$ has copowers. 
 We saw in Proposition~\ref{prop:powers} that $\underline{s_T\ce}$ has powers by simplicial sets of the form $\Delta[n]$; but general simplicial sets are colimits of these, and so general powers can be constructed using powers by the $\Delta[n]$ and conical limits. 
 Thus $\underline{s_T\ce}$ has powers, and so is complete and cocomplete. 
 It will therefore be locally finitely presentable provided that the functors $\Delta[n]\pitchfork-\colon s_T\ce\to s_T\ce$ are finitary: see for example \cite[Proposition~2.4]{BourkeLack-AccInftyCosmoi}.
 This follows by induction on $n$ once again. On the one hand $\Delta[0]\pitchfork-$ is (isomorphic to) the identity, and so preserves all colimits. On the other hand,
  as in the proof of Theorem~\ref{prop:powers},
  $\Delta[n]\pitchfork-$ can be constructed as a finite limit of
  $\Delta[n-1]\pitchfork-$ and various other functors, so provided
  that these other functors are finitary, the result will follow since
  finite limits commute with filtered colimits. But these other
  functors are themselves all constructed from the finitary functor
  $T$ using various finite limits.
\end{proof}

\begin{theorem}\label{thm:Cat-presentable}
  If $\ce$ is locally finitely presentable and $T$ is finitary, then
  $\Cat_T(\ce)$ is reflective in $s_T\ce$ and the inclusion is finitary.
  It follows that $\Cat_T(\ce)$ is locally finitely presentable as a $2$-category.
\end{theorem}

\begin{proof}
  Since $\Cat_T(\ce)$ is closed in $s_T\ce$ under limits and filtered
  colimits, it is reflective with finitary inclusion, and so locally
  finitely presentable as a simplicially-enriched category. So it is finitarily
  reflective in $[\cg\op,s\Set]$ where $\cg=(\Cat_T(\ce))_f$, and so
  also finitarily reflective in $[\cg\op,\Cat]$. Thus it is locally
  finitely presentable as a 2-category.
\end{proof}

\appendix

\section{Some results on locally finitely presentable categories}\label{apx:lfp}

The results in this appendix are surely known, but we could not find a
suitable reference so have treated them here. They are used in the
proofs of Theorem~\ref{thm:lfp} and Theorem~\ref{thm:set-lfp}.
For a locally finitely presentable category $\ca$, let $\ca_f$ be its full subcategory consisting of all finitely presentable objects.

\begin{proposition}\label{prop:comma}
  Let $\ca$ and $\cb$ be locally finitely presentable categories, and
  $F\colon\ca\to\cb$ a finitary functor. Then the comma category
  $\cb/F$ is locally finitely presentable, the projections
  $U\colon\cb/F\to\ca$ and $V\colon\cb/F\to\cb$ are left adjoints (and so in particular finitary), and $U$ is also a right adjoint.
\end{proposition}

\begin{proof}
  First observe that $\cb/F$ is cocomplete and the projections
  preserve colimits: the colimit of a diagram in $\cb/F$ involving
  maps $b_i\colon B_i\to FA_i$ is the map $\colim_iB_i\to
  F(\colim_iA_i)$. Given this description of colimits, it is clear
  that an object $H\to FG$ in $\cb/F$ is finitely presentable if
  $H$ is finitely presentable in $\cb$ and $G$ is finitely presentable
  in $\ca$.

Since $U$ and $V$ are cocontinuous, they will be left adjoints if $\cb/F$ is locally finitely presentable. In fact it is easy to construct the right adjoint to $U$ directly: it sends $A\in\ca$ to $1\colon FA\to FA$. The left adjoint to $U$ sends $A$ to $0\to FA$.

  Thus $\cb/F$ will be locally finitely presentable provided that the
  objects of the form $H\to FG$ as above constitute a strong
  generator (in the sense of \cite[Section~3.6]{Kelly-book}), by \cite[Theorem~1.11 and Remark below Definition~1.9]{Adamek-Rosicky-book}.\footnote{Whereas the definition of strong generator in \cite[0.6]{Adamek-Rosicky-book} differs from that of \cite[Section~3.6]{Kelly-book}, \cite[Proof of Theorem~1.11]{Adamek-Rosicky-book} works with respect to the latter definition without essential changes.} Suppose then that
  \[
    \begin{tikzcd}
      B \dar["b"'] \rar["g"] & B' \dar["b'"] \\
      FA \rar["Ff"] & FA'
    \end{tikzcd}
  \]
  is a morphism in $\cb/F$, which is inverted by
  $\cb/F((h,G),-)\colon\cb/F\to\Set$ for all $h\colon H\to FG$ with
  $G\in\ca_f$ and $H\in\cb_f$. Since the initial object $0$ of $\cb$
  is finitely presentable, this includes objects of the form $0\to
  FG$, and it follows that $\ca(G,f)$ is invertible for all
  $G\in\ca_f$, and so that $f$ is invertible.

  To see that $g$ is invertible, let $y\colon H\to B'$ be given, with $H\in \cb_f$. Write
  $A=\colim_iG_i$ as a filtered colimit of finitely presentable
  objects in $\ca$. Since $H$ is finitely presentable, $F$ is finitary, and $f$ is
  invertible, the composite $H\xrightarrow{y} B'\xrightarrow{b'} FA'$ factorizes through
  $G_i\to A\xrightarrow{f} A'$ for some $i$, and now there is a unique $x$ 
  as in
  \[
    \begin{tikzcd}
      H \ar[drr,"y"] \ar[dr,dashed,"x"'] \ar[d] \\
      FG_i \ar[dr] & B \ar[d,"b"] \rar["g" '] & B' \dar["b' "] \\
      & FA \rar["Ff"] & FA'
    \end{tikzcd}
  \]
  and so $\cb(H,g)$ is invertible and so $g$ is invertible. This
  proves that $\cb/F$ is locally finitely presentable. 
\end{proof}

\begin{proposition}\label{prop:ins}
  Let $\ca$ and $\cb$ be locally finitely presentable categories, let
  $F,G\colon\ca\to\cb$ be functors, with $F$ finitary and $G$ having a left
  adjoint $L$. Then the inserter $\mathrm{Ins}(F,G)$ is locally
  finitely presentable and the projection
  $P\colon\mathrm{Ins}(F,G)\to\ca$ is a finitary right adjoint.
\end{proposition}

\begin{proof}
  As observed in \cite[Proposition~2.14]{BirdPhD}, the inserter
  $\mathrm{Ins}(F,G)$ can equivalently be described as the category of
  algebras for the (finitary) endofunctor $LF$ of the locally finitely
  presentable category $\ca$. A finitary endofunctor on a locally finitely presentable category has a (finitary) free monad, and the category of algebras for the monad is the category of algebras for the endofunctor. Since the category of algebras for a finitary monad on a locally finitely presentable category is again locally finitely presentable, and the forgetful functor is finitary, the result follows. 
\end{proof}

\begin{proposition}\label{prop:eq}
  Let $\ca$ and $\cb$ be locally finitely presentable categories, let
  $F,G\colon\ca\to\cb$ be functors, with $F$ finitary and $G$ having a left
  adjoint $L$, and let $\alpha,\beta\colon F\to G$ be a pair of
  natural transformations. Then the equifier $\mathrm{Eq}(\alpha,\beta)$ is locally
  finitely presentable and the projection
  $P\colon\mathrm{Eq}(\alpha,\beta)\to\ca$ is a finitary right adjoint.
\end{proposition}

\begin{proof}
  This time we largely follow the argument of
  \cite[Proposition~2.16]{BirdPhD} (which as stated would need $F$ to
  have a left adjont as well). First observe that the equifier of
  $\alpha$ and $\beta$ is equally the equifier of the induced maps
  $\alpha',\beta'\colon LF\to 1$, and that $LF$ is once again
  finitary. Now let $\eta \colon 1\to S$ be the (pointwise) coequalizer of
  $\alpha'$ and $\beta'$. 
  The resulting $S$ is a colimit of finitary
  functors and so is still finitary. The equifier of $\alpha'$ and
  $\beta'$ is likewise the inverter of $\eta$.

  Now $\eta S.\eta=S\eta.\eta$ by naturality, but $\eta$ is an
  epimorphism since it is a coequalizer, so $\eta S=S\eta$ and the
  pointed endofunctor $(S,\eta)$ is {\em well-pointed} in the sense of
  \cite[Section~5]{Kelly-transfinite}. Thus by \cite[Proposition~5.2]{Kelly-transfinite} an algebra for the pointed
  endofunctor is the same as an object $A\in\ca$ with $\eta A$
  invertible; in other words, an object of the inverter. Thus the desired category $\mathrm{Eq}(\alpha,\beta)$ is the
  category of algebras for the finitary well-pointed endofunctor
  $(S,\eta)$. 
  The existence of a left adjoint of $P\colon \mathrm{Eq}(\alpha,\beta)\to\ca$ follows from \cite[Theorem~6.2]{Kelly-transfinite} (with $\mathcal E=\mathcal E'=\{\text{isomorphisms in $\ca$}\}$ and $\alpha=\omega$); $\mathrm{Eq}(\alpha,\beta)$ (which is a full subcategory of $\ca$) is closed under filtered colimits in $\ca$ because $S$ is finitary. 
\end{proof}

\bibliography{my}
\bibliographystyle{plain}

\end{document}